\documentclass[draft]{amsart}

\usepackage{amsmath,amssymb,amsthm,enumerate} %% default
\usepackage{amsfonts} %% for mathbb
\usepackage{color}
\usepackage{dsfont} %% for the indicator function

\theoremstyle{plain}
\newtheorem{pretheo}{Theorem}[section]
\newtheorem{preassu}[pretheo]{Assumption}
\newtheorem{precoro}[pretheo]{Corollary}
\newtheorem{predefi}[pretheo]{Definition}
\newtheorem{preexam}[pretheo]{Example}
\newtheorem{prelemm}[pretheo]{Lemma}
\newtheorem{preprop}[pretheo]{Proposition}
\newtheorem{prerema}[pretheo]{Remark}

\newenvironment{theo}{\begin{pretheo}}{\end{pretheo}}
\newenvironment{assu}{\begin{preassu}}{\end{preassu}}

\newenvironment{lemm}{\begin{prelemm}}{\end{prelemm}}

\newenvironment{rema}{\begin{prerema}\rm}{\end{prerema}}

\DeclareMathOperator{\dv}{div}

\DeclareMathOperator{\spp}{supp}

\newcommand{\intd}{\,d}
\newcommand{\jmp}[1]{\ensuremath{[\![#1]\!]}}
\newcommand{\lc}{{\rm loc}}
\newcommand{\pd}{\partial}
\newcommand{\wht}[1]{\widehat{#1}}
\newcommand{\wtd}[1]{\widetilde{#1}}

\newcommand{\Bf}{\mathbf{f}}
\newcommand{\Bg}{\mathbf{g}}

\newcommand{\Bn}{\mathbf{n}}

\newcommand{\Bu}{\mathbf{u}}
\newcommand{\Bv}{\mathbf{v}}

\newcommand{\BB}{\mathbf{B}}
\newcommand{\BC}{\mathbf{C}}

\newcommand{\BF}{\mathbf{F}}

\newcommand{\BI}{\mathbf{I}}

\newcommand{\BM}{\mathbf{M}}
\newcommand{\BN}{\mathbf{N}}

\newcommand{\BR}{\mathbf{R}}

\newcommand{\CA}{\mathcal{A}}

\newcommand{\CF}{\mathcal{F}}
\newcommand{\CG}{\mathcal{G}}

\newcommand{\CL}{\mathcal{L}}

\newcommand{\CR}{\mathcal{R}}
\newcommand{\CS}{\mathcal{S}}
\newcommand{\CT}{\mathcal{T}}

\newcommand{\Fp}{\mathfrak{p}}

\newcommand{\Fs}{\mathfrak{s}}

\newcommand{\SST}{\mathsf{T}}

\newcommand{\vps}{\varepsilon}

\newcommand{\vph}{\varphi}

\numberwithin{equation}{section} 
\allowdisplaybreaks[4]

\begin{document}
\title[Elliptic problems associated with two-phase incompressible flows]
{Unique solvability of elliptic problems associated with two-phase incompressible flows in unbounded domains}

\author[Hirokazu Saito]{Hirokazu Saito}
\address{Faculty of Industrial Science and Technology,
Tokyo University of Science,
102-1 Tomino, Oshamambe-cho, Yamakoshi-gun,
Hokkaido 049-3514, Japan}
\email{hsaito@rs.tus.ac.jp}

\author[Xin Zhang]{Xin Zhang}
\address{Faculty of Science and Engineering,
Waseda University, Okubo 3-4-1, Shinjuku-ku,
Tokyo 169-8555, Japan}
\email{xinzhang@aoni.waseda.jp}



\thanks{
The work of the first author was supported by JSPS KAKENHI Grant Number JP17K14224,
and the second author was supported by the Top Global University Project.
}

%\date{\today}

%\dedicatory{}

\begin{abstract}
This paper shows the unique solvability of elliptic problems 
associated with two-phase incompressible flows,
which are governed by the two-phase Navier-Stokes equations,
in unbounded domains such as the whole space separated by a compact interface
and the whole space separated by a non-compact interface.
As a by-product of the unique solvability of elliptic problems,
we obtain the Helmholtz-Weyl decomposition for two-phase incompressible flows.
\end{abstract}

\maketitle

%\tableofcontents

%%%%%%%%%%%%%%%%%%%%%%%%%%%%%%%%%%%%%%%%%%%%%%%%%%%%%%%%%%%%%%%%%%%%%%
\section{Introduction and main results}
\subsection{Introduction}
Let $\Omega_+$ be a bounded domain in the $N$-dimensional Euclidean space  $\BR^N$,
$N\geq 2$, with boundary $\Sigma$,
and let $\Omega_-=\BR^N\setminus(\Omega_+\cup\Sigma)$.
Let us define
$\rho=\rho_+\mathds{1}_{\Omega_+}+\rho_-\mathds{1}_{\Omega_-}$
for positive constants $\rho_\pm$,
where $\mathds{1}_A$ is the indicator function of $A\subset\BR^N$.
We set for an open set $G$ of $\BR^N$ and for $q\in(1,\infty)$
\begin{equation*}
\wht H_q^1(G)=\{f\in L_{q,\lc}(\overline{G}) : \nabla f\in L_q(G)^N\}, 
\end{equation*}
and define a solenoidal space $J_q(\BR^N\setminus\Sigma)$
and a space $G_q(\BR^N\setminus\Sigma)$ as follows:
\begin{align*}
J_q(\BR^N\setminus\Sigma)
&=\{\Bu \in L_q(\BR^N\setminus\Sigma)^N : (\Bu,\nabla\vph)_{\BR^N\setminus\Sigma}=0
\text{ for any }\vph\in\wht H_{q'}^1(\BR^N)\}, \\
G_q(\BR^N\setminus\Sigma)
&=\{\Bv\in L_q(\BR^N\setminus\Sigma)^N : \Bv=\rho^{-1}\nabla\psi \text{ for some } \psi\in\wht H_q^1(\BR^N)\},
\end{align*}
where $q'=q/(q-1)$ and
$$
(\Bu,\nabla\vph)_{\BR^N\setminus\Sigma}=\int_{\BR^N\setminus\Sigma} \Bu(x)\cdot\nabla\vph(x)\intd x.
$$
Here the central dot denotes the scalar product of $\BR^N$.

Let $\Bf\in L_q(\BR^N\setminus\Sigma)^N$.
One of the purposes of this paper is to show the unique solvability of the following weak elliptic problem:
Find $u\in \wht H_q^1(\BR^N)$ such that
\begin{equation}\label{weakeq:exterior}
(\rho^{-1}\nabla u,\nabla \vph)_{\BR^N\setminus\Sigma}
=(\Bf,\nabla\vph)_{\BR^N\setminus\Sigma} \quad \text{for any $\vph\in\wht H_{q'}^1(\BR^N)$.}
\end{equation}
This weak elliptic problem arises from the study of two-phase incompressible flows
governed by the two-phase Navier-Stokes equations.
The momentum equation of the two-phase Navier-Stoke equations is linearized as 
$$
\rho\pd_t\Bu-\mu\Delta\Bu+\nabla\Fp= \Bg \quad \text{in $\BR^N\setminus\Sigma$,}
$$
where $\Bg$ is a given function and $\mu=\mu_+\mathds{1}_{\Omega_+}+\mu_-\mathds{1}_{\Omega_-}$
for positive constants $\mu_{\pm}$ describing the viscosity coefficients,
and then the unique solvability of \eqref{weakeq:exterior} enables us to eliminate 
the pressure $\Fp$ from the linearized equation. 
This elimination of pressure plays an important role in applications such as 
the generation of analytic $C_0$-semigroups, the maximal regularity, and 
the local and global solvability of the two-phase Navier-Stokes equations
 (cf. \cite{PruSim2016}, \cite{MarSai2017}, and \cite{SSZ2018}).
Another very important application of the unique solvability of \eqref{weakeq:exterior} is
a two-phase version of the Helmholtz-Weyl decomposition as follows:
\begin{equation}\label{eq:1-decomp}
L_q(\BR^N\setminus\Sigma)^N=J_q(\BR^N\setminus\Sigma)\oplus G_q(\BR^N\setminus\Sigma),
\end{equation}
where $\oplus$ denotes the direct sum.
Note that this decomposition is equivalent to the unique solvablility of \eqref{weakeq:exterior}.

Pr$\ddot{\rm u}$ss and Simonett \cite[Proposition 8.6.2]{PruSim2016} proved
the unique solvability of a weak elliptic problem associated with two-phase incompressible flows
in the case where $\Omega_\pm$ are both bounded domains,
while the case of unbounded domains is not very well known to the best of our knowledge.
This motivates us to study the unique solvability of \eqref{weakeq:exterior}.
Furthermore, as examples of unbounded domains with non-compact interface $\Sigma$,
we also treat the whole space with a flat interface 
and the whole space with a bent interface in the present paper.
The former is in Subsection \ref{subsec:whole1-weak} below,
while the latter is in Subsection \ref{subsec:whole2-weak} below.

At this point, we introduce a short history of one-phase case
for the unique solvability of weak elliptic problems and the Helmholtz-Weyl decomposition.

We first introduce the classical weak Neumann problem: Find $u\in\wht H_q^1(D)$ such that
\begin{equation}\label{eq:NP}
(\nabla u,\nabla\vph)_D =(\Bf,\nabla\vph)_D
\quad\text{for any $\vph \in \wht{H}_{q'}^1(D)$,}
\end{equation}
where $\Bf\in L_q(D)^N$ and $D$ is a domain in $\BR^N$.
It is well known that the unique solvability of \eqref{eq:NP} is equivalent to 
the following Helmholtz-Weyl decomposition:
\begin{equation}\label{eq:HW}
L_q(D)^N = L_{q,\sigma} (D) \oplus G_q(D),
\end{equation}
where $L_{q,\sigma} (D)$ and $G_q(D)$ are given by
\begin{align*}
&L_{q,\sigma} (D)
=\overline{C_{0,\sigma}^\infty(D)}^{\|\cdot\|_{L_q(D)}}, \quad
C_{0,\sigma}^\infty(D)=\{\Bu\in C_0^\infty(D)^N : \dv\Bu=0 \text{ in $D$}\}, \\
&G_q(D)
=\{\Bv\in L_q(D)^N : \Bv=\nabla\psi \text{ for some } \psi\in\wht H_q^1(D)\}.
\end{align*}

The investigation of \eqref{eq:HW} (or \eqref{eq:NP}) can be traced back to Weyl \cite{Weyl1940}.   
Although \eqref{eq:HW} holds for any domain $D$ in $\BR^N$ when $q=2$ (cf. e.g. \cite{Sohr2001}), 
the general $L_q$-framework is more involved. 
According to \cite{Solonnikov1977,FM1977,Miyakawa1982,SS1992,FS1994,Miya1994,FS1996,AS2003,GS2010},
we can conclude that \eqref{eq:HW} is valid for any $q\in(1,\infty)$
whenever $D$ is $\BR^N$ itself, the half space, 
a bounded or an exterior domain in $\BR^N$ with smooth boundary, 
a perturbed half space, a flat layer, an aperture domain, or a bounded convex domain.
One also knows in \cite{FMM1998} that 
\eqref{eq:HW} holds only when $3/2 -\vps < q  < 3+\vps$
for some $\vps = \vps (D)>0$,
assuming $D$ is a bounded Lipschitz domain in $\BR^N$.
Note that \eqref{eq:HW} may fail for some unbounded domain and $q\in(1,\infty)$, 
which is pointed out in \cite{MB1986}.
One however has a chance to obtain \eqref{eq:HW} for any $q\in(1,\infty)$
and for general unbounded domains, called uniform $C^1$ domains,
by introducing mixed $L_q$-spaces due to \cite{FKS2005,FKS2007}.

The classical \eqref{eq:HW} is widely used for one-phase problems with non-slip boundary condition.
On the other hand, 
to handle one-phase incompressible flows with a free surface,
we make use of the weak Dirichlet problem as follows:
Let $\Gamma$ be a connected component of the boundary of $D$ and 
\begin{equation*}
\wht H_{q,\Gamma}^1(D)=\{u\in \wht H_q^1(D): u=0 \text{ on $\Gamma$}\} \quad (1<q<\infty).
\end{equation*} 
One then says that the weak Dirichlet problem is uniquely solvable in $\wht H^1_{q,\Gamma}(D)$ 
if and only if for any $\Bf\in L_q(D)^N$ there is a unique solution $u\in\wht H_{q,\Gamma}^1(D)$ to
\begin{equation*}\label{eq:weak_op}
(\nabla u,\nabla\vph)_D=(\Bf,\nabla\vph)_D
\quad \text{for any $\vph \in \wht{H}_{q',\Gamma}^1(D)$}
\end{equation*}
and there holds the estimate: $\|\nabla u\|_{L_q(D)} \leq C \|\Bf\|_{L_q(D)}$
for some positive constant $C$ independent of $u$, $\Bf$, and $\vph$. 
In \cite{SS1996, Shi2013, PruSim2016},
it is proved that the weak Dirichlet problem is uniquely solvable in $\wht H^1_{q,\Gamma}(D)$ 
when $D$ is $\BR^N,$ the half space, a bounded or an exterior domain with smooth boundary
(cf. also \cite{Abels2006, Sai2018}).

Finally, we introduce the strong elliptic problem associated with \eqref{weakeq:exterior}.
For an open set $G$ of $\BR^N$ and $q\in(1,\infty)$, one sets
\begin{equation*}
\wht H_q^2(G)=\{f\in L_{1,\lc}(G) : \nabla f\in \wht H_q^1(G)^N\}.
\end{equation*}
The strong elliptic problem is then stated as follows:
Find $v_\pm\in\wht H_q^1(\Omega_\pm)\cap \wht H_q^2(\Omega_\pm)$ such that
\begin{equation}\label{exteq:1}
\left\{\begin{aligned}
\Delta v_\pm &= \dv\Bf_\pm && \text{in $\Omega_\pm$,} \\
\rho_+ v_+&=\rho_-v_- && \text{on $\Sigma$,} \\
\Bn\cdot \nabla ( v_+- v_-)
&=\Bn\cdot (\Bf_+-\Bf_-) && \text{on $\Sigma$,}
\end{aligned}\right.
\end{equation}
where $\Bn$ is a unit normal vector on $\Sigma$ pointing from $\Omega_+$ into $\Omega_-$.
Throughout this paper, $\Bn$ is seen as an $N$-vector of function defined on $\BR^N$
(cf. \cite[Corollary A.3]{SchadeShibata2015} and Assumption \ref{assu:1} below).
In this paper, we first prove the unique solvability of \eqref{exteq:1},
and then we prove the unique solvability of \eqref{weakeq:exterior} by using the result of the strong elliptic problem.
This approach is also applied to the problems with non-compact interfaces 
in Sections \ref{sec:whole1} and \ref{sec:whole2} below.

\medskip
\noindent
{\bf Notation.}
Let $G$ be an open set in $\BR^N$, 
and let $u=u(x)$ and $\Bv=\Bv(x)=(v_1(x),\dots,v_N(x))^\SST\footnote{$\BM^\SST$ denotes the transpose of $\BM$.} $
be respectively a scalar-valued function on $G$ and a vector-valued function on $G$,
where $x=(x_1,\dots,x_N)$.
Then, for $\pd_j=\pd/\pd x_j$,
\begin{equation*}
\nabla u = \left(\pd_1 u,\dots,\pd_N u\right)^\SST, \quad
\nabla \Bv =\{\pd_j v_k : j,k=1,\dots,N\}, \quad 
\end{equation*}
and also $\dv\Bv = \sum_{j=1}^{N}\pd_j v_k$.
Furthermore, for $\Bu=\Bu(x)=(u_1(x),\dots,u_N(x))^\SST$ and $v=v(x)$ defined on $G$,
\begin{equation*}
(u,v)_G=\int_G u(x) v(x) \intd x, \quad
(\Bu,\Bv)_G=\int_G \Bu(x)\cdot\Bv(x)\intd x
=\sum_{j=1}^N\int_G u_j(x) v_j(x)\intd x.
\end{equation*}

Let $X$ be a Banach space.
Then $X^m$, $m\geq 2$, denotes the $m$-product space of $X$,
while the norm of $X^m$ is usually denoted by $\|\cdot\|_X$ instead of $\|\cdot\|_{X^m}$
for the sake of simplicity.
For another Banach space $Y$, $\CL(X,Y)$ stands for the Banach space of all bounded linear operators
from $X$ to $Y$. In addition, $\CL(X)=\CL(X,X)$.

Let $p\geq 1$ or $p=\infty$, and let $q\in(1,\infty)$.
The Lebesgue spaces on $G$ are denoted by $L_p(G)$
with norm $\|\cdot\|_{L_p(G)}$, 
while the Sobolev spaces on $G$ are denoted by $H_p^n(G)$, $n\in\BN$, with norm $\|\cdot\|_{H_p^n(G)}$.
Set $H_p^0(G)=L_p(G)$ and $\BN_0=\BN\cup\{0\}$.
For any multi-index $\alpha=(\alpha_1,\dots,\alpha_N)\in\BN_0^N$,
\begin{equation*}
\pd^\alpha u =\pd_x^\alpha u= \frac{\pd^{|\alpha| }u(x) }{\pd x_1^{\alpha_1}\dots \pd x_N^{\alpha_N}}
\quad \text{ with $|\alpha|=\alpha_1+\dots+\alpha_N$.}
\end{equation*}

Let $s\in(0,\infty)\setminus\BN$, and then
$[s]$ stands for the largest integer less than $s$,
while $W_q^s(G)=\{f\in L_q(G): \|f\|_{W_q^s(G)}<\infty\}$ with
\begin{align*}
\|f\|_{W_q^s(G)}=\|f\|_{H_q^{[s]}(G)}+\sum_{|\alpha|=[s]}
\left(\int_G\int_G\frac{|\pd^\alpha f(x)-\pd^\alpha f(y)|^q}{|x-y|^{N+(s-[s])q}}\intd x dy \right)^{1/q}.
\end{align*}
In addition, 
$C_0^\infty(G)$ stands for
the set of all functions in $C^\infty(G)$ whose supports are compact subsets of $G$,
while for a domain $D$ of $\BR^N$
\begin{equation*}
E_q(D)=\{\Bf \in L_q(D)^N : \dv\Bf \in L_q(D)\}
\end{equation*}
endowed with $\|\Bf\|_{E_q(D)}=\|\Bf\|_{L_q(D)}+\|\dv\Bf\|_{L_q(D)}$.

%%%%%%%%%%%%%%%%%%%%%%%%%%%%%%%%%%%%%%%%%%%%%%%%%%
\subsection{Main results}
We first introduce assumptions for $\Omega_\pm$.

\begin{assu}\label{assu:1}
\begin{enumerate}[{\rm (a)}]
\item
$r$ is a real number satisfying $r>N$.
\item
$\Omega_+$ is a bounded domain of $\BR^N$ with boundary $\Sigma$ of class $W_r^{2-1/r}$.
\item
$\Omega_-=\BR^N\setminus(\Omega_+\cup \Sigma)$.
\end{enumerate}
\end{assu}

Now we state our main result for the strong elliptic problem \eqref{exteq:1}.

\begin{theo}\label{theo:main}
Suppose that Assumption $\ref{assu:1}$ holds and $\rho_\pm$ are positive constants.
Let $q\in(1,\infty)$ and $q'=q/(q-1)$ with $\max(q,q')\leq r$.
\begin{enumerate}[$(1)$]
\item\label{theo:main-1}
{\bf Existence}.
Let $\Bf_\pm\in E_q(\Omega_\pm)$ with $\Bn\cdot\Bf_\pm\in H_q^1(\Omega_\pm)$.
Then the strong elliptic problem \eqref{exteq:1} admits solutions 
$v_\pm \in \wht H_q^1(\Omega_\pm)\cap \wht H_q^2(\Omega_\pm)$ satisfying
\begin{align}\label{theo:main-est1}
\|\nabla^2 v_\pm\|_{L_q(\Omega_\pm)}
&\leq C\sum_{\Fs\in\{+,-\}}\left(\|\Bf_\Fs\|_{E_q(\Omega_\Fs)}+\|\Bn\cdot\Bf_\Fs\|_{H_q^1(\Omega_\Fs)}\right), \\
\|\nabla v_\pm\|_{L_q(\Omega_\pm)}
&\leq C\sum_{\Fs\in\{+,-\}}\|\Bf_\Fs\|_{L_q(\Omega_\Fs)},
\end{align}
with some positive constant $C=C(N,q,r,\rho_+,\rho_-)$.
Additionally, if $\Bn\cdot(\Bf_+-\Bf_-)=0$ on $\Sigma$, then $v_\pm$ satisfy
\begin{equation}\label{theo:main-est2}
\|\nabla^2 v_\pm\|_{L_q(\Omega_\pm)}
\leq C\sum_{\Fs\in\{+,-\}}\|\Bf_\Fs\|_{E_q(\Omega_\Fs)}
\end{equation}
for some positive constant $C=C(N,q,r,\rho_+,\rho_-)$.
\item\label{theo:main-2}
{\bf Uniqueness}.
If $v_\pm\in \wht H_q^1(\Omega_\pm) \cap \wht H_q^2(\Omega_\pm)$ satisfies
\begin{equation*}
\Delta v_\pm=0\text{ in $\Omega_\pm$,} \quad
\rho_+v_+=\rho_-v_- \text{\ and \ } \Bn\cdot\nabla (v_+-v_-)=0 \text{ on $\Sigma$,}
\end{equation*}
then $v_\pm =\rho_\pm^{-1}c$ for some constant $c$.
\end{enumerate}
\end{theo}

For the weak elliptic problem \eqref{weakeq:exterior},
our main result reads as

\begin{theo}\label{theo:main2}
Suppose that Assumption $\ref{assu:1}$ holds
and $\rho=\rho_+\mathds{1}_{\Omega_+}+\rho_-\mathds{1}_{\Omega_-}$
for positive constants $\rho_\pm$.
Let $q\in(1,\infty)$ and $q'=q/(q-1)$ with $\max(q,q')\leq r$.
\begin{enumerate}[$(1)$]
\item
{\bf Existence}.
Let $\Bf\in L_q(\BR^N\setminus\Sigma)^N$.
Then the weak elliptic problem \eqref{weakeq:exterior} admits a solution $u\in\wht H_q^1(\BR^N)$ satisfying
the estimate:
$
\|\nabla u\|_{L_q(\BR^N)}\leq C\|\Bf\|_{L_q(\BR^N\setminus\Sigma)}
$
for some positive constant $C=C(N,q,r,\rho_+,\rho_-)$.
\item
{\bf Uniqueness}.
If $u\in\wht H_q^1(\BR^N)$ satisfies
$$
(\rho^{-1}\nabla u,\nabla \vph)_{\BR^N\setminus\Sigma} =0 \quad \text{for any $\vph\in \wht H_{q'}^1(\BR^N)$,}
$$
then $u=c$ for some constant $c$.
\end{enumerate}
\end{theo}

Furthermore, we have by Theorem \ref{theo:main2} the two-phase version of 
the Helmholtz-Weyl decomposition as follows:

\begin{theo}\label{theo:HW-decomp}
Suppose that the same assumption as in Theorem $\ref{theo:main2}$ holds.
Then the decomposition \eqref{eq:1-decomp} holds.
\end{theo}

This paper is organized as follows:
The next section introduces some function spaces and lemmas,
which are used in Section \ref{sec:whole1}.
Section \ref{sec:whole1} treats strong elliptic problems with and without resolvent parameter $\lambda$
and a weak elliptic problem in the whole space with a flat interface.
Section \ref{sec:whole2} treats strong and weak elliptic problems similar to Section \ref{sec:whole1}
in the whole space with a bent interface,
and proves the unique solvability of the problems by using results obtained in Section \ref{sec:whole1}.
In Section \ref{sec:bounded},
we first introduce the unique solvability of a strong elliptic problem with resolvent parameter $\lambda$ in a bounded domain,
which is proved by the standard localization technique together with a result given in Section \ref{sec:whole2}.
Next, we prove the unique solvability of the strong elliptic problem without $\lambda$
by using the result with $\lambda$ and the Riesz-Schauder theory. 
Section \ref{sec:exterior} proves our main results as stated above, i.e. Theorems \ref{theo:main} and \ref{theo:main2},
by the main result of Section \ref{sec:bounded} with a cut-off technique.

%%%%%%%%%%%%%%%%%%%%%%%%%%%%%%%%%%%%%%%%%%%%%%%%%%%%%%%%%%%%%%%%%%%%%%
\section{Preliminaries}\label{sec:2}
Let us define
\begin{align*}
&F_q(\BR_\pm^N) 
=\{\Bf_\pm=(f_{\pm 1},\dots, f_{\pm N})^\SST \in E_q(\BR_\pm^N) : f_{\pm N}\in H_q^1(\BR_\pm^N)\}, \\
&\|\Bf_\pm\|_{F_q(\BR_\pm^N)}
=\|\Bf_\pm\|_{E_q(\BR_\pm^N)}+\|f_{\pm N}\|_{H_q^1(\BR_\pm^N)},
\end{align*}
where $\BR_\pm^N$ are half spaces given by
\begin{equation*}
\BR_\pm^N=\{x=(x',x_N) : x'=(x_1,\dots,x_{N-1})\in\BR^{N-1},\, \pm x_N>0\}.
\end{equation*}
The following lemma is proved in \cite{Shibata-book} (cf. also \cite[Theorem III.2.1]{Galdi11}).

\begin{lemm}\label{lemm:approx1}
Let $q\in (1,\infty)$. 
Then, for any $\Bf_\pm\in F_q(\BR_\pm^N)$,
there exists a sequence $\{\Bf_\pm^{(j)}\}_{j=1}^\infty\subset C_0^\infty(\BR^N)^N$ such that
$\lim_{j\to\infty}\|\Bf_\pm^{(j)}-\Bf_\pm\|_{F_q(\BR_\pm^N)}=0$.
\end{lemm}

For positive constants $\rho_\pm$, we set
\begin{equation}\label{defi:A}
A_\pm = \sqrt{\rho_\pm\lambda+|\xi'|^2}, \quad
(\xi',\lambda)\in \BR^{N-1} \times (\BC\setminus (-\infty,0]),
\end{equation}
where we have chosen a branch cut along the negative real axis and a branch of the square root so that $\Re\sqrt{z}>0$
for $z\in\BC\setminus(-\infty,0]$.
In addition, we set
\begin{align*}
\Sigma_{\sigma,\lambda_0}=\{z\in\BC : |\arg z|< \pi-\sigma,\, |\lambda|> \lambda_0\}
\quad (0<\sigma<\pi/2,\,\lambda_0\geq 0).
\end{align*}
Then we have

\begin{lemm}\label{lemm:symbol1}
Let $s\in\BR$, $\alpha'\in\BN_0^{N-1}$, and $\xi'\in\BR^{N-1}$ 
\begin{enumerate}[$(1)$]
\item
Let $\sigma\in(0,\pi/2)$ and $\lambda\in\Sigma_{\sigma,0}$.
Then 
\begin{alignat*}{2}
|\pd_{\xi'}^{\alpha'} A_\pm^s|
&\leq C(|\lambda|^{1/2}+|\xi'|)^{s-|\alpha'|} && \quad (\xi'\in\BR^{N-1}\setminus\{0\}), \\
|\pd_{\xi'}^{\alpha'}(\rho_+ A_-+\rho_- A_+)^s|
&\leq C(|\lambda|^{1/2}+|\xi'|)^{s-|\alpha'|} && \quad (\xi'\in\BR^{N-1}\setminus\{0\}),
\end{alignat*}
where $C=C(N,s,\alpha',\sigma,\rho_+,\rho_-)$ is a positive constant.
\item
There exists a positive constant $C=C(N,s,\alpha')$ such that
\begin{equation*}
|\pd_{\xi'}^{\alpha'}|\xi'|^s|\leq C|\xi'|^{s-|\alpha'|}. 
\end{equation*}
\end{enumerate}
\end{lemm}

\begin{proof}
The proof is similar to \cite[Lemma 5.2]{ShiShi12}, so that the detailed proof may be omitted.
\end{proof}

For $a=(a_1,\dots,a_{N-1},a_N)$, we set $a'=(a_1,\dots,a_{N-1})$.
Then the partial Fourier transform of $f=f(x)$, $x=(x_1,\dots,x_N)$,
and its inverse transform are respectively defined by
\begin{align}\label{PFT}
&\wht f(x_N)=\wht f(\xi',x_N)
= \int_{\BR^{N-1}} e^{-ix'\cdot\xi'}f(x',x_N) \intd x', \\
&\CF_{\xi'}^{-1}[\wht f(\xi',x_N)](x')
= \frac{1}{(2\pi)^{N-1}}\int_{\BR^{N-1}}e^{ix'\cdot\xi'}\wht f(\xi',x_N) \intd \xi'. \notag
\end{align}
The following two lemmas are proved in \cite[Lemma 5.4]{ShiShi12}.

\begin{lemm}\label{lemm:tech-1}
Let $q\in(1,\infty)$ and $\sigma\in(0,\pi/2)$.
Assume that $k(\xi',\lambda)$ and $\ell(\xi',\lambda)$ are defined on $(\BR^{N-1}\setminus\{0\})\times\Sigma_{\sigma,0}$,
which are many times  differentiable with respect to $\xi'$,
and satisfy for any multi-index $\alpha'\in\BN_0^{N-1}$ and $(\xi',\lambda)\in(\BR^{N-1}\setminus\{0\})\times\Sigma_{\sigma,0}$
\begin{equation*}
|\pd_{\xi'}^{\alpha'}k(\xi',\lambda)|\leq c_1(\alpha')(|\lambda|^{1/2}+|\xi'|)^{-|\alpha'|}, \quad 
|\pd_{\xi'}^{\alpha'}\ell(\xi',\lambda)|\leq c_1(\alpha')|\xi'|^{-|\alpha'|},
\end{equation*}
where $c_1(\alpha')$ is a positive constant independent of $\xi'$ and $\lambda$.
Furthermore, define the operators $K_j(\lambda)$ and $L_j(\lambda)$ $(j=1,2, \lambda\in\Sigma_{\sigma,0})$ by the formulas:
\begin{align*}
[K_1(\lambda) f_\pm](x)
&=\int_0^\infty \CF_{\xi'}^{-1}\left[k(\xi',\lambda)\lambda^{1/2}
e^{-(A x_N+B y_N)}\wht f_\pm(\xi',\pm y_N)\right](x') \intd y_N, \\
[L_1(\lambda) f_\pm](x)
&=\int_0^\infty \CF_{\xi'}^{-1}\left[\ell(\xi',\lambda)|\xi'|
e^{-(Ax_N+B y_N)}\wht f_\pm(\xi',\pm y_N)\right](x') \intd y_N, 
\end{align*}
where $A,B\in\{A_+,A_-\}$ and $x=(x',x_N)\in\BR_+^N;$
\begin{align*}
[K_2(\lambda) f_\pm](x)
&=\int_0^\infty \CF_{\xi'}^{-1}\left[k(\xi',\lambda)\lambda^{1/2}
e^{-(-Ax_N+B y_N)}\wht f_\pm(\xi',\pm y_N)\right](x') \intd y_N, \\
[L_2(\lambda) f_\pm](x)
&=\int_0^\infty \CF_{\xi'}^{-1}\left[\ell(\xi',\lambda)|\xi'|
e^{-(-Ax_N+B y_N)}\wht f_\pm(\xi',\pm y_N)\right](x') \intd y_N, 
\end{align*}
where $A,B\in\{A_+,A_-\}$ and $x=(x',x_N)\in\BR_-^N$.
Then, 
\begin{equation*}
K_1(\lambda),L_1(\lambda)\in \CL(L_q(\BR_\pm^N),L_q(\BR_+^N)), \quad
K_2(\lambda),L_2(\lambda)\in\CL(L_q(\BR_\pm^N),L_q(\BR_-^N)),
\end{equation*}
and also their operator norms do not exceed some positive constant $C$
depending on $N$, $q$, $\sigma$, $\rho_+$, $\rho_-$, and $\max\{c_1(\alpha'): |\alpha'|\leq N+1\}$,
but independent of $\lambda$.
\end{lemm}

\begin{lemm}\label{lemm:tech-2}
Let $q\in(1,\infty)$.
Assume that $m(\xi')$ is defined on $\BR^{N-1}\setminus\{0\}$,
which is many times differentiable with respect to $\xi'$,
and satisfies 
for any multi-index $\alpha'\in\BN_0^{N-1}$ and $\xi'\in\BR^{N-1}\setminus\{0\}$
$$
|\pd_{\xi'}^{\alpha'}m(\xi')|\leq c_2(\alpha')|\xi'|^{-|\alpha'|}, 
$$
where $c_2(\alpha')$ is a positive constant independent of $\xi'$.
Furthermore, define the operators $M_j$ $(j=1,2)$ by the formulas:
$$
[M_1 f_\pm](x)=\int_0^\infty \CF_{\xi'}^{-1}\left[m(\xi')|\xi'|e^{-|\xi'|(x_N+y_N)}\wht f_\pm(\xi',\pm y_N)\right](x')\intd y_N
$$
for $x=(x',x_N)\in\BR_+^N;$
$$
[M_2 f_\pm](x)=\int_0^\infty \CF_{\xi'}^{-1}\left[m(\xi')|\xi'|e^{-|\xi'|(-x_N+y_N)}\wht f_\pm(\xi',\pm y_N)\right](x')\intd y_N
$$
for $x=(x',x_N)\in\BR_-^N$. Then,
$$
M_1\in\CL(L_q(\BR_\pm^N),L_q(\BR_+^N)), \quad M_2\in\CL(L_q(\BR_\pm^N),L_q(\BR_-^N)),
$$ 
and also their operator norms do not exceed some positive constant $C$
depending on $N$, $q$, and $\max\{c_2(\alpha'): |\alpha'|\leq N+1\}$.
\end{lemm}

We next prove

\begin{lemm}\label{lemm:multiplier}
Let $q\in(1,\infty)$ and $\sigma\in(0,\pi/2)$. Assume that $k(\xi',\lambda)$ and $\ell(\xi',\lambda)$
are defined on $(\BR^{N-1}\times\{0\})\times\Sigma_{\sigma,0}$,
which are many times differentiable with respect to $\xi'$,
and satisfy for any multi-index $\alpha'\in\BN_0^{N-1}$
and $(\xi',\lambda)\in(\BR^{N-1}\setminus\{0\})\times\Sigma_{\sigma,0}$
\begin{equation*}
|\pd_{\xi'}^{\alpha'}k(\xi',\lambda)|\leq c_3(\alpha')(|\lambda|^{1/2}+|\xi'|)^{-|\alpha'|}, \quad
|\pd_{\xi'}^{\alpha'}\ell(\xi',\lambda)|
\leq c_3(\alpha')(|\lambda|^{1/2}+|\xi'|)^{-1-|\alpha'|}
\end{equation*}
where $c_3(\alpha')$ is a positive constant independent of $\xi'$ and $\lambda$.
Furthermore, define the operators $K_j(\lambda)$ and $L_j(\lambda)$ $(j=1,2,\lambda\in\Sigma_{\sigma,0})$ by the formulas:
\begin{align*}
[K_1(\lambda)f_\pm](x)=\CF_{\xi'}^{-1}\left[k(\xi',\lambda) e^{-Ax_N}\wht f_\pm (\xi',0)\right](x'), \\
[L_1(\lambda)f_\pm](x)=\CF_{\xi'}^{-1}\left[\ell(\xi',\lambda) e^{-Ax_N}\wht f_\pm (\xi',0)\right](x'), 
\end{align*}
where $A\in\{A_+,A_-\}$ and $x=(x',x_N)\in\BR_+^N;$
\begin{align*}
[K_2(\lambda)f_\pm](x)=\CF_{\xi'}^{-1}\left[k(\xi',\lambda) e^{Ax_N}\wht f_\pm (\xi',0)\right](x'), \\
[L_2(\lambda)f_\pm](x)=\CF_{\xi'}^{-1}\left[\ell(\xi',\lambda) e^{Ax_N}\wht f_\pm (\xi',0)\right](x'), 
\end{align*}
where $A\in\{A_+,A_-\}$ and $x=(x',x_N)\in\BR_-^N$.
Then the following assertions hold.
\begin{enumerate}[$(1)$]
\item
$K_1(\lambda)\in\CL(H_q^1(\BR_\pm^N),H_q^1(\BR_+^N))$ with
\begin{equation}\label{0815:1_2019}
\|(\lambda^{1/2} K_1(\lambda)f_\pm,\nabla K_1(\lambda) f_\pm)\|_{L_q(\BR_+^N)} 
\leq C\|(\lambda^{1/2} f_\pm,\nabla f_\pm)\|_{L_q(\BR_\pm^N)},
\end{equation}
while $K_2(\lambda)\in\CL(H_q^1(\BR_\pm^N),H_q^1(\BR_-^N))$ with
\begin{equation}\label{0815:2_2019}
\|(\lambda^{1/2} K_2(\lambda)f_\pm,\nabla K_2(\lambda) f_\pm)\|_{L_q(\BR_-^N)} 
\leq C\|(\lambda^{1/2} f_\pm,\nabla f_\pm)\|_{L_q(\BR_\pm^N)}.
\end{equation}
In addition,
$K_1(\lambda)\in\CL(H_q^2(\BR_\pm^N),H_q^2(\BR_+^N))$ with
\begin{align}
&\|(\lambda K_1(\lambda)f_\pm,\lambda^{1/2}\nabla K_1(\lambda)f_\pm,\nabla^2 K_1(\lambda) f_\pm)\|_{L_q(\BR_+^N)}  \label{0815:3_2019}\\
&\leq C\|(\lambda f_\pm,\lambda^{1/2}\nabla f_\pm,\nabla^2 f_\pm)\|_{L_q(\BR_\pm^N)}, \notag
\end{align}
while $K_2(\lambda)\in \CL(H_q^2(\BR_\pm^N),H_q^2(\BR_-^N))$ with
\begin{align}
&\|(\lambda K_2(\lambda)f_\pm,\lambda^{1/2}\nabla K_2(\lambda)f_\pm,\nabla^2 K_2(\lambda) f_\pm)\|_{L_q(\BR_-^N)} \label{0815:4_2019} \\
&\leq C\|(\lambda f_\pm,\lambda^{1/2}\nabla f_\pm,\nabla^2 f_\pm)\|_{L_q(\BR_\pm^N)}. \notag
\end{align}
\item
$L_1(\lambda)\in\CL(H_q^1(\BR_\pm^N),H_q^2(\BR_+^N))$ with
\begin{align}
&\|(\lambda L_1(\lambda)f_\pm,\lambda^{1/2}\nabla L_1(\lambda)f_\pm,\nabla^2 L_1(\lambda) f_\pm)\|_{L_q(\BR_+^N)} \label{0815:5_2019} \\
&\leq C\|(\lambda^{1/2}f_\pm,\nabla f_\pm)\|_{L_q(\BR_\pm^N)},  \notag \\
&\|(\lambda^{1/2} L_1(\lambda)f_\pm,\nabla L_1(\lambda) f_\pm)\|_{L_q(\BR_+^N)} \label{0815:6_2019} \\
&\leq C\left(\|f_\pm\|_{L_q(\BR_\pm^N)}+ |\lambda|^{-1/2}\|\nabla f_\pm\|_{L_q(\BR_\pm^N)}\right),  \notag
\end{align}
while $L_2(\lambda)\in\CL(H_q^1(\BR_\pm^N),H_q^2(\BR_-^N))$ with
\begin{align}
&\|(\lambda L_2(\lambda)f_\pm,\lambda^{1/2}\nabla L_2(\lambda)f_\pm,\nabla^2 L_2(\lambda) f_\pm)\|_{L_q(\BR_-^N)} \label{0815:7_2019} \\
&\leq C\|(\lambda^{1/2}f_\pm,\nabla f_\pm)\|_{L_q(\BR_\pm^N)},  \notag \\
&\|(\lambda^{1/2} L_2(\lambda)f_\pm,\nabla L_2(\lambda) f_\pm)\|_{L_q(\BR_-^N)} \label{0815:8_2019} \\
&\leq C\left(\|f_\pm\|_{L_q(\BR_\pm^N)}+ |\lambda|^{-1/2}\|\nabla f_\pm\|_{L_q(\BR_\pm^N)}\right). \notag
\end{align}
\end{enumerate}
Here $C$ is a positive constant depending on $N$, $q$, $\sigma$, $\rho_+$, $\rho_-$, and
$\max\{c_3(\alpha'): |\alpha'|\leq N+1\}$, but independent of $\lambda$. 
\end{lemm}

\begin{proof}
(1). First, we prove \eqref{0815:1_2019}.
Let $x\in\BR_+^N$, and then note that
\begin{align}\label{vole:1}
e^{-Ax_N}\wht f_\pm(\xi',0)
&=-\int_0^\infty \frac{\pd}{\pd y_N}\left(e^{-A(x_N+y_N)}\wht f_\pm(\xi',\pm y_N)\right)\intd y_N \\
&=\int_0^\infty A e^{-A(x_N+y_N)}\wht{f_\pm}(\xi',\pm y_N)\intd y_N \notag \\
&\mp\int_0^\infty e^{-A(x_N+y_N)}\wht{\pd_N f_\pm}(\xi',\pm y_N)\intd y_N. \notag
\end{align}
By these relations, we have
\begin{align*}
[K_1(\lambda)f_\pm](x)
&=\int_0^\infty \CF_{\xi'}^{-1}\left[k(\xi',\lambda)A e^{-A(x_N+y_N)}\wht{f_\pm}(\xi',\pm y_N)\right](x')\intd y_N \\
&\mp\int_0^\infty \CF_{\xi'}^{-1}\left[k(\xi',\lambda) e^{-A(x_N+y_N)}\wht{\pd_N f_\pm}(\xi',\pm y_N)\right](x')\intd y_N \\
&=:I_\pm \mp J_\pm.
\end{align*}
Let $R(A)=\rho_+$ when $A=A_+$ and $R(A)=\rho_-$ when $A=A_-$,
which gives us the following formulas:
\begin{equation}\label{eq1:20191113}
A=\frac{A^2}{A}=\frac{1}{A}\left(R(A)\lambda -\sum_{j=1}^{N-1}(i\xi_j)^2\right), \quad
1=\frac{A^2}{A^2}=\frac{1}{A^2}\left(R(A)\lambda -\sum_{j=1}^{N-1}(i\xi_j)^2\right).
\end{equation}
By using these formulas, we can write $I_\pm$ and $J_\pm$ as 
\begin{align*}
I_\pm&=\int_0^\infty \CF_{\xi'}^{-1}\left[\frac{k R(A)}{A} 
\lambda^{1/2} e^{-A(x_N+y_N)}\wht{\lambda^{1/2} f_\pm}(\xi',\pm y_N)\right](x')\intd y_N \\
&-\sum_{j=1}^{N-1}\int_0^\infty \CF_{\xi'}^{-1}
\left[\frac{k i\xi_j}{A|\xi'|} |\xi'|e^{-A(x_N+y_N)}\wht{\pd_j f_\pm}(\xi',\pm y_N)\right](x')\intd y_N, \\
J_\pm
&=
\int_0^\infty \CF_{\xi'}^{-1}\left[\frac{k R(A)\lambda^{1/2}}{A^2} 
\lambda^{1/2} e^{-A(x_N+y_N)}\wht{\pd_N f_\pm}(\xi',\pm y_N)\right](x')\intd y_N \\
&-\sum_{j=1}^{N-1}\int_0^\infty \CF_{\xi'}^{-1}
\left[\frac{k (i\xi_j)^2}{A^2 |\xi'|} |\xi'|e^{-A(x_N+y_N)}\wht{\pd_N f_\pm}(\xi',\pm y_N)\right](x')\intd y_N,
\end{align*}
where $k=k(\xi',\lambda)$.
Set $\Xi=(i\xi_1,\dots,i\xi_{N-1},-A)^\SST$, and thus  
\begin{align*}
&(\lambda^{1/2},\nabla)I_\pm \\
&=\int_0^\infty \CF_{\xi'}^{-1}\left[\frac{(\lambda^{1/2},\Xi) k R(A)}{A} 
\lambda^{1/2} e^{-A(x_N+y_N)}\wht{\lambda^{1/2} f_\pm}(\xi',\pm y_N)\right](x')\intd y_N \\
&-\sum_{j=1}^{N-1}\int_0^\infty \CF_{\xi'}^{-1}
\left[\frac{(\lambda^{1/2},\Xi) k i\xi_j}{
A|\xi'|} |\xi'|e^{-A(x_N+y_N)}\wht{\pd_j f_\pm}(\xi',\pm y_N)\right](x') \intd y_N, \\
&(\lambda^{1/2},\nabla)J_\pm \\
&=\int_0^\infty \CF_{\xi'}^{-1}\left[\frac{(\lambda^{1/2},\Xi)k R(A)\lambda^{1/2}}{A^2}\lambda^{1/2}e^{-A(x_N+y_N)}
\wht{\pd_N f_\pm}(\xi',\pm y_N)\intd y_N\right](x')\intd y_N \\
&-\sum_{j=1}^{N-1}\int_0^\infty \CF_{\xi'}^{-1}\left[\frac{(\lambda^{1/2},\Xi)k (i\xi_j)^2}{A^2 |\xi'|}|\xi'|e^{-A(x_N+y_N)}
\wht{\pd_N f_\pm}(\xi',\pm y_N)\intd y_N\right](x')\intd y_N.
\end{align*}
By Lemma \ref{lemm:symbol1} and Leibniz's formula,
we have for any multi-index $\alpha'\in\BN_0^{N-1}$
\begin{alignat*}{2}
\left|\pd_{\xi'}^{\alpha'}\left\{\frac{(\lambda^{1/2},\Xi)k(\xi',\lambda)R(A)}{A}\right\}\right|
&\leq C|\xi'|^{-|\alpha'|} \quad && (\xi'\in\BR^{N-1}\setminus\{0\}), \\
\left|\pd_{\xi'}^{\alpha'}\left\{\frac{(\lambda^{1/2},\Xi)k(\xi',\lambda)}{A}\frac{i\xi_j}{|\xi'|}\right\}\right|
&\leq C|\xi'|^{-|\alpha'|} \quad && (\xi'\in\BR^{N-1}\setminus\{0\}), \\
\left|\pd_{\xi'}^{\alpha'}\left\{\frac{(\lambda^{1/2},\Xi)k(\xi',\lambda)R(A)\lambda^{1/2}}{A^2}\right\}\right|
&\leq C|\xi'|^{-|\alpha'|} \quad && (\xi'\in\BR^{N-1}\setminus\{0\}), \\
\left|\pd_{\xi'}^{\alpha'}\left\{\frac{(\lambda^{1/2},\Xi)k(\xi',\lambda)i\xi_j}{A^2} \frac{i\xi_j}{|\xi'|}\right\}\right|
&\leq C|\xi'|^{-|\alpha'|} \quad && (\xi'\in\BR^{N-1}\setminus\{0\}),
\end{alignat*}
with some positive constant $C=C(N,\alpha',\sigma,\rho_+,\rho_-)$ independent of $\xi'$ and $\lambda$.
Combining these properties with Lemma \ref{lemm:tech-1} furnishes
\begin{align*}
\|(\lambda^{1/2} I_\pm,\nabla I_\pm)\|_{L_q(\BR_+^N)}
\leq C\|(\lambda^{1/2}f_\pm,\nabla f_\pm)\|_{L_q(\BR_\pm^N)}, \\
\|(\lambda^{1/2} J_\pm,\nabla J_\pm)\|_{L_q(\BR_+^N)}
\leq C\|(\lambda^{1/2}f_\pm,\nabla f_\pm)\|_{L_q(\BR_\pm^N)}, 
\end{align*}
which implies \eqref{0815:1_2019} holds. 
Analogously, we can prove \eqref{0815:2_2019}-\eqref{0815:4_2019}. 
This completes the proof of (1).

(2). The estimates \eqref{0815:5_2019} and \eqref{0815:7_2019} follow respectively from
\eqref{0815:1_2019} and \eqref{0815:2_2019}.
In addition, \eqref{0815:6_2019} and \eqref{0815:8_2019} follow respectively from \eqref{0815:5_2019} and \eqref{0815:7_2019}.
This completes the proof of (2).
\end{proof}

Let ${\rm sign}(a)$ be the sign function of $a$, that is, 
${\rm sign}(a)=1$ when $a>0$, 
${\rm sign}(a)=-1$ when $a<0$, and ${\rm sign}(a)=0$ when $a=0$.  
Then we have

\begin{lemm}\label{lemm:residue}
Let $\xi'\in\BR^{N-1}\setminus\{0\}$ and $\vps>0$.
Then, for any $\lambda\in\BC\setminus(-\infty,0]$, 
\begin{equation}\label{residue:1}
\frac{1}{2\pi}\int_{-\infty}^\infty e^{-\vps|\xi|^2} \frac{i\xi_N e^{ia\xi_N}}{\rho_\pm\lambda+|\xi|^2}\intd\xi_N
=-\frac{{\rm sign}(a)}{2}e^{\vps\rho_\pm\lambda}e^{-A_\pm|a|},
\end{equation}
where $\xi=(\xi',\xi_N)$. In addition, 
\begin{equation}\label{residue:2}
\frac{1}{2\pi}\int_{-\infty}^\infty e^{-\vps|\xi|^2} \frac{i\xi_N e^{ia\xi_N}}{|\xi|^2}\intd\xi_N
=-\frac{{\rm sign}(a)}{2}e^{-|\xi'||a|}.
\end{equation}
\end{lemm}

\begin{proof}
These formulas follow from the residue theorem immediately,
so that the detailed proof may be omitted.
\end{proof}

%%%%%%%%%%%%%%%%%%%%%%%%%%%%%%%%%%%%%%%%%%%%%%%%%%%%%%%%%%%%%%%%%%%%%%
\section{Problems in the whole space with flat interface}\label{sec:whole1}
Let us introduce the flat interface:
\begin{equation*}
\BR_0^N
=\{x=(x',x_N) : x'=(x_1,\dots,x_{N-1})\in\BR^{N-1},\,x_N=0\}.
\end{equation*}
This section mainly considers two strong elliptic problems as follows:
\begin{equation}\label{s-eq:2-bdd}
\left\{\begin{aligned}
\rho_\pm \lambda v_\pm- \Delta v_\pm &= -\dv\Bf_\pm + g_\pm && \text{in $\BR_\pm^N$,} \\
\rho_+ v_+&=\rho_- v_- && \text{on $\BR_0^N$,} \\
\pd_N v_+-\pd_N v_-
&=f_{+N}-f_{-N} + h_+-h_- && \text{on $\BR_0^N$,}
\end{aligned}\right.
\end{equation}
where $\lambda$ is the resolvent parameter varying in $\Sigma_{\sigma,0}$ $(0<\sigma<\pi/2)$, and 
\begin{equation}\label{s-eq:2-bdd_2}
\left\{\begin{aligned}
\Delta v_\pm &= \dv\Bf_\pm && \text{in $\BR_\pm^N$,} \\
\rho_+ v_+&=\rho_- v_- && \text{on $\BR_0^N$,} \\
\pd_N v_+ - \pd_N v_- 
 &=f_{+N}-f_{-N} && \text{on $\BR_0^N$.}
\end{aligned}\right.
\end{equation}
Throughout this section, we assume that $\rho_\pm$ are positive constants.
Concerning \eqref{s-eq:2-bdd} and \eqref{s-eq:2-bdd_2}, we prove the following theorems.

\begin{theo}\label{theo:2-bdd}
Let $\sigma\in(0,\pi/2)$ and $q\in(1,\infty)$.
Then, for any 
$$
\Bf_\pm=(f_{\pm 1},\dots,f_{\pm N})^\SST\in F_q(\BR_\pm^N), \quad
g_\pm \in L_q(\BR_\pm^N),  \quad h_\pm \in H_q^1(\BR_\pm^N)
$$
and for any $\lambda\in\Sigma_{\sigma,0}$,
the strong elliptic problem \eqref{s-eq:2-bdd} admits unique solutions
$v_\pm \in H_q^2(\BR_\pm^N)$.
In addition, the solutions $v_\pm$ satisfy
\begin{align}\label{est:1-whole}
&\sum_{\Fs\in\{+,-\}}\|(\lambda v_\Fs,\lambda^{1/2}\nabla v_\Fs,\nabla^2 v_\Fs)\|_{L_q(\BR_\Fs^N)} \\
&\leq C_1\sum_{\Fs\in\{+,-\}} 
\|(\dv\Bf_\Fs,g_\Fs,\lambda^{1/2} f_{\Fs N},\nabla f_{\Fs N}, \lambda^{1/2}h_\Fs,\nabla h_\Fs)\|_{L_q(\BR_\Fs^N)}, \notag
\end{align}
and also
\begin{align}\label{est:2-whole}
&\sum_{\Fs\in\{+,-\}}\|(\lambda^{1/2}v_\Fs ,\nabla v_\Fs)\|_{L_q(\BR_\Fs^N)} 
\leq C_1 \sum_{\Fs\in\{+,-\}}
\Big(\|\Bf_\Fs\|_{L_q(\BR_\Fs^N)}  \\
&+|\lambda|^{-1/2}\|g_\Fs\|_{L_q(\BR_\Fs^N)}  
 +\|h_\Fs\|_{L_q(\BR_\Fs^N)} 
+|\lambda|^{-1/2}\|\nabla h_\Fs\|_{L_q(\BR_\Fs^N)} \Big), \notag
\end{align}
where $C_1=C_1(N,q,\sigma,\rho_+,\rho_-)$ is a positive constant independent of $\lambda$.
Additionally, if $f_{+N}-f_{-N}=0$ on $\BR_0^N$, then $v_\pm$ satisfy
\begin{align}\label{est:1-whole-jumpzero}
&\sum_{\Fs\in\{+,-\}}\|(\lambda v_\Fs,\lambda^{1/2}\nabla v_\Fs,\nabla^2 v_\Fs)\|_{L_q(\BR_\Fs^N)} \\
&\leq C_1\sum_{\Fs\in\{+,-\}} 
\|(\dv\Bf_\Fs,g_\Fs,\lambda^{1/2}h_\Fs,\nabla h_\Fs)\|_{L_q(\BR_\Fs^N)}. \notag
\end{align}
\end{theo}

\begin{theo}\label{theo:2-bdd_2}
Let $q\in(1,\infty)$.
\begin{enumerate}[$(1)$]
\item {\bf Existence}. \label{theo:2-bdd_2-1}
Let $\Bf_\pm=(f_{\pm 1},\dots,f_{\pm N})^\SST \in F_q(\BR_\pm^N)$.
Then the strong problem \eqref{s-eq:2-bdd_2} admits solutions
$v_\pm \in \wht H_q^1(\BR_\pm ^N) \cap \wht H_q^2(\BR_\pm^N)$ satisfying
\begin{align}
\sum_{\Fs\in\{+,-\}}\|\nabla^2 v_\Fs \|_{L_q(\BR_\Fs^N)}
&\leq 
C_2\sum_{\Fs\in\{+,-\}}\|(\dv\Bf_\Fs, \nabla f_{\Fs N})\|_{L_q(\BR_\Fs^N)}, \label{0819:1_2019} \\
\sum_{\Fs\in\{+,-\}}\|\nabla v_\Fs \|_{L_q(\BR_\Fs^N)}
&\leq C_2\sum_{\Fs\in\{+,-\}}\|\Bf_\Fs\|_{L_q(\BR_\Fs^N)}, \label{0819:2_2019} 
\end{align}
with some positive constant $C_2=C_2(N,q,\rho_+,\rho_-)$.
Additionally, if $f_{+N}-f_{-N}=0$ on $\BR_0^N$, then $v_\pm$ satisfy
\begin{equation}\label{0824:1_2019}
\sum_{\Fs\in\{+,-\}}\|\nabla^2 v_\Fs \|_{L_q(\BR_\Fs^N)}
\leq 
C_2\sum_{\Fs\in\{+,-\}}\|\dv\Bf_\Fs\|_{L_q(\BR_\Fs^N)}.
\end{equation}
\item {\bf Uniqueness}. \label{theo:2-bdd_2-2}
If $v_\pm\in\wht H_q^1(\BR_\pm^N)\cap \wht H_q^2(\BR_\pm^N)$ satisfy
\begin{equation}\label{unique:flat}
\Delta v_\pm=0\text{ in $\BR_\pm^N$,} \quad
\rho_+v_+=\rho_-v_- \text{\ and \ } \pd_N v_+-\pd_N v_- =0 \text{ on $\BR_0^N$} 
\end{equation}
then $v_\pm =\rho_\pm^{-1}c$ for some constant $c$.
\end{enumerate}
\end{theo}

%%%%%%%%%%%%%%%%%%%%%%%%%%%%%%%%%%%%%%%%%%%%%%%%%%
\subsection{Auxiliary problems}\label{subsec:whole1-auxi}
We start with 
\begin{equation}\label{eq1:whole-2}
\rho_\pm\lambda U_{\pm,\lambda}-\Delta U_{\pm,\lambda} = -\dv\Bf_\pm+ g_\pm \quad \text{in $\BR_\pm^N$.}
\end{equation}
For this problem, we prove

\begin{lemm}\label{lemm:whole-lambda}
Let $\sigma\in(0,\pi/2)$ and $q\in(1,\infty)$.
Then, for any $\lambda\in\Sigma_{\sigma,0}$,
$\Bf_\pm=(f_{\pm 1},\dots,f_{\pm N})^\SST\in C_0^\infty(\BR^N)^N$, 
and $g_\pm \in C_0^\infty(\BR_\pm^N)$,
there exist $U_{\pm,\lambda}\in H_q^2(\BR^N)$, satisfying \eqref{eq1:whole-2},
such that the following assertions hold.
\begin{enumerate}[$(1)$]
\item\label{lemm:whole-lambda-1}
There hold the estimates:
\begin{align*}
&\|(\lambda U_{\pm,\lambda},\lambda^{1/2}\nabla U_{\pm,\lambda},\nabla^2 U_{\pm,\lambda})\|_{L_q(\BR^N)} 
\leq C\|(\dv\Bf_\pm,g_\pm)\|_{L_q(\BR_\pm^N)}, \\
&\|(\lambda^{1/2}U_{\pm,\lambda},\nabla U_{\pm,\lambda})\|_{L_q(\BR^N)} 
\leq C\left(\|\Bf_\pm\|_{L_q(\BR_\pm^N)} +|\lambda|^{-1/2}\|g_\pm\|_{L_q(\BR_\pm^N)}\right),
\end{align*}
with a positive constant $C=C(N,q,\sigma,\rho_+,\rho_-)$ independent of $\lambda$. 
\item\label{lemm:whole-lambda-2}
The traces of $\pd_N U_{\pm,\lambda}$ on $\BR_0^N$ are given by
\begin{align*}
(\pd_N U_{\pm,\lambda})(x',0)
&=
f_{\pm N}(x',0)  
\mp \sum_{j=1}^{N-1}\int_0^\infty \CF_{\xi'}^{-1}\left[i\xi_j e^{-A_\pm y_N}\wht f_{\pm j}(\xi',\pm y_N)
\right](x')\intd y_N \notag \\
&- \int_0^\infty \CF_{\xi'}^{-1}\left[A_\pm e^{-A_\pm y_N}\wht f_{\pm N}(\xi',\pm y_N)
\right](x')\intd y_N,
\end{align*}
where $A_{\pm}$ are defined as \eqref{defi:A} and
the symbols $\CF_{\xi'}^{-1}$ and $\wht{\,\cdot\,}$ are defined as \eqref{PFT}. 
\end{enumerate}
\end{lemm}

\begin{proof}
Let us introduce the Fourier transform of $u=u(x)$ on $\BR^N$
and the inverse Fourier transform of $v=v(\xi)$ on $\BR^N$ as follows:
\begin{equation*}
\CF[u](\xi)=\int_{\BR^N} e^{-ix\cdot\xi} u(x)\intd x, \quad
\CF_{\xi}^{-1}[v](x) = \frac{1}{(2\pi)^N}\int_{\BR^N} e^{ix\cdot\xi} v(\xi) \intd \xi.
\end{equation*}

Let $u_\pm=u_\pm (x',x_N)$ be functions defined on $\BR_\pm^N$,
and define the odd extensions $E_\pm^o u_\pm$ of $u_\pm$
and the even extensions $E_\pm^e u_\pm$ of $u_\pm$ as follows:
\begin{align*}
(E_+^o u_+)(x)
&=\left\{\begin{aligned}
& u_+(x',x_N) && (x_N>0), \\
& -u_+(x',-x_N) && (x_N<0), 
\end{aligned}\right. \\
(E_-^o u_-)(x)
&=\left\{\begin{aligned}
& -u_-(x',-x_N) && (x_N>0), \\
& u_-(x',x_N) && (x_N<0), 
\end{aligned}\right. \\
(E_+^e u_+)(x)
&=\left\{\begin{aligned}
& u_+(x',x_N) && (x_N>0), \\
& u_+(x',-x_N) && (x_N<0), 
\end{aligned}\right. \\
(E_-^e u_-)(x)
&=\left\{\begin{aligned}
& u_-(x',-x_N) && (x_N>0), \\
& u_-(x',x_N) && (x_N<0).
\end{aligned}\right. 
\end{align*} 
In addition, we set
\begin{equation}\label{ext:F}
\BF_\pm = (F_{\pm 1},\dots,F_{\pm N-1},F_{\pm N})^\SST = \left(E_\pm^o f_{\pm 1},\dots,
E_\pm^o f_{\pm N-1}, E_\pm^e f_{\pm N}\right)^\SST.
\end{equation}
It then holds that
$$
\pd_j F_{\pm j} = E_\pm^o\pd_j f_{\pm j} \quad (j=1,\dots,N-1), \quad
\pd_N F_{\pm N}=E_\pm^o \pd_N f_{\pm N},
$$
which implies 
\begin{equation}\label{eq1-2:whole}
\dv\BF_\pm=E_\pm^oz_\pm \quad \text{for } z_\pm=\dv\Bf_\pm.
\end{equation}
We thus see that 
\begin{equation}\label{eq1:whole}
\CF[\dv \BF_\pm](\xi)=i\xi \cdot \CF[\BF_\pm](\xi)=\CF[E_\pm^oz_\pm](\xi).
\end{equation}

Let us define for $\vps>0$
\begin{equation*}
V_\pm^\vps(x)=-\CF_\xi^{-1}\left[e^{-\vps|\xi|^2}\frac{i\xi \cdot \CF[\BF_\pm](\xi)}{\rho_\pm\lambda+|\xi|^2}\right](x)
\end{equation*}
and 
\begin{equation}\label{0816:1_2019}
F_{\pm, \dv}^\vps=\CF_\xi^{-1}\left[e^{-\vps|\xi|^2}\CF[\dv\BF_\pm](\xi)\right](x).
\end{equation}
Then $V_\pm^\vps$ solve by \eqref{eq1:whole}
\begin{equation}\label{eq3:whole}
\rho_\pm\lambda V_\pm^\vps -\Delta V_\pm^\vps=-F_{\pm,\dv}^\vps \quad \text{ in $\BR^N$.}
\end{equation}
On the other hand, one has
\begin{align*}
\lambda^{1/2} V_\pm^\vps(x)
&=-\CF_\xi^{-1}\left[e^{-\vps|\xi|^2}\frac{\lambda^{1/2} i\xi \cdot \CF[\BF_\pm](\xi)}{\rho_\pm\lambda+|\xi|^2}\right](x), \\
\pd_j V_\pm^\vps (x)
&= \CF_\xi^{-1}\left[e^{-\vps|\xi|^2}\frac{\xi_j\xi \cdot \CF[\BF_\pm](\xi)}{\rho_\pm\lambda+|\xi|^2}\right](x) \quad (j=1,\dots,N), 
\end{align*}
and also by \eqref{eq1:whole}
\begin{align*}
\lambda V_\pm^\vps(x)
&=-\CF_\xi^{-1}\left[e^{-\vps|\xi|^2}\frac{\lambda}{\rho_\pm\lambda+|\xi|^2}\CF[\dv\BF_\pm](\xi)\right](x), \\
\lambda^{1/2}\pd_j V_\pm^\vps (x)
&=- \CF_\xi^{-1}\left[e^{-\vps|\xi|^2}\frac{\lambda^{1/2} i\xi_j}{\rho_\pm\lambda+|\xi|^2}\CF[\dv\BF_\pm](\xi)\right](x) \quad (j=1,\dots,N), \\
\pd_k\pd_l V_\pm^\vps (x) 
&=\CF_\xi^{-1}\left[e^{-\vps|\xi|^2}\frac{\xi_k\xi_l}{\rho_\pm\lambda+|\xi|^2}\CF[\dv\BF_\pm](\xi)\right](x) \quad (k,l=1,\dots,N).
\end{align*}
For $j,k,l=1,\dots,N$ and for any multi-index $\alpha\in\BN_0^N$,
\begin{align*}
\left|\pd_{\xi}^{\alpha} \left(\frac{\lambda}{\rho_\pm\lambda+|\xi|^2}\right)\right|
&\leq C|\xi|^{-|\alpha|} \quad (\xi\in\BR^N\setminus\{0\}), \\
\left|\pd_{\xi}^{\alpha} \left( \frac{\lambda^{1/2}i\xi_j}{\rho_\pm\lambda+|\xi|^2}\right)\right|
&\leq C|\xi|^{-|\alpha|} \quad (\xi\in\BR^N\setminus\{0\}), \notag \\
\left|\pd_{\xi}^{\alpha} \left( \frac{\xi_k\xi_l}{\rho_\pm\lambda+|\xi|^2}\right)\right|
&\leq C|\xi|^{-|\alpha|} \quad (\xi\in\BR^N\setminus\{0\}), \notag
\end{align*}
with some positive constant $C=C(N,\alpha,\sigma)$ (cf. \cite[Section 3]{ShiShi12}).
Thus, applying  the classical Fourier multiplier theorem
to the above formulas of $V_\pm^\vps$ and setting
\begin{equation}\label{0816:2_2019}
\BF_\pm^\vps=\CF_\xi^{-1}\left[e^{-\vps |\xi|^2}\CF\left[\BF_\pm\right](\xi)\right](x)
\end{equation}
furnish the following estimates: 
\begin{align}\label{est:approx}
\|(\lambda^{1/2}V_\pm^\vps,\nabla V_\pm^\vps)\|_{L_q(\BR^N)}
&\leq C\|\BF_\pm^\vps\|_{L_q(\BR^N)}, \\
\|(\lambda V_\pm^\vps,\lambda^{1/2}\nabla V_\pm^\vps,\nabla^2 V_\pm^\vps)\|_{L_q(\BR^N)}
&\leq C\|F_{\pm,\dv}^\vps\|_{L_q(\BR^N)}, \notag
\end{align}
where $C$ is a positive constant independent of $\vps$.
Analogously,
\begin{align}\label{Cauchy:1}
|\lambda|\|V_\pm^\vps-V_\pm^{\vps'}\|_{L_q(\BR^N)}
&\leq C\|F_{\pm,\dv}^\vps-F_{\pm,\dv}^{\vps'}\|_{L_q(\BR^N)}, \\
|\lambda|^{1/2}\|\nabla(V_\pm^\vps-V_\pm^{\vps'})\|_{L_q(\BR^N)}
&\leq C\|F_{\pm,\dv}^\vps-F_{\pm,\dv}^{\vps'}\|_{L_q(\BR^N)}, \notag \\
\|\nabla^2(V_\pm^\vps-V_\pm^{\vps'})\|_{L_q(\BR^N)}
&\leq C\|F_{\pm,\dv}^\vps-F_{\pm,\dv}^{\vps'}\|_{L_q(\BR^N)}, \notag
\end{align}
for $\vps,\vps'>0$ and a positive constant $C$ independent of $\vps$ and $\vps'$.

Next, we compute the formulas of $\pd_N V_\pm^\vps$ on $\BR_0^N$.
We have
$$
\CF[E_\pm^oz_\pm](\xi)
=
\int_0^\infty\left(\pm e^{- i y_N\xi_N} \mp e^{i y_N\xi_N}\right)\wht z_\pm(\xi',\pm y_N)\intd y_N,
$$
while by \eqref{eq1:whole}
\begin{equation}\label{eq1:residue}
\pd_N V_\pm^\vps =-\CF_{\xi}^{-1}\left[e^{-\vps|\xi|^2}\frac{i\xi_N}{\rho_\pm\lambda+|\xi|^2}\CF[E_\pm^oz_\pm](\xi)\right](x).
\end{equation}
These formulas yield
\begin{align*}
&\pd_N V_\pm^\vps =-\int_0^\infty\CF_{\xi'}^{-1}\bigg[\wht z_\pm (\xi',\pm y_N) \\ 
&\cdot\left( \frac{1}{2\pi}\int_{-\infty}^\infty e^{-\vps|\xi|^2}\frac{i\xi_N}{\rho_\pm\lambda+|\xi|^2}
\left(\pm e^{i(x_N -  y_N)\xi_N} \mp e^{i(x_N + y_N)\xi_N}\right)\intd\xi_N\right)\bigg](x')\intd y_N.
\end{align*}
Inserting \eqref{residue:1} into the above formula of $\pd_N V_\pm^\vps$ furnishes
\begin{equation}\label{eq5-1:whole_0}
(\pd_N V_\pm^\vps)(x',0)=
\mp e^{\vps\rho_\pm\lambda} \int_0^\infty \CF_{\xi'}^{-1}
\left[e^{-A_\pm y_N} \wht z_\pm(\xi',\pm y_N)\right](x') \intd y_N.
\end{equation}

Now we consider the limit: $\vps\to 0$.
By the well-known property of the heat kernel $\CF_{\xi}^{-1}[e^{-\vps |\xi|^2}](x)$, we have
\begin{align*}
\lim_{\vps\to 0}\|\BF_\pm^\vps-\BF_\pm\|_{L_q(\BR^N)}=0, \quad
\lim_{\vps\to 0}\|F_{\pm,\dv}^\vps-\dv\BF_\pm\|_{L_q(\BR^N)}=0.
\end{align*}
Combining the last property with \eqref{Cauchy:1} shows that
there exist $V_\pm \in H_q^2(\BR^N)$ such that $\lim_{\vps\to 0}\|V_\pm^\vps-V_\pm\|_{H_q^2(\BR^N)}=0$.
One then sees that $V_\pm$ satisfy, by \eqref{eq1-2:whole}, \eqref{eq3:whole}, and \eqref{est:approx}, 
the equations:
\begin{equation*}
\rho_\pm\lambda V_\pm-\Delta V_\pm =-\dv\BF_\pm = -\dv\Bf_\pm \quad \text{in $\BR_\pm^N$}
\end{equation*}
and the estimates:
\begin{align}\label{est:whole-1}
&\|(\lambda^{1/2}V_\pm,\nabla V_\pm)\|_{L_q(\BR^N)}
\leq C\|\BF_\pm\|_{L_q(\BR^N)} \leq C\|\Bf_\pm\|_{L_q(\BR_\pm^N)}, \\
&\|(\lambda V_\pm,\lambda^{1/2}\nabla V_\pm,\nabla^2 V_\pm)\|_{L_q(\BR^N)}
\leq C\|\dv\BF_\pm\|_{L_q(\BR^N)}\leq C\|\dv\Bf_\pm\|_{L_q(\BR_\pm^N)}, \notag
\end{align}
while there holds by \eqref{eq5-1:whole_0}
\begin{equation*}
(\pd_N V_\pm)(x',0)=\mp \int_0^\infty \CF_{\xi'}^{-1}\left[e^{-A_\pm y_N} \wht z_\pm(\xi',\pm y_N)\right](x') \intd y_N.
\end{equation*}
Combining these formulas with
$$
\wht z_\pm(\xi',\pm y_N)=\sum_{j=1}^{N-1}i\xi_j\wht f_{\pm j}(\xi',\pm y_N)+(\pd_N \wht f_{\pm N})(\xi',\pm y_N)
$$
and with integration by parts furnishes
\begin{align}\label{eq5-2:whole}
(\pd_N  V_\pm)(x',0) &= f_{\pm N}(x',0)  \\
&\mp \sum_{j=1}^{N-1}\int_0^\infty \CF_{\xi'}^{-1}\left[i\xi_j e^{-A_\pm y_N}\wht f_{\pm j}(\xi',\pm y_N)
\right](x')\intd y_N \notag \\
&- \int_0^\infty \CF_{\xi'}^{-1}\left[A_\pm e^{-A_\pm y_N}\wht f_{\pm N}(\xi',\pm y_N)
\right](x')\intd y_N .
\notag
\end{align}

Finally, we set 
\begin{equation*}
W_\pm(x) = \CF_\xi^{-1}\left[\frac{\CF\left[E_\pm^e g_\pm\right](\xi)}
{\rho_\pm\lambda+|\xi|^2}\right](x) \quad (x\in\BR^N).
\end{equation*}
Then $W_\pm$ satisfy the equations:
$$
\rho_\pm\lambda W_\pm-\Delta W_\pm = E_\pm^e g_\pm=g_\pm \quad \text{in $\BR_\pm^N$,}
$$
and also satisfy, similarly to \eqref{est:approx}, the estimates: 
\begin{align}\label{est:whole-2}
&\|(\lambda^{1/2}W_\pm,\nabla W_\pm)\|_{L_q(\BR^N)}
\leq C|\lambda|^{-\frac{1}{2}}\|E_\pm^e g_\pm\|_{L_q(\BR^N)}
\leq C|\lambda|^{-\frac{1}{2}}\|g_\pm\|_{L_q(\BR_\pm^N)}, \\
&\|(\lambda W_\pm,\lambda^{1/2}\nabla W_\pm,\nabla^2 W_\pm)\|_{L_q(\BR^N)}
\leq C\|E_\pm^e g_\pm\|_{L_q(\BR^N)}
\leq C\|g_\pm\|_{L_q(\BR_\pm^N)}, \notag
\end{align}
where $C$ is a positive constant independent of $\lambda$.
In addition, we have\footnote{These properties follow from the uniqueness of solutions to $\lambda u-\Delta u = f$ in $\BR^N$,
see e.g. \cite[Subsection 2.4]{Saito2019}.}
\begin{equation}\label{eq5-3:whole}
\pd_N W_\pm =0 \quad \text{on $\BR_0^N$.}
\end{equation}
Thus, setting $U_{\pm,\lambda} = V_\pm+W_\pm$,
we see that $U_{\pm,\lambda}$ satisfy all of the properties required in Lemma \ref{lemm:whole-lambda}
by \eqref{est:whole-1}-\eqref{eq5-3:whole}.
This completes the proof of Lemma \ref{lemm:whole-lambda}.
\end{proof}

Next, we consider the case $\lambda=0$ and $g_\pm=0$ of \eqref{eq1:whole-2}:
\begin{equation}\label{eq1:whole-2_2}
\Delta U_{\pm,0} = \dv\Bf_\pm \quad \text{in $\BR_\pm^N$.}
\end{equation}
Concerning these equations, we prove

\begin{lemm}\label{lemm:whole-zero}
Let $q\in(1,\infty)$. Then, for any $\Bf_\pm=(f_{\pm1},\dots,f_{\pm N})^\SST\in C_0^\infty(\BR^N)$,
there exist $U_{\pm,0}\in\wht H_q^1(\BR^N)\cap \wht H_q^2(\BR^N)$, satisfying \eqref{eq1:whole-2_2},
such that the following assertions hold.
\begin{enumerate}[$(1)$]
\item
There hold the estimates:
\begin{equation*}
\|\nabla^2 U_{\pm,0}\|_{L_q(\BR^N)}
\leq C\|\dv\Bf_\pm\|_{L_q(\BR_\pm^N)}, \quad
\|\nabla U_{\pm,0}\|_{L_q(\BR^N)}
\leq C \|\Bf_\pm\|_{L_q(\BR_\pm^N)},
\end{equation*}
with a positive constant $C=C(N,q,\rho_+,\rho_-)$.
\item
The traces of $\pd_N U_{\pm,0}$ on $\BR_0^N$ are given by
\begin{align*}
(\pd_N U_{\pm,0})(x',0)
&=
f_{\pm N}(x',0)  
\mp \sum_{j=1}^{N-1}\int_0^\infty \CF_{\xi'}^{-1}\left[i\xi_j e^{-|\xi'|y_N}\wht f_{\pm j}(\xi',\pm y_N)
\right](x')\intd y_N \notag \\
&- \int_0^\infty \CF_{\xi'}^{-1}\left[|\xi'| e^{-|\xi'| y_N}\wht f_{\pm N}(\xi',\pm y_N)
\right](x')\intd y_N,
\end{align*}
where the symbols $\CF_{\xi'}^{-1}$ and $\wht{\,\cdot\,}$ are defined as \eqref{PFT}. 
\end{enumerate}
\end{lemm}

\begin{proof}
Let $\BF_\pm$, $F_{\pm,\dv}^\vps$, and $\BF_\pm^\vps$
be given in \eqref{ext:F}, \eqref{0816:1_2019}, and \eqref{0816:2_2019}, respectively.
Set
\begin{equation*}
U_{\pm,0}^\vps=-\CF_\xi^{-1}\left[e^{-\vps|\xi|^2}\frac{i\xi \cdot \CF[\BF_\pm](\xi)}{|\xi|^2}\right](x) \quad (\vps>0).
\end{equation*}
We then have
\begin{equation*}
\pd_j U_{\pm,0}^\vps(x)
= \CF_\xi^{-1}\left[e^{-\vps |\xi|^2}\frac{\xi_j\xi\cdot\CF[\BF_\pm](\xi)}{|\xi|^2}\right](x) \quad (j=1,\dots,N), 
\end{equation*}
and also by \eqref{eq1:whole}
\begin{equation*}
\pd_k\pd_l U_{\pm,0}^\vps(x)
=\CF_\xi^{-1}\left[e^{-\vps |\xi|^2}\frac{\xi_k\xi_l}{|\xi|^2}\CF[\dv\BF_\pm](\xi)\right](x) \quad (k,l=1,\dots,N).
\end{equation*}
By the $L_q$ boundedness of the Riesz operators,
\begin{align*}
&\|\nabla U_{\pm,0}^\vps\|_{L_q(\BR^N)} \leq C\|\BF_\pm^\vps\|_{L_q(\BR^N)}, \quad
\|\nabla^2 U_{\pm,0}^\vps\|_{L_q(\BR^N)}\leq C\|F_{\pm,\dv}^\vps\|_{L_q(\BR^N)}, \\
&\|\nabla(U_{\pm,0}^\vps-U_{\pm,0}^{\vps'})\|_{L_q(\BR^N)} \leq C\|\BF_\pm^\vps-\BF_\pm^{\vps'}\|_{L_q(\BR^N)}, \\
&\|\nabla^2(U_{\pm,0}^\vps-U_{\pm,0}^{\vps'})\|_{L_q(\BR^N)} \leq C\|F_{\pm,\dv}^\vps-F_{\pm,\dv}^{\vps'}\|_{L_q(\BR^N)}, 
\end{align*}
where $\vps,\vps'>0$ and $C$ is a positive constant independent of $\vps$ and $\vps'$.
Thus, similarly to the proof of Lemma \ref{lemm:whole-lambda},
we can construct $U_{\pm,0}\in\wht H_q^1(\BR^N)\cap \wht H_q^2(\BR^N)$
satisfying
$\Delta U_{\pm,0}=\dv\Bf_\pm$ in $\BR_\pm^N$ and the estimates:
\begin{align*}
\|\nabla U_{\pm,0}\|_{L_q(\BR^N)} \leq C\|\Bf_\pm\|_{L_q(\BR_\pm^N)}, \quad
\|\nabla^2 U_{\pm,0}\|_{L_q(\BR^N)} \leq C\|\dv\Bf_\pm\|_{L_q(\BR_\pm^N)}.
\end{align*}
We also obtain the formulas of $\pd_N U_{\pm,0}$, stated in Lemma \ref{lemm:whole-zero} (2), from \eqref{residue:2} and
$$
\pd_N U_{\pm,0}^\vps=-\CF_\xi^{-1}\left[e^{-\vps |\xi|^2}\frac{i\xi_N}{|\xi|^2}\CF[E_\pm^o z_\pm](\xi)\right](x)
$$
in the same manner that
we have obtained the formulas of $\pd_N U_{\pm,\lambda}$,
stated in Lemma \ref{lemm:whole-lambda} (2), from \eqref{residue:1} and \eqref{eq1:residue}.
This completes the proof of Lemma \ref{lemm:whole-zero}.
\end{proof}

%%%%%%%%%%%%%%%%%%%%%%%%%%%%%%%%%%%%%%%%%%%%%%%%%%
\subsection{Proof of Theorem \ref{theo:2-bdd}}
We prove Theorem \ref{theo:2-bdd} in this subsection.
In view of Lemma \ref{lemm:approx1}, it suffices to consider 
\begin{equation*}
\Bf_\pm =(f_{\pm 1},\dots, f_{\pm N})^\SST\in C_0^\infty(\BR^N)^N, \quad
g_\pm\in C_0^\infty(\BR_\pm^N), \quad h_\pm\in C_0^\infty(\BR^N).
\end{equation*}
Let $U_{\pm,\lambda}$ be solutions of \eqref{eq1:whole-2} constructed in Lemma \ref{lemm:whole-lambda}
for $\lambda\in\Sigma_{\sigma,0}$.
Setting $v_\pm=U_{\pm,\lambda}+w_\pm$ in \eqref{s-eq:2-bdd} yields
\begin{equation}\label{eq20:whole}
\left\{\begin{aligned}
\rho_\pm\lambda w_\pm-\Delta w_\pm &= 0 && \text{in $\BR_\pm^N$,} \\
\rho_+ w_+-\rho_- w_- &=g_1 && \text{on $\BR_0^N$,} \\
\pd_N w_+-\pd_N w_-&=g_2 && \text{on $\BR_0^N$.}
\end{aligned}\right.
\end{equation}
where 
\begin{align*}
g_1&=-(\rho_+U_{+,\lambda}-\rho_-U_{-,\lambda}), \\
g_2&=-(\pd_N U_{+,\lambda}-\pd_N U_{-,\lambda})+f_{+N}-f_{-N}+h_+-h_-.
\end{align*}
We apply the partial Fourier transform, given in \eqref{PFT},
to \eqref{eq20:whole} in order to obtain
\begin{align*}
\left\{\begin{aligned}
\rho_\pm\lambda\wht w_\pm(x_N)-(\pd_N^2-|\xi'|^2)\wht w_\pm(x_N)&=0, \quad \text{$\pm x_N>0$,} \\
\rho_+\wht w_+(0)-\rho_-\wht w_-(0)&=\wht g_1(0), \\
\pd_N\wht w_+(0)-\pd_N\wht w_-(0)&=\wht g_2(0).
\end{aligned}\right.
\end{align*}
Solving these ordinary differential equations with respect to $x_N$,
we have
\begin{align*}
\wht w_\pm(x_N)=\left(\pm \frac{A_\mp}{\rho_+A_-+\rho_-A_+}\wht g_1(0)
-\frac{\rho_\mp}{\rho_+A_-+\rho_-A_+}\wht g_2(0)\right)e^{-A_\pm (\pm x_N)}.
\end{align*}
Thus, setting $w_\pm=\CF_{\xi'}^{-1}[\wht w_\pm(x_N)](x')$,
we see that $w_{\pm}$ are solutions to \eqref{eq20:whole}.

From now on, we estimate $w_\pm$. To this end, we decompose $w_\pm$ as follows:
\begin{equation}\label{0818:7_2019}
w_\pm=\pm w_\pm^1-\rho_\mp(w_\pm^2-w_\pm^3+w_\pm^4),
\end{equation}
where
\begin{align*}
w_\pm^1&=\CF_{\xi'}^{-1}\left[\frac{A_\mp }{\rho_+A_-+\rho_-A_+}e^{-A_\pm (\pm x_N)}\wht g_1(0)\right](x'), \\
w_\pm^2&=\CF_{\xi'}^{-1}\left[\frac{1}{\rho_+A_-+\rho_-A_+}e^{- A_\pm (\pm x_N)} \wht h_+(0)\right](x'), \\
w_\pm^3&=\CF_{\xi'}^{-1}\left[\frac{1}{\rho_+A_-+\rho_-A_+}e^{- A_\pm (\pm x_N)} \wht h_-(0)\right](x'), \\
w_\pm^4&=\CF_{\xi'}^{-1}\left[\frac{1}{\rho_+A_-+\rho_-A_+}e^{- A_\pm (\pm x_N)} \wht g_3(0)\right](x'),
\end{align*}
with 
$$
g_3=g_2-(h_+-h_-)=-(\pd_N U_{+,\lambda}-\pd_N U_{-,\lambda})+f_{+N}- f_{-N}.
$$
By Lemma \ref{lemm:symbol1} and Leibniz's formula, we have for $A\in\{A_+,A_-,i\xi_1,\dots,i\xi_{N-1}\}$
\begin{equation}\label{0818:1_2019}
\left|\pd_{\xi'}^{\alpha'}\left(\frac{A}{\rho_+ A_-+\rho_-A_+}\right)\right|
\leq C(|\lambda|^{1/2}+|\xi'|)^{-|\alpha'|},
\end{equation}
where $\xi'=(\xi_1,\dots,\xi_{N-1})\in\BR^{N-1}\setminus\{0\}$, $\lambda\in\Sigma_{\sigma,0}$,
and $C=C(N,\alpha',\sigma)>0$. 
Thus, applying Lemma \ref{lemm:multiplier}
to the above formula of $w_\pm^1$ yields
\begin{align}\label{est:whole-w1}
&\|(\lambda w_\pm^1,\lambda^{1/2}\nabla w_\pm^1,\nabla^2 w_\pm^1)\|_{L_q(\BR_\pm^N)} \\
&\leq C\sum_{\Fs\in\{+,-\}}
\|(\lambda U_{\Fs,\lambda},\lambda^{1/2}\nabla U_{\Fs,\lambda},\nabla^2 U_{\Fs,\lambda})\|_{L_q(\BR_\Fs^N)}, \notag \\
&\|(\lambda^{1/2} w_\pm^1,\nabla w_\pm^1)\|_{L_q(\BR_\pm^N)} 
\leq C\sum_{\Fs\in\{+,-\}}\|(\lambda^{1/2} U_{\Fs,\lambda},\nabla U_{\Fs,\lambda})\|_{L_q(\BR_\pm^N)}. \notag
\end{align}
Analogously, we have from Lemmas \ref{lemm:symbol1} and \ref{lemm:multiplier}
\begin{align}\label{est:whole-w2}
&\|(\lambda w_\pm^2,\lambda^{1/2}\nabla w_\pm^2,\nabla^2 w_\pm^2)\|_{L_q(\BR_\pm^N)}
\leq C\|(\lambda^{1/2}h_+,\nabla h_+)\|_{L_q(\BR_+^N)}, \\
&\|(\lambda^{1/2}w_\pm^2,\nabla w_\pm^2)\|_{L_q(\BR_\pm^N)}
\leq C\left(\|h_+\|_{L_q(\BR_+^N)}+|\lambda|^{-1/2}\|\nabla h_+\|_{L_q(\BR_+^N)}\right), \notag
\end{align}
and also
\begin{align}\label{est:whole-w2^2}
&\|(\lambda w_\pm^3,\lambda^{1/2}\nabla w_\pm^3,\nabla^2 w_\pm^3)\|_{L_q(\BR_\pm^N)}
\leq C\|(\lambda^{1/2}h_-,\nabla h_-)\|_{L_q(\BR_-^N)}, \\
&\|(\lambda^{1/2}w_\pm^3,\nabla w_\pm^3)\|_{L_q(\BR_\pm^N)}
\leq C\left(\|h_-\|_{L_q(\BR_-^N)}+|\lambda|^{-1/2}\|\nabla h_-\|_{L_q(\BR_-^N)}\right). \notag
\end{align}
Furthermore, Lemmas \ref{lemm:symbol1} and \ref{lemm:multiplier} yield
\begin{align}\label{est:whole-w3-1}
&\|(\lambda w_\pm^4,\lambda^{1/2}\nabla w_\pm^4,\nabla^2 w_\pm^4)\|_{L_q(\BR_\pm^N)} \\
&\leq C \sum_{\Fs\in\{+,-\}}\|(\lambda^{1/2}\pd_N U_{\Fs,\lambda}, \nabla \pd_N U_{\Fs,\lambda},
\lambda^{1/2} f_{\Fs N},\nabla f_{\Fs N})\|_{L_q(\BR_\Fs^N)}. \notag 
\end{align}
On the other hand, if $f_{+N}-f_{-N}=0$ on $\BR_-^N$, then we have
\begin{align}\label{est:whole-w3-1-1}
&\|(\lambda w_\pm^4,\lambda^{1/2}\nabla w_\pm^4,\nabla^2 w_\pm^4)\|_{L_q(\BR_\pm^N)}  \\
&\leq C \sum_{\Fs\in\{+,-\}}\|(\lambda^{1/2}\pd_N U_{\Fs,\lambda}, \nabla \pd_N U_{\Fs,\lambda}
)\|_{L_q(\BR_\Fs^N)}.  \notag
\end{align}

Next, we estimate $\|(\lambda^{1/2}w_\pm^4,\nabla w_\pm^4)\|_{L_q(\BR_\pm^N)}$.
Note that by Lemma \ref{lemm:whole-lambda} \eqref{lemm:whole-lambda-2}
\begin{align*}
\wht g_3(0)
&=\sum_{\Fs\in\{+,-\}}\sum_{j=1}^{N-1}\int_0^\infty
i\xi_j e^{-A_{\Fs}y_N}\wht f_{\Fs j}(\xi',\Fs y_N)\intd y_N \\
&+\sum_{\Fs\in\{+,-\}}\Fs\int_0^\infty
A_\Fs e^{-A_\Fs y_N}\wht f_{\Fs N}(\xi',\Fs y_N)\intd y_N.
\end{align*}
We thus obtain
\begin{align*}
w_\pm^4
&=
\sum_{\Fs\in\{+,-\}}\sum_{j=1}^{N-1}\int_0^\infty
\CF_{\xi'}^{-1}\left[\frac{i\xi_j e^{-(A_\pm(\pm x_N)+ A_{\Fs}y_N)}}{\rho_+A_-+\rho_-A_+} \wht f_{\Fs j}(\xi',\Fs y_N)\right](x')\intd y_N \\
&+\sum_{\Fs\in\{+,-\}}
\Fs\int_0^\infty
\CF_{\xi'}^{-1}\left[\frac{A_\Fs e^{-(A_\pm(\pm x_N)+ A_{\Fs}y_N)}}{\rho_+A_-+\rho_-A_+} \wht f_{\Fs N}(\xi',\Fs y_N)\right](x')\intd y_N.
\end{align*}
Combining these formulas with Lemma \ref{lemm:tech-1}, \eqref{eq1:20191113}, and \eqref{0818:1_2019} furnishes
\begin{align*}
\|(\lambda^{1/2}w_\pm^4,\nabla w_\pm^4)\|_{L_q(\BR_\pm^N)}
\leq C\left(\|\Bf_+\|_{L_q(\BR_+^N)}+\|\Bf_-\|_{L_q(\BR_-^N)}\right).
\end{align*}
Recalling $v_\pm=U_{\pm,\lambda}+w_\pm$ and \eqref{0818:7_2019},
we obtain solutions $v_\pm$ of \eqref{s-eq:2-bdd} satisfying \eqref{est:1-whole}-\eqref{est:1-whole-jumpzero}
by the last estimate, \eqref{est:whole-w1}-\eqref{est:whole-w3-1-1}, and Lemma \ref{lemm:whole-lambda}.

Finally, we prove the uniqueness of solutions to \eqref{s-eq:2-bdd}.
Assume that $v_\pm \in H_q^2(\BR_\pm^N)$
satisfy the homogeneous equations, i.e.
\begin{align*}
\left\{\begin{aligned}
\rho_\pm\lambda v_\pm -\Delta v_\pm &=0 && \text{in $\BR_\pm^N$,} \\
\rho_+v_+&=\rho_-v_- && \text{on $\BR_0^N$,} \\
\pd_N v_+&=\pd_N v_- && \text{on $\BR_0^N$.}
\end{aligned}\right.
\end{align*}
Let $\Bf_\pm\in C_0^\infty(\BR_\pm^N)^N$.
Since we know the existence of solutions to \eqref{s-eq:2-bdd} for these $\Bf_\pm$ as already discussed above,
we have $w_\pm\in H_{q'}^2(\BR_\pm^N)$, $q'=q/(q-1)$, satisfying
\begin{equation*}
\left\{\begin{aligned}
\rho_\pm\lambda w_\pm - \Delta w_\pm &= -\dv\Bf_\pm && \text{in $\BR_\pm^N$,} \\
\rho_+ w_+&=\rho_- w_- && \text{on $\BR_0^N$,} \\
\pd_N w_+&=\pd_N w_- && \text{on $\BR_0^N$.}
\end{aligned}\right.
\end{equation*}
It then holds by integration by parts that
\begin{align*}
&(\rho_+\nabla v_+, \Bf_+)_{\BR_+^N}+(\rho_-\nabla v_-,\Bf_-)_{\BR_-^N}
=-(\rho_+v_+,\dv\Bf_+)_{\BR_+^N}-(\rho_-v_-,\dv\Bf_-)_{\BR_-^N} \\
&=(\rho_+v_+,\rho_+\lambda w_+-\Delta w_+)_{\BR_+^N}+(\rho_-v_-,\rho_-\lambda w_--\Delta w_-)_{\BR_-^N} \\
&=(\rho_+\lambda v_+-\Delta v_+, \rho_+ w_+)_{\BR_+^N}
+(\rho_-\lambda v_--\Delta v_-,\rho_- w_-)_{\BR_-^N}=0,
\end{align*}
and thus $\rho_\pm v_\pm=c_\pm$ for some constants $c_\pm$. 
By the equation: $\rho_\pm\lambda v_\pm-\Delta v_\pm=0$ in $\BR_\pm^N$,
we conclude $c_\pm=0$ since $\lambda\neq 0$. This implies the uniqueness of solutions to \eqref{s-eq:2-bdd},
which completes the proof of Theorem \ref{theo:2-bdd}.

%%%%%%%%%%%%%%%%%%%%%%%%%%%%%%%%%%%%%%%%%%%%%%%%%%
\subsection{Proof of Theorem \ref{theo:2-bdd_2}}\label{subsec:3-3}
We prove Theorem \ref{theo:2-bdd_2} in this subsection.
In view of Lemma \ref{lemm:approx1}, it suffices to consider 
\begin{equation*}
\Bf_\pm =(f_{\pm 1},\dots, f_{\pm N})^\SST\in C_0^\infty(\BR^N)^N.
\end{equation*}
Let $U_{\pm,0}$ be the solutions of \eqref{eq1:whole-2_2} constructed in Lemma \ref{lemm:whole-zero},
and set $v_\pm=U_{\pm,0}+ w_\pm$ in \eqref{s-eq:2-bdd_2}.
Then,
\begin{equation}\label{eq6:whole}
\left\{\begin{aligned}
\Delta w_\pm &= 0 && \text{in $\BR_\pm^N$,} \\
\rho_+ w_+-\rho_- w_-&=g_1 && \text{on $\BR_0^N$,} \\
\pd_N w_+-\pd_N w_-&=g_2  && \text{on $\BR_0^N$,}
\end{aligned}\right.
\end{equation}
where 
$$
g_1=-(\rho_+ U_{+,0}-\rho_- U_{-,0}), \quad
g_2=-(\pd_N U_{+,0}-\pd_N U_{-,0})+f_{+N}-f_{-N}.
$$
We apply the partial Fourier transform, given in \eqref{PFT}, to \eqref{eq6:whole} 
in order to obtain
\begin{equation*}
\left\{\begin{aligned}
(\pd_N^2-|\xi'|^2)\wht w_\pm(x_N)&=0, \quad \pm x_N>0, \\
\rho_+\wht w_+(0)-\rho_-\wht w_-(0)&=\wht g_1 (0), \\
\pd_N\wht w_+(0)-\pd_N\wht w_-(0)&=\wht g_2(0).
\end{aligned}\right.
\end{equation*}
Solving these ordinary differential equations with respect to $x_N$,
we have
\begin{equation*}
\wht w_\pm(x_N)=\left(\frac{1}{\rho_++\rho_-}\right)
\left(\pm\wht g_1(0)-\rho_\mp\frac{\wht g_2(0)}{|\xi'|}\right)e^{\mp|\xi'|x_N}.
\end{equation*}
Thus, setting $w_\pm=\CF_{\xi'}^{-1}[\wht w_\pm(x_N)](x')$,
we see that $w_{\pm}$ are solutions to \eqref{eq6:whole}.

From now on, we estimate $w_\pm$. To this end, we decompose $w_\pm$ as follows:
\begin{equation}\label{0818:9_2019}
w_\pm=\pm\left(\frac{1}{\rho_++\rho_-}\right) w_\pm^1-\left(\frac{\rho_\mp}{\rho_++\rho_-}\right) w_\pm^2,
\end{equation}
where
\begin{equation*}
w_\pm^1=\CF_{\xi'}^{-1}\left[\wht g_1(0) e^{\mp |\xi'|x_N}\right](x'), \quad
w_\pm^2=\CF_{\xi'}^{-1}\left[\frac{\wht g_2(0)}{|\xi'|}e^{\mp |\xi'|x_N}\right](x').
\end{equation*}

Let us consider $\|\nabla w_\pm^1\|_{L_q(\BR_\pm^N)}$ and $\|\nabla^2 w_\pm^1\|_{L_q(\BR_\pm^N)}$.
Similarly to \eqref{vole:1}, one has 
\begin{align*}
\wht g_1(0)e^{\mp |\xi'|x_N}
&=\int_0^\infty|\xi'|e^{-|\xi'|(\pm x_N+y_N)}\wht g_1(\xi',\pm y_N) \intd y_N \\
&\mp \int_0^\infty e^{-|\xi'|(\pm x_N+y_N)}\wht{\pd_N g_1}(\xi',\pm y_N)\intd y_N, 
\end{align*}
which, combined with $|\xi'|=-\sum_{j=1}^{N-1}(i\xi_j)^2/|\xi'|$, furnishes
\begin{align*}
\wht g_1(0)e^{\mp |\xi'|x_N}
&=-\sum_{j=1}^{N-1}\int_0^\infty \frac{i\xi_j}{|\xi'|}e^{-|\xi'|(\pm x_N+y_N)}\wht{\pd_j g_1}(\xi',\pm y_N) \intd y_N \\
&\mp \int_0^\infty e^{-|\xi'|(\pm x_N+y_N)}\wht {\pd_N g_1}(\xi',\pm y_N)\intd y_N. \notag
\end{align*}
Inserting these relations into $w_\pm^1$ yields
\begin{align*}
w_\pm^1 &=
-\sum_{j=1}^{N-1}\int_0^\infty \CF_{\xi'}^{-1}\left[\frac{i\xi_j}{|\xi'|}e^{-|\xi'|(\pm x_N+y_N)}\wht{\pd_j g_1}(\xi',\pm y_N)\right](x') \intd y_N  \\
&\mp \int_0^\infty \CF_{\xi'}^{-1}\left[e^{-|\xi'|(\pm x_N+y_N)}\wht {\pd_N g_1}(\xi',\pm y_N)\right](x')\intd y_N.
\end{align*}
Thus, for $k,l=1,\dots,N-1$,
\begin{align*}
\pd_k w_\pm^1 
&=
\sum_{j=1}^{N-1}\int_0^\infty \CF_{\xi'}^{-1}\left[\frac{\xi_k\xi_j}{|\xi'|^2}|\xi'|e^{-|\xi'|(\pm x_N+y_N)}\wht{\pd_j g_1}(\xi',\pm y_N)\right](x') \intd y_N  \\
&\mp
\int_0^\infty \CF_{\xi'}^{-1}\left[\frac{i\xi_k}{|\xi'|}|\xi'|e^{-|\xi'|(\pm x_N+y_N)}\wht {\pd_N g_1}(\xi',\pm y_N)\right](x')\intd y_N,  \\
\pd_N w_\pm^1 
&=\pm
\sum_{j=1}^{N-1}\int_0^\infty \CF_{\xi'}^{-1}\left[\frac{i\xi_j}{|\xi'|}|\xi'|e^{-|\xi'|(\pm x_N+y_N)}\wht{\pd_j g_1}(\xi',\pm y_N)\right](x') \intd y_N  \\
&+ 
\int_0^\infty \CF_{\xi'}^{-1}\left[|\xi'|e^{-|\xi'|(\pm x_N+y_N)}\wht {\pd_N g_1}(\xi',\pm y_N)\right](x')\intd y_N, \\
\pd_l\pd_k w_\pm^1 
&=
\sum_{j=1}^{N-1}\int_0^\infty \CF_{\xi'}^{-1}\left[\frac{\xi_k\xi_j}{|\xi'|^2}|\xi'|e^{-|\xi'|(\pm x_N+y_N)}\wht{\pd_l\pd_j g_1}(\xi',\pm y_N)\right](x') \intd y_N  \\
&\mp
\int_0^\infty \CF_{\xi'}^{-1}\left[\frac{i\xi_k}{|\xi'|}|\xi'|e^{-|\xi'|(\pm x_N+y_N)}\wht {\pd_l\pd_N g_1}(\xi',\pm y_N)\right](x')\intd y_N,  \\
\pd_k\pd_N w_\pm^1 &=\pd_N\pd_k w_\pm^1\\
&=
\pm\sum_{j=1}^{N-1}\int_0^\infty \CF_{\xi'}^{-1}\left[\frac{i\xi_j}{|\xi'|}|\xi'|e^{-|\xi'|(\pm x_N+y_N)}\wht{\pd_k\pd_j g_1}(\xi',\pm y_N)\right](x') \intd y_N  \\
&+ 
\int_0^\infty \CF_{\xi'}^{-1}\left[|\xi'|e^{-|\xi'|(\pm x_N+y_N)}\wht {\pd_k \pd_N g_1}(\xi',\pm y_N)\right](x')\intd y_N.
\end{align*}
Analogously, for $\Delta' g_1=\sum_{j=1}^{N-1}\pd_j^2 g_1$,
\begin{align*}
\pd_N^2 w_\pm^1 
&=-
\int_0^\infty \CF_{\xi'}^{-1}\left[|\xi'|e^{-|\xi'|(\pm x_N+y_N)}\wht{\Delta' g_1}(\xi',\pm y_N)\right](x') \intd y_N  \\
&\pm\sum_{j=1}^{N-1}
\int_0^\infty \CF_{\xi'}^{-1}\left[\frac{i\xi_j}{|\xi'|}|\xi'|e^{-|\xi'|(\pm x_N+y_N)}\wht {\pd_j\pd_N g_1}(\xi',\pm y_N)\right](x')\intd y_N.
\end{align*}
It then follows from Lemmas \ref{lemm:symbol1} and \ref{lemm:tech-2} that
\begin{align}\label{eq6-5:whole}
\|\nabla w_\pm^1\|_{L_q(\BR_\pm^N)}
&\leq C\sum_{\Fs\in\{+,-\}}\|\nabla U_{\Fs,0}\|_{L_q(\BR_\Fs^N)}, \\ 
\|\nabla^2 w_\pm^1\|_{L_q(\BR_\pm^N)}
&\leq C\sum_{\Fs\in\{+,-\}}\|\nabla^2 U_{\Fs,0}\|_{L_q(\BR_\pm^N)}. \notag
\end{align}
We similarly see that
\begin{equation}\label{eq9:whole}
\|\nabla^2 w_\pm^2\|_{L_q(\BR_\pm^N)}
\leq C\sum_{\Fs\in\{+,-\}}\left(\|\nabla\pd_N U_{\Fs,0}\|_{L_q(\BR_\Fs^N)}+\|\nabla f_{\Fs N}\|_{L_q(\BR_\Fs^N)}\right)
\end{equation}
and that, if $f_{+N}-f_{-N}=0$ on $\BR_0^N$,
\begin{equation}\label{eq9:whole-2}
\|\nabla^2 w_\pm^2\|_{L_q(\BR_\pm^N)}
\leq C\sum_{\Fs\in\{+,-\}}\|\nabla\pd_N U_{\Fs,0}\|_{L_q(\BR_\Fs^N)}.
\end{equation}

Next, we estimate $\|\nabla w_\pm^2\|_{L_q(\BR_\pm^N)}$. 
By Lemma \ref{lemm:whole-zero}, 
\begin{equation*}
\wht g_2(0)
=\sum_{\Fs\in\{+,-\}}\int_0^\infty e^{-|\xi'|y_N}
\left(i\xi'\cdot \wht\Bf_\Fs'(\xi',\Fs y_N)+\Fs|\xi'|\wht f_{\Fs N}(\xi',\Fs y_N)\right)\intd y_N,
\end{equation*}
which implies that
\begin{align*}
w_\pm^2 &=\sum_{\Fs\in\{+,-\}}\bigg\{ 
\sum_{j=1}^{N-1}
\int_0^\infty\CF_{\xi'}^{-1}\bigg[\frac{i\xi_j}{|\xi'|} e^{-|\xi'|(\pm x_N+y_N)}\wht f_{\Fs j}(\xi',\Fs y_N)\bigg](x')\intd y_N \\
&+\Fs \int_0^\infty\CF_{\xi'}^{-1}\left[e^{-|\xi'|(\pm x_N+y_N)}\wht f_{\Fs N}(\xi',\Fs y_N)\right](x')\intd y_N\bigg\}.
\end{align*}
Then, for $k=1,\dots,N-1$,
\begin{align*}
\pd_k w_\pm^2 
 &=\sum_{\Fs\in\{+,-\}}\bigg\{  
-\sum_{j=1}^{N-1}
\int_0^\infty\CF_{\xi'}^{-1}\bigg[\frac{\xi_k\xi_j}{|\xi'|^2} |\xi'|e^{-|\xi'|(\pm x_N+y_N)}\wht f_{\Fs j}(\xi',\Fs y_N)\bigg](x')\intd y_N \notag \\
&+\Fs \int_0^\infty\CF_{\xi'}^{-1}\left[\frac{i\xi_k}{|\xi'|} |\xi'|e^{-|\xi'|(\pm x_N+y_N)}\wht f_{\Fs N}(\xi',\Fs y_N)\right](x')\intd y_N\bigg\},  \\
\pd_N w_\pm^2 
&=\sum_{\Fs\in\{+,-\}}\bigg\{ 
\mp \sum_{j=1}^{N-1}
\int_0^\infty\CF_{\xi'}^{-1}\bigg[\frac{i\xi_j}{|\xi'|} |\xi'|e^{-|\xi'|(\pm x_N+y_N)}\wht f_{\Fs j}(\xi',\Fs y_N)\bigg](x')\intd y_N \notag \\
&\mp \Fs \int_0^\infty\CF_{\xi'}^{-1}\left[|\xi'|e^{-|\xi'|(\pm x_N+y_N)}\wht f_{\Fs N}(\xi',\Fs y_N)\right](x')\intd y_N\bigg\}. \notag
\end{align*}
Combining these formulas with Lemmas \ref{lemm:symbol1} and \ref{lemm:tech-2} furnishes
\begin{equation*}
\|\nabla w_\pm^2\|_{L_q(\BR_\pm^N)}\leq C\left(\|\Bf_+\|_{L_q(\BR_+^N)}+\|\Bf_-\|_{L_q(\BR_-^N)}\right).
\end{equation*}
Recalling $v_\pm=U_{\pm,0}+w_\pm$ and \eqref{0818:9_2019},
we obtain solutions $v_\pm$ of \eqref{s-eq:2-bdd_2} satisfying \eqref{0819:1_2019}-\eqref{0824:1_2019}
by the last estimate, \eqref{eq6-5:whole}-\eqref{eq9:whole-2}, and Lemma \ref{lemm:whole-zero}.

Finally, we prove the uniqueness of solutions to \eqref{s-eq:2-bdd_2}.
Let $v_\pm$ satisfy \eqref{unique:flat} and $\Bf_\pm\in C_0^\infty(\BR_\pm^N)^N$.
By the discussion above, 
there exists 
$w_\pm\in \wht H_{q'}^1(\BR_\pm^N)\cap \wht H_{q'}^2(\BR_\pm^N)$ with $q'=q/(q-1)$ satisfying
\begin{equation*}
\left\{\begin{aligned}
\Delta w_\pm &= \dv\Bf_\pm && \text{in $\BR_\pm^N$,} \\
\rho_+ w_+ &= \rho_- w_- && \text{on $\BR_0^N$,} \\
\pd_N w_+&=\pd_N w_- && \text{on $\BR_0^N$.}
\end{aligned}\right.
\end{equation*}
At this point, we introduce the following lemma (cf. e.g. \cite[Section II.6]{Galdi11}, \cite{Shibata-book}, \cite[Proof of Theorem A.4]{Shi2013}).

\begin{lemm}\label{lemm:Hardy}
Let $q\in(1,\infty)$ and set
\begin{equation*}
d_q(x)=
\left\{\begin{aligned}
&(1+|x|^2)^{1/2} && \text{when $q\neq N$,} \\
&(1+|x|^2)^{1/2}\log(2+|x|^2)^{1/2}  && \text{when $q=N$.}
\end{aligned}\right.
\end{equation*}
Then, for any $u\in \wht H_q^1(\BR^N)$, there exists a constant $c_u$ such that
\begin{equation*}
\left\|\frac{u-c_u}{d_q}\right\|_{L_q(\BR^N)} \leq C\|\nabla u\|_{L_q(\BR^N)},
\end{equation*}
where $C=C(N,q)$ is a positive constant independent of $u$ and $c_u$.
\end{lemm}

Now we set 
\begin{equation*}
\Bf=\Bf_+\mathds{1}_{\BR_+^N}+\Bf_-\mathds{1}_{\BR_-^N}, \quad 
z=\rho_+ z_+\mathds{1}_{\BR_+^N}+\rho_-z_-\mathds{1}_{\BR_-^N} \text{ for $z\in\{v,w\}$}.
\end{equation*}
Then $v\in \wht H_q^1(\BR^N)$ with 
$\nabla v=\rho_+ (\nabla v_+)\mathds{1}_{\BR_+^N}+\rho_-(\nabla v_-)\mathds{1}_{\BR_-^N}$
by $\rho_+ v_+=\rho_- v_-$ on $\BR_0^N$,
while $w\in \wht H_{q'}^1(\BR^N)$ with 
$\nabla w=\rho_+ (\nabla w_+)\mathds{1}_{\BR_+^N}+\rho_-(\nabla w_-)\mathds{1}_{\BR_-^N}$
by $\rho_+ w_+=\rho_-w_-$ on $\BR_0^N$.
By Lemma \ref{lemm:Hardy}, there exist constants $c_v$ and $c_w$ such that
\begin{equation*}
\left\|\frac{v- c_v}{d_q}\right\|_{L_q(\BR^N)}\leq C\|\nabla v\|_{L_q(\BR^N)}, \quad 
\left\|\frac{w- c_w}{d_{q'}}\right\|_{L_{q'}(\BR^N)}\leq C\|\nabla w\|_{L_{q'}(\BR^N)}.
\end{equation*}

Let $\psi_R$ be a cut-off function of Sobolev's type as follows\footnote{
See e.g. \cite[Proof of Theorem II.7.1]{Galdi11}.}:
For $R>0$ large enough and for $\psi\in C_0^\infty(\BR)$ satisfying $\psi(t)=1$ when $|t|\leq 1/2$ and $\psi(t)=0$ when $|t|\geq 1$,
\begin{equation*}
\psi_R(x)=
\left\{\begin{aligned}
&1 && \text{when }|x|\leq e^{\sqrt[3]{\log{R}}}, \\
&\psi\left(\frac{\log\log|x|}{\log\log R}\right) && \text{when } |x|\geq e^{\sqrt[3]{\log{R}}}.
\end{aligned}\right.
\end{equation*}
Note that for $D_R=\{x\in\BR^N: e^{\sqrt{\log R}}\leq |x|\leq R\}$
\begin{equation}\label{cut-off:1}
|\nabla\psi_R(x)|\leq \frac{C}{\log\log R}\frac{1}{|x|\log|x|}, \quad 
\spp\nabla \vph_R \subset D_R,
\end{equation}
where $C$ is a positive constant independent of $R$.

Choose $R>0$ large enough so that $\psi_R=1$ on $\spp\Bf$.
It then holds that
\begin{align*}
&(\nabla v,\Bf)_{\BR^N}=(\nabla(v-c_v),\Bf)_{\BR^N}=-(v-c_v,\dv\Bf)_{\BR^N}=-(v-c_v,\psi_R\dv\Bf)_{\BR^N} \\
&=-(v-c_v,\psi_R\Delta w_+)_{\BR_+^N}-(v-c_v,\psi_R\Delta w_-)_{\BR_-^N} \\
&=-(v-c_v,\dv(\psi_R\nabla w_+)-\nabla\psi_R\cdot\nabla w_+)_{\BR_+^N} \\
&-(v-c_v,\dv(\psi_R\nabla w_-)-\nabla\psi_R\cdot\nabla w_-)_{\BR_-^N} \\
&=(\rho_+\nabla v_+,\psi_R\nabla w_+)_{\BR_+^N}+(v-c_v,\nabla\psi_R\cdot\nabla w_+)_{\BR_+^N} \\
&+(\rho_-\nabla v_-,\psi_R\nabla w_-)_{\BR_-^N}+(v-c_v,\nabla\psi_R\cdot\nabla w_-)_{\BR_-^N}. 
\end{align*}
On the other hand, 
\begin{align*}
&(\rho_+\nabla v_+,\psi_R\nabla w_+)_{\BR_+^N}+(\rho_-\nabla v_-,\psi_R\nabla w_-)_{\BR_-^N} \\
&=(\psi_R\nabla v_+,\nabla(\rho_+w_+-c_w))_{\BR_+^N}+(\psi_R\nabla v_-,\nabla(\rho_-w_--c_w))_{\BR_-^N} \\
&=-(\dv(\psi_R\nabla v_+),\rho_+w_+-c_w)_{\BR_+^N}-(\dv(\psi_R\nabla v_-),\rho_-w_--c_w)_{\BR_-^N} \\
&=-(\nabla\psi_R\cdot \nabla v_+,w-c_w)_{\BR_+^N}-(\nabla\psi_R\cdot\nabla v_-,w-c_w)_{\BR_-^N},
\end{align*}
and thus
\begin{align}\label{0821:1_2019}
(\nabla v,\Bf)_{\BR^N}&=
(v-c_v,\nabla\psi_R\cdot\nabla w_+)_{\BR_+^N}+(v-c_v,\nabla\psi_R\cdot\nabla w_-)_{\BR_-^N} \\
&-(\nabla\psi_R\cdot \nabla v_+,w-c_w)_{\BR_+^N}-(\nabla\psi_R\cdot\nabla v_-,w-c_w)_{\BR_-^N}. \notag
\end{align}
We here see that by Lemma \ref{lemm:Hardy} and \eqref{cut-off:1}
\begin{align*}
&|(v-c_v,\nabla\psi_R\cdot\nabla w_\pm)_{\BR_\pm^N}| \\
&\leq \|\frac{v-c_v}{d_q}\|_{L_q(\BR^N)}\|d_{q}\nabla\psi_R\|_{L_\infty(\BR^N)}\|\nabla w_\pm\|_{L_{q'}(\BR_\pm^N)} \\
&\leq \frac{C}{\log\log R}\|\nabla v\|_{L_q(\BR^N)}\|\nabla w_\pm\|_{L_{q'}(\BR_\pm^N)} \to 0 \quad \text{as $R\to \infty$,}
\end{align*}
and also
\begin{equation*}
\lim_{R\to \infty} (\nabla\psi_R\cdot \nabla v_\pm,w-c_w)_{\BR_\pm^N}=0.
\end{equation*}
By these properties, letting $R\to\infty$ in \eqref{0821:1_2019} furnishes
\begin{equation*}
(\rho_+\nabla v_+,\Bf_+)_{\BR_+^N}+(\rho_-\nabla v_-,\Bf_-)_{\BR_-^N}=(\nabla v,\Bf)_{\BR^N}=0,
\end{equation*}
which implies $\rho_\pm v_\pm= c_\pm$ for some constants $c_\pm$.
Since $\rho_+ v_+=\rho_- v_-$ on $\BR_0^N$, we have $c_+=c_-$.
Thus setting $c=c_+=c_-$ yields $v_\pm=\rho_\pm^{-1}c$, which implies the uniqueness of solutions to \eqref{s-eq:2-bdd_2}.
This completes the proof of Theorem \ref{theo:2-bdd_2}.

%%%%%%%%%%%%%%%%%%%%%%%%%%%%%%%%%%%%%%%%%%%%%%%%%%
\subsection{Weak elliptic problem with flat interface}\label{subsec:whole1-weak}
For $\Bf\in L_q(\BR^N\setminus\BR_0^N)^N$ with $q\in(1,\infty)$,
this subsection considers the unique solvability of the following weak elliptic problem:
Find $u\in\wht H_q^1(\BR^N)$ such that
\begin{equation}\label{w-eq:whole}
(\rho^{-1}\nabla u,\nabla\vph)_{\BR^N\setminus\BR_0^N}=(\Bf,\nabla\vph)_{\BR^N\setminus\BR_0^N}
\quad \text{for any }\vph\in \wht H_{q'}^1(\BR^N),
\end{equation}
where $q'=q/(q-1)$ and $\rho=\rho_+\mathds{1}_{\BR_+^N}+\rho_-\mathds{1}_{\BR_-^N}$ for positive constants $\rho_\pm$.
More precisely, we prove

\begin{theo}\label{w-theo:whole}
Let $q\in(1,\infty)$ and $q'=q/(q-1)$.  
\begin{enumerate}[$(1)$]
\item {\bf Existence}.
Let $\Bf\in L_q(\BR^N\setminus\BR_0^N)^N$.
Then the weak elliptic problem \eqref{w-eq:whole} admits a solution $u\in\wht H_q^1(\BR^N)$ satisfying
\begin{equation}\label{0919:1_2019}
\|\nabla u\|_{L_q(\BR^N)} \leq 2C_2(\rho_++\rho_-)\|\Bf\|_{L_q(\BR^N\setminus\BR_0^N)},
\end{equation}
where $C_2$ is the same constant as in Theorem $\ref{theo:2-bdd_2}$.
\item {\bf Uniqueness}.
If $u\in \wht H_q^1(\BR^N)$ satisfies
\begin{equation}\label{hom-eq:whole}
(\rho^{-1}\nabla u,\nabla\vph)_{\BR^N\setminus\BR_0^N}=0 \quad \text{for any $\vph\in\wht H_{q'}^1(\BR^N)$,}
\end{equation}
then $u=c$ for some constant $c$.
\end{enumerate}
\end{theo}

\begin{proof}
(1). Since $C_0^\infty(\BR^N)$ is dense in $L_q(\BR^N\setminus\BR_0^N)$,
it suffices to consider $\Bf\in C_0^\infty(\BR^N)^N$.
By Theorem \ref{theo:2-bdd_2},
we have $v_\pm\in \wht H_q^1(\BR_\pm^N)\cap \wht H_q^2(\BR_\pm^N)$ satisfying \eqref{s-eq:2-bdd_2}
and \eqref{0819:2_2019} with $\Bf_+=\Bf$ and $\Bf_-=\Bf$,
where we note that  $f_{+N}-f_{-N}=0$ on $\BR_0^N$ in this case. 
Setting $u_\pm=\rho_\pm v_\pm$ and
$u=u_+\mathds{1}_{\BR_+^N}+u_-\mathds{1}_{\BR_-^N}$ furnishes that
\begin{equation*}
\left\{
\begin{aligned}
\rho_\pm^{-1}\Delta u_\pm&=\dv\Bf_\pm && \text{in $\BR_\pm^N$,} \\
u_+&=u_- && \text{on $\BR_0^N$,} \\
\rho_+^{-1}\pd_N u_+&=\rho_-^{-1}\pd_N u_- && \text{on $\BR_0^N$,}
\end{aligned}
\right.
\end{equation*}
and that $u\in\wht H_q^1(\BR^N)$ with $\nabla u=\rho_+(\nabla v_+)\mathds{1}_{\BR_+^N}+\rho_-(\nabla v_-)\mathds{1}_{\BR_-^N}$
by $\rho_+v_+=\rho_-v_-$ on $\BR_0^N$,
Combining the last property with \eqref{0819:2_2019} yields
\begin{equation*}
\|\nabla u\|_{L_q(\BR^N)}
\leq C_2(\rho_++\rho_-)\sum_{\Fs\in\{+,-\}}\|\Bf_\Fs\|_{L_q(\BR_\Fs^N)}
\leq 2C_2(\rho_++\rho_-)\|\Bf\|_{L_q(\BR^N\setminus\BR_0^N)},
\end{equation*}
which implies that $u$ satisfies \eqref{0919:1_2019}.

Next, we prove that $u$ is a solution to \eqref{w-eq:whole}.
Let $\vph\in \wht H_{q'}^1(\BR^N)$.
Then there exists $\{\vph_j\}_{j=1}^\infty\subset C_0^\infty(\BR^N)$ such that
$\lim_{j\to \infty}\|\nabla(\vph_j-\vph)\|_{L_{q'}(\BR^N)}=0$. Thus,
\begin{align*}
(\Bf,\nabla\vph)_{\BR^N\setminus\BR_0^N}
&=\lim_{j\to\infty}(\Bf,\nabla\vph_j)_{\BR^N\setminus\BR_0^N}
=-\lim_{j\to\infty}\left\{(\dv\Bf_+,\vph_j)_{\BR_+^N}+(\dv\Bf_-,\vph_j)_{\BR_-^N}\right\} \\
&=-\lim_{j\to\infty}\left\{(\rho_+^{-1}\Delta u_+,\vph_j)_{\BR_+^N}+(\rho_-^{-1}\Delta u_-,\vph_j)_{\BR_+^N}\right\} \\
&=\lim_{j\to\infty}\left\{(\rho_+^{-1}\nabla u_+,\nabla\vph_j)_{\BR_+^N}+(\rho_-^{-1}\nabla u_-,\nabla\vph_j)_{\BR_+^N}\right\} \\
&=\lim_{j\to\infty}(\rho^{-1}\nabla u,\nabla\vph_j)_{\BR^N\setminus\BR_0^N}=(\rho^{-1}\nabla u,\nabla\vph)_{\BR^N\setminus\BR_0^N}.
\end{align*}
This completes the proof of (1).

(2).
Let $u\in\wht H_q^1(\BR^N)$ satisfy \eqref{hom-eq:whole}, and let $\Bf\in C_0^\infty(\BR^N)^N$.
By (1), we have $v\in \wht H_{q'}^1(\BR^N)$ satisfying
\begin{equation*}
(\rho^{-1}\nabla v,\nabla\psi)_{\BR^N\setminus\BR_0^N}
=(\Bf,\nabla\psi)_{\BR^N\setminus\BR_0^N} \quad \text{for any $\psi\in \wht H_q^1(\BR^N)$.}
\end{equation*}
Choosing $\psi= u$ in this equality yields
\begin{align*}
(\Bf,\nabla u)_{\BR^N}=(\Bf,\nabla u)_{\BR^N\setminus\BR_0^N}=(\rho^{-1}\nabla v,\nabla u)_{\BR^N\setminus\BR_0^N}
=(\rho^{-1}\nabla u,\nabla v)_{\BR^N\setminus\BR_0^N}=0,
\end{align*}
which implies $u=c$ for some constant $c$. 
This completes the proof of (2).
\end{proof}

%%%%%%%%%%%%%%%%%%%%%%%%%%%%%%%%%%%%%%%%%%%%%%%%%%%%%%%%%%%%%%%%%%%%%%
\section{Problems in the whole space with bent interface}\label{sec:whole2}

Let $\Phi:\BR_x^N\to\BR_y^N$ be a bijection of class $C^1$
and $\Phi^{-1}$ the inverse mapping of $\Phi$,
where the subscripts $x$, $y$ denote their variables.
Let $(\nabla_x\Phi)(x)=\CA+\BB(x)$ and 
$(\nabla_y \Phi^{-1})(\Phi(x))=\CA_{-1}+\BB_{-1}(x)$.
Assume that $\CA$, $\CA_{-1}$ are orthonormal matrices with constant coefficients
and $\det\CA=\det\CA_{-1}=1$,
and also assume that $\BB(x)$, $\BB_{-1}(x)$ are matrix-valued functions of $H_r^1(\BR^N)$, $r\in(N,\infty)$,
satisfying
\begin{align}\label{bent:condi1}
\|(\BB,\BB_{-1})\|_{L_\infty(\BR^N)}&\leq M_1 \quad  (0<M_1 \leq 1/2), \\
\|(\nabla \BB,\nabla\BB_{-1})\|_{L_r(\BR^N)} &\leq M_2 \quad  (M_2\geq  1). \notag
\end{align}
In what follows, we will choose $M_1$ small enough eventually.

Let $\wtd \Omega_\pm=\Phi(\BR_\pm^N)$,  $\wtd\Sigma=\Phi(\BR_0^N)$,
and $\wtd\Bn=\wtd\Bn(y)$ be the unit normal vector on $\wtd\Sigma$
pointing from $\wtd\Omega_+$ into $\wtd\Omega_-$.
Let $\CA_{ij}$ and $B_{ij}(x)$ be respectively the $(i,j)$-component of $\CA_{-1}$
and the $(i,j)$-component of $\BB_{-1}(x)$.
In addition, setting $\Phi^{-1}=(\Phi_{-1,1},\dots,\Phi_{-1,N})$,
we see that $\wtd\Sigma$ is represented as $\Phi_{-1,N}(y)=0$.
This representation implies that
\begin{align}\label{normal:bent}
\wtd \Bn(\Phi(x))
&=-\frac{(\nabla_y \Phi_{-1,N})(\Phi(x))}{|(\nabla_y \Phi_{-1,N})(\Phi(x))|}
=-\frac{\left(\CA_{N1}+B_{N1}(x),\dots, \CA_{NN}+B_{NN}(x)\right)^\SST
}{\sqrt{\sum_{j=1}^N\left(\CA_{Nj}+B_{Nj}(x)\right)^2}}
\\
&=\frac{(\CA_{-1}+\BB_{-1}(x))^\SST\Bn_0}{|(\CA_{-1}+\BB_{-1}(x))^\SST\Bn_0|} 
\notag
\end{align}
for $\Bn_0=(0,\dots,0,-1)^\SST$.
Especially, $\wtd \Bn$ is defined on $\BR^N$ through the relation \eqref{normal:bent}.

\begin{rema}
Let $J$ be the Jacobian of $\Phi$, i.e. $J=\det(\nabla\Phi)$,
and let $d=|(\CA_{-1}+\BB_{-1}(x))^\SST\Bn_0|$.
Since $M_1$ is small enough,
there are positive constants $C_3$ and $C_4$, independent of $M_1$, $M_2$, and $r$, such that
the following inequalities hold:
\begin{align}\label{rema:jacob}
&C_3 \leq  J(x), d(x) \leq  C_4 \quad (x\in\BR^N), \quad 
\sup_{x\in\BR^N}|1- J(x)|\leq C_4 M_1, \\
&\|\nabla J\|_{L_r(\BR^N)}\leq C_4 M_2, \quad \|\nabla d\|_{L_r(\BR^N)} \leq C_4 M_2. \notag
\end{align}
\end{rema}

This section mainly considers two strong elliptic problems as follows: 
\begin{equation}\label{s-eq:bent-1}
\left\{\begin{aligned}
\rho_\pm\lambda \wtd v_\pm-\Delta \wtd v_\pm 
&= - \dv\wtd\Bf_\pm+ \wtd g_\pm && \text{in $\wtd\Omega_\pm$,} \\
\rho_+\wtd v_+ &= \rho_-\wtd v_- && \text{on $\wtd\Sigma$,} \\
\wtd\Bn\cdot\nabla(\wtd v_+- \wtd v_-) 
&=\wtd\Bn\cdot(\wtd\Bf_+-\wtd \Bf_-)+\wtd h_+-\wtd h_- && \text{on $\wtd \Sigma$,}
\end{aligned}\right. 
\end{equation}
and also
\begin{equation}\label{s-eq:bent-2}
\left\{\begin{aligned}
\Delta \wtd v_\pm &= \dv\wtd\Bf_\pm  && \text{in $\wtd\Omega_\pm$,} \\
\rho_+\wtd v_+ &= \rho_-\wtd v_- && \text{on $\wtd\Sigma$,} \\
\wtd\Bn\cdot\nabla(\wtd v_+- \wtd v_-) 
&=\wtd\Bn\cdot(\wtd\Bf_+-\wtd \Bf_-) && \text{on $\wtd \Sigma$.}
\end{aligned}\right. 
\end{equation}
Assume that $\rho_\pm$ are positive constants throughout this section and
set $F_q(\wtd\Omega_\pm)$ as
\begin{equation*}
F_q(\wtd\Omega_\pm)=\{\wtd \Bf_\pm\in E_q(\wtd\Omega_\pm): \wtd\Bn\cdot\wtd \Bf_\pm\in H_q^1(\wtd\Omega_\pm)\}.
\end{equation*}
Concerning \eqref{s-eq:bent-1} and \eqref{s-eq:bent-2}, we prove the following theorems.

\begin{theo}\label{theo:bent}
Let $M_1$, $M_2$, $r$, $\wtd\Omega_\pm$, and $\wtd\Sigma$ be as above,
and let $\sigma\in(0,\pi/2)$ and $q\in(1,r]$. 
Then there exist $M_1\in(0,1/2)$ and $\lambda_1\geq 1$ such that, for any
$$
\wtd\Bf_\pm \in F_q(\wtd \Omega_\pm), \quad
\wtd g_\pm\in L_q(\wtd\Omega_\pm), \quad
\wtd h_\pm\in H_q^1(\wtd\Omega_\pm)
$$
and for any $\lambda\in\Sigma_{\sigma,\lambda_1}$,
the strong elliptic problem \eqref{s-eq:bent-1} admits unique solutions $\wtd v_\pm\in H_q^2(\wtd\Omega_\pm)$.
In addition, the solutions $\wtd v_\pm$ satisfies
\begin{align}\label{0903:11_2019}
&\|(\lambda\wtd v_\pm,\lambda^{1/2}\nabla \wtd v_\pm,\nabla^2 \wtd v_\pm)\|_{L_q(\wtd\Omega_\pm)} \\
&\leq C\sum_{\Fs\in\{+,-\}}\left\|\left(\dv\wtd\Bf_\Fs,\wtd g_\Fs,\lambda^{1/2}(\wtd\Bn\cdot\wtd\Bf_\Fs),
\nabla(\wtd\Bn\cdot\wtd\Bf_\Fs),\lambda^{1/2}\wtd h_\Fs,\nabla\wtd h_\Fs\right)\right\|_{L_q(\wtd\Omega_\Fs)}, 
\notag
\end{align}
and also
\begin{align}\label{0903:12_2019}
\|(\lambda^{1/2}\wtd v_\pm,\nabla \wtd v_\pm)\|_{L_q(\wtd\Omega_\pm)} 
&\leq C\sum_{\Fs\in\{+,-\}}\Big(\|\wtd\Bf_\Fs\|_{L_q(\wtd\Omega_\Fs)}+|\lambda|^{-1/2}\|\wtd g_\Fs\|_{L_q(\wtd\Omega_\Fs)} \\
&+\|\wtd h_\Fs\|_{L_q(\wtd\Omega_\Fs)}+|\lambda|^{-1/2}\|\nabla\wtd h_\Fs\|_{L_q(\wtd\Omega_\Fs)}\Big), \notag
\end{align}
where $C=C(M_2,r,N,q,\sigma,\rho_+,\rho_-)$ is a positive constant independent of $M_1$, $\lambda_1$, and $\lambda$.
Additionally, if $\wtd\Bn\cdot(\wtd\Bf_+-\wtd\Bf_-)=0$ on $\wtd\Sigma$, then $\wtd v_\pm$ satisfy
\begin{align}\label{0903:11_2019-2}
&\|(\lambda\wtd v_\pm,\lambda^{1/2}\nabla \wtd v_\pm,\nabla^2 \wtd v_\pm)\|_{L_q(\wtd\Omega_\pm)} \\
&\leq C\sum_{\Fs\in\{+,-\}}\left\|\left(\dv\wtd\Bf_\Fs,\wtd g_\Fs,
\lambda^{1/2}\wtd h_\Fs,\nabla\wtd h_\Fs\right)\right\|_{L_q(\wtd\Omega_\Fs)}. \notag
\end{align}
Here, $M_1$ depends only on $N$, $q$, $\sigma$, $\rho_+$, and $\rho_-$, while
$\lambda_1$ depends only on $M_2$, $r$, $N$, $q$, $\sigma$, $\rho_+$, and $\rho_-$.
\end{theo}

\begin{theo}\label{theo:bent2}
Let $M_1$, $M_2$, $r$, $\wtd\Omega_\pm$, and $\wtd\Sigma$ be as above,
and let $q\in(1,\infty)$.
Assume that $\max(q,q')\leq r$ for $q'=q/(q-1)$.
Then there exists $M_1\in(0,1/2)$, depending only on $N$, $q$, $\rho_+$, and $\rho_-$,
such that the following assertions hold.
\begin{enumerate}[$(1)$]
\item
{\bf Existence}. Let $\wtd\Bf_\pm \in F_q(\wtd \Omega_\pm)$.
Then the strong elliptic problem \eqref{s-eq:bent-2} admits solutions
$\wtd v_\pm\in \wht H_q^1(\wtd\Omega_\pm)\cap \wht H_q^2(\wtd\Omega_\pm)$ satisfying
\begin{align}
\|\nabla^2 \wtd v_\pm\|_{L_q(\wtd\Omega_\pm)}
&\leq C\sum_{\Fs\in\{+,-\}}\left(\|\wtd\Bf_\Fs\|_{E_q(\wtd\Omega_\Fs)}
+\|\wtd\Bn\cdot\wtd\Bf_\Fs\|_{H_q^1(\wtd\Omega_\Fs)}\right), \label{0920:1_2019} \\
\|\nabla \wtd v_\pm\|_{L_q(\wtd\Omega_\pm)}
&\leq C'\sum_{\Fs\in\{+,-\}}\|\wtd \Bf_\Fs\|_{L_q(\wtd\Omega_\Fs)}, \label{0920:2_2019}
\end{align}
with positive constants $C=C(M_2,r,N,q,\rho_+,\rho_-)$ and $C'=C'(N,q,\rho_+,\rho_-)$ independent of $M_1$.
Additionally, if $\wtd\Bn\cdot(\wtd\Bf_+-\wtd\Bf_-)=0$ on $\wtd\Sigma$,
then $\wtd v_\pm$ satisfy
\begin{equation}\label{0920:3_2019}
\|\nabla^2 \wtd v_\pm\|_{L_q(\wtd\Omega_\pm)}
\leq C\sum_{\Fs\in\{+,-\}}\|\wtd\Bf_\Fs\|_{E_q(\wtd\Omega_\Fs)}.
\end{equation}
\item
{\bf Uniqueness}.
If $\wtd v_\pm\in \wht H_q^1(\wtd\Omega_\pm)\cap \wht H_q^2(\wtd\Omega_\pm)$ satisfy
\begin{equation*}
\Delta \wtd v_\pm=0 \text{ in $\wtd\Omega_\pm$,} \quad
\rho_+ \wtd v_+= \rho_-\wtd v_- \text{\  and \  } \wtd\Bn\cdot \nabla(\wtd v_+-\wtd v_-)=0 \text{ on $\wtd\Sigma$},
\end{equation*}
then $\wtd v_\pm=\rho_\pm^{-1}c$ for some constant $c$.
\end{enumerate}
\end{theo}

At this point, we introduce the following lemma, see \cite[Lemma 2.4]{Shi2013}, which play
an important role in proving Theorems \ref{theo:bent} and \ref{theo:bent2} in the following subsections.

\begin{lemm}\label{lemm:bent1}
Let  $q\in(1,\infty)$ and $r\in(N,\infty)$. Assume that $q\leq r$.
Then there exists a positive constant $C=C(N,q,r)$ such that,
for any $\vps>0$, $a\in L_r(\BR_+^N)$, and $b\in H_q^1(\BR_+^N)$,
$$
\|ab\|_{L_q(\BR_+^N)}\leq \vps\|\nabla b\|_{L_q(\BR_+^N)}+C\vps^{-\frac{N}{r-N}}\|a\|_{L_r(\BR_+^N)}^\frac{r}{r-N}\|b\|_{L_q(\BR_+^N)}.
$$
\end{lemm}

%%%%%%%%%%%%%%%%%%%%%%%%%%%%%%%%%%%%%%%%%%%%%%%%%%
\subsection{Proof of Theorem \ref{theo:bent}}\label{subset:1_bent}
To prove Theorem \ref{theo:bent},
we first reduce \eqref{s-eq:bent-1} to a problem in the whole space with the flat interface
by the change of variables: $y=\Phi(x)$.
Let $D_j=\pd/\pd y_j$ and $\pd_j=\pd/\pd x_j$.
One then notes the following fundamental relations: For $j,k=1,\dots,N$,
\begin{align*}
D_j&=\sum_{l=1}^N(\CA_{lj}+B_{lj}(x))\pd_l, \quad
\nabla_y=(\CA_{-1}+\BB_{-1}(x))^\SST \nabla_x, \\
D_jD_k
&=\sum_{l,m=1}^N \CA_{lj}\CA_{mk}\pd_l\pd_m \\
&+\sum_{l,m=1}^N\left(\CA_{lj}B_{mk}(x)+\CA_{mk}B_{lj}(x)+B_{lj}(x)B_{mk}(x)\right)\pd_l\pd_m \\
&+\sum_{l,m=1}^N(\CA_{lj}+B_{lj}(x))(\pd_l B_{mk})(x)\pd_m.
\end{align*}

Let $\wtd\Bu$ be a vector function on $\BR_y^N\setminus\wtd\Sigma=\wtd\Omega_+\cup\wtd\Omega_-$,
and let $\Bu=\wtd\Bu\circ\Phi$.
Then, for any $\wtd\vph\in C_0^\infty(\BR_y^N\setminus\wtd\Sigma)$ and for $\vph=\wtd\vph\circ\Phi$,
\begin{align*}
(\dv_y\wtd\Bu,\wtd\vph)_{\BR_y^N\setminus\wtd\Sigma} 
&=-(\wtd\Bu,\nabla_y\wtd\vph)_{\BR_y^N\setminus\wtd\Sigma} 
=-(J\Bu,(\CA_{-1}+\BB_{-1})^\SST\nabla_x\vph)_{\BR^N\setminus\BR_0^N} \\
&=(\dv_x\{J(\CA_{-1}+\BB_{-1})\Bu\},\vph)_{\BR^N\setminus\BR_0^N} \\
&=([J^{-1}\dv_x\{J(\CA_{-1}+\BB_{-1})\Bu\}]\circ\Phi^{-1},\wtd\vph)_{\BR_y^N\setminus\Sigma},
\end{align*}
which implies
\begin{equation}\label{change:div}
(\dv_y\wtd\Bu)\circ\Phi=J^{-1}\dv_x\{J(\CA_{-1}+\BB_{-1})\Bu\}.
\end{equation}
Analogously, we observe for a scalar function $\wtd u$ on $\BR_y^N\setminus\wtd\Sigma$
and $u=\wtd u\circ \Phi$ that
\begin{align*}
(\Delta_y \wtd u,\wtd\vph)_{\BR_y^N\setminus\wtd\Sigma}
&=(\nabla_y\wtd u,\nabla_y\wtd\vph)_{\BR_y^N\setminus\wtd\Sigma} \\
&=(J(\CA_{-1}+\BB_{-1})^\SST\nabla_x u,(\CA_{-1}+\BB_{-1})^\SST\nabla_x\vph)_{\BR_x^N\setminus\BR_0^N} \\
&=(\dv_x\{J(\CA_{-1}+\BB_{-1})(\CA_{-1}+\BB_{-1})^\SST\nabla_x u\},\vph)_{\BR_x^N\setminus\BR_0^N},
\end{align*}
and thus
\begin{equation}\label{change:lap}
(\Delta_y \wtd u)\circ\Phi=
J^{-1}\dv_x\{J(\CA_{-1}+\BB_{-1})(\CA_{-1}+\BB_{-1})^\SST\nabla_x u\}.
\end{equation}
For simplicity, one sets
$$
\BC_{-1}(x)=\CA_{-1}\BB_{-1}(x)^\SST+\BB_{-1}(x)\CA_{-1}^\SST+\BB_{-1}(x)\BB_{-1}(x)^\SST,
$$
and then
\begin{equation}\label{0824:11_2019}
(\CA_{-1}+\BB_{-1}(x))(\CA_{-1}+\BB_{-1}(x))^\SST=\BI+\BC_{-1}(x).
\end{equation}
Furthermore, let us define 
$$
v_\pm=\wtd v_\pm\circ \Phi, \quad \Bf_\pm=\wtd\Bf_\pm\circ\Phi, \quad
g_\pm= J (\wtd g_\pm\circ\Phi), \quad h_\pm= J d(\wtd h_\pm\circ\Phi).
$$
Then, by \eqref{normal:bent} and \eqref{change:div}-\eqref{0824:11_2019},
we obtain an equivalent system of \eqref{s-eq:bent-1} as follows:
\begin{equation}\label{s-eq:bent-prtb}
\left\{\begin{aligned}
\rho_\pm\lambda v_\pm-\Delta v_\pm
&=-\dv(\BF_\pm+\CF(v_\pm))+\rho_\pm\CG_{\lambda}(v_\pm)+g_\pm
&& \text{in $\BR_\pm^N$,} \\
\rho_+ v_+&=\rho_- v_- && \text{on $\BR_0^N$,} \\
\Bn_0\cdot\nabla  (v_+- v_-)
&=\Bn_0\cdot\{(\BF_++\CF_+(v_+))-(\BF_- +\CF_-(v_-))\} \\
&+h_+-h_-
&& \text{on $\BR_0^N$,}
\end{aligned}\right.
\end{equation}
where we have set 
\begin{align}\label{0903:3_2019}
&\BF_\pm=J(\CA_{-1}+\BB_{-1})\Bf_\pm, \quad 
\CF(v_\pm)
= (1-J)\nabla v_\pm -J\BC_{-1}\nabla v_\pm, \\
&\CG_\lambda(v_\pm)
=\lambda(1-J)v_\pm. \notag
\end{align}

\begin{rema}\label{rema:linear}
\begin{enumerate}[(1)]
\item\label{rema:linear-1}
The above $\CF$ and $\CG_\lambda$ are linear mappings. 
\item
By \eqref{bent:condi1}, there is a constant $C_5>0$, independent of $M_1$, $M_2$, and $r$, such that
\begin{equation}\label{0904:1_2019}
\|\BC_{-1}\|_{L_\infty(\BR^N)}\leq C_5 M_1, \quad \|\nabla\BC_{-1}\|_{L_r(\BR^N)} \leq C_5 M_2.
\end{equation}
\end{enumerate}
\end{rema}

In what follows, $F_{\pm N}$ stands for the $N$th component of $\BF_\pm$,
while $\CF_N(v_\pm)$ the $N$th component of $\CF(v_\pm)$.
By \eqref{rema:jacob} and \eqref{change:div},
\begin{equation}\label{0824:12_2019}
\|\dv\BF_\pm\|_{L_q(\BR_\pm^N)} \leq C_4^{1-1/q}\|\dv\wtd\Bf_\pm\|_{L_q(\wtd\Omega_\pm)}.
\end{equation}
In addition, since it holds by \eqref{normal:bent} that
\begin{equation}\label{0920:4_2019}
F_{\pm N}=-\Bn_0\cdot \{J(\CA_{-1}+\BB_{-1})\Bf_\pm\}=-J d(\wtd\Bn\cdot\wtd\Bf_\pm),
\end{equation}
one has by \eqref{rema:jacob} and Lemma \ref{lemm:bent1} with $\vps=1$
\begin{equation}\label{0903:1_2019}
\|F_{\pm N}\|_{H_q^j(\BR_\pm^N)}\leq C_{M_2,r}\|\wtd\Bn\cdot\wtd\Bf_\pm\|_{H_q^j(\wtd\Omega_{\pm})}
\quad (j=0,1).
\end{equation}
Here and subsequently, $C_{M_2,r}$ stands for generic positive constants depending on $M_2$, $N$, $q$, and $r$,
but independent of $M_1$.
Similarly to \eqref{0824:12_2019} and \eqref{0903:1_2019}, 
\begin{align}\label{0903:15_2019}
\|\BF_\pm\|_{L_q(\BR_\pm^N)}&\leq C_6\|\wtd\Bf_\pm\|_{L_q(\wtd\Omega_\pm)}, \quad 
\|g_\pm\|_{L_q(\BR_\pm^N)} \leq C_6\|\wtd g_\pm\|_{L_q(\wtd\Omega_\pm)}, \\ 
\|h_\pm\|_{H_q^j(\BR_\pm^N)}&\leq C_{M_2,r}\|\wtd h_\pm\|_{H_q^j(\wtd\Omega_\pm)} \quad (j=0,1), \notag
\end{align}
where $C_6$ is a positive constant independent of $M_1$, $M_2$, and $r$.
Concerning $\CF(v_\pm)$ and $\CG_\lambda(v_\pm)$,
we have by Lemma \ref{lemm:bent1}, \eqref{rema:jacob}, and \eqref{0904:1_2019}

\begin{lemm}\label{lemm:bent-reminder}
Let $q\in(1,\infty)$, $r\in(N,\infty)$, and $\vps>0$.
Assume that $q\leq r$.
Then there exist 
a positive constant $\alpha_{M_2,r,\vps}$ depending on $M_2$, $r$, and $\vps$,
but independent of $M_1$,
and a positive constant $C_7$ independent of $M_1$, $M_2$, $r$, and $\vps$, such that
for any $w_\pm\in H_q^2(\BR_\pm^N)$ 
\begin{align*}
&\|(\dv\CF(w_\pm),\nabla \CF(w_\pm))\|_{L_q(\BR_\pm^N)} \\
&\leq C_7 (M_1+\vps)\|\nabla^2 w_\pm\|_{L_q(\BR_\pm^N)} 
+ \alpha_{M_2,r,\vps}\|\nabla w_\pm\|_{L_q(\BR_\pm^N)}, 
\end{align*}
and also
\begin{align*}
\|\CF(w_\pm)\|_{L_q(\BR_\pm^N)}
&\leq C_7 M_1\|\nabla w_\pm\|_{L_q(\BR_\pm^N)}, \notag \\
\|\rho_\pm\CG_\lambda(w_\pm)\|_{L_q(\BR_\pm^N)}
&\leq C_7 M_1|\lambda|\|w_\pm\|_{L_q(\BR_\pm^N)} \quad (\lambda\in\BC). \notag
\end{align*}
\end{lemm}

From now on, we solve \eqref{s-eq:bent-prtb} by using the contraction mapping principle
together with Theorem \ref{theo:2-bdd}.
Let $\lambda_1\geq 1$ and $\lambda\in\Sigma_{\sigma,\lambda_1}$
for some $\sigma\in(0,\pi/2)$. 
We set $X_{q,\lambda}=H_q^2(\BR_+^N)\times H_q^2(\BR_-^N)$ endowed with the norm:
\begin{equation*}
\|(w_+,w_-)\|_{X_{q,\lambda}}=\sum_{\Fs\in\{+,-\}}
\|(\lambda w_\Fs,\lambda^{1/2}\nabla w_\Fs,\nabla^2 w_\Fs)\|_{L_q(\BR_\Fs^N)},
\end{equation*}
and also for $R>0$
\begin{equation*}
X_{q,\lambda}(R)=\{(w_+,w_-)\in X_{q,\lambda} : \|(w_+,w_-)\|_{X_{q,\lambda}}\leq R\}.
\end{equation*}
The next lemma then follows from Lemma \ref{lemm:bent-reminder} immediately.

\begin{lemm}\label{lemm:bent-reminder-2}
Let $\lambda_1$ and $\lambda$ be as above.
Let $q\in(1,\infty)$, $r\in(N,\infty)$, and $\vps>0$.
Assume that $q\leq r$.
Then, for any $w_\pm\in H_q^2(\BR_\pm^N)$,
\begin{align*}
\sum_{\Fs\in\{+,-\}}\|\nabla \CF(w_\Fs)\|_{L_q(\BR_\Fs^N)}
&\leq \left\{C_7(M_1+\vps)+\alpha_{M_2,r,\vps}\lambda_1^{-1/2}\right\}\|(w_+,w_-)\|_{X_{q,\lambda}}, \\
\sum_{\Fs\in\{+,-\}}\|\lambda^{1/2} \CF(w_\Fs)\|_{L_q(\BR_\Fs^N)}
&\leq C_7 M_1\|(w_+,w_-)\|_{X_{q,\lambda}}, \\
\sum_{\Fs\in\{+,-\}}\|\rho_\Fs\CG_\lambda (w_\Fs)\|_{L_q(\BR_\Fs^N)}
&\leq C_7 M_1 \|(w_+,w_-)\|_{X_{q,\lambda}}, \notag 
\end{align*}
and also
\begin{align*}
&\sum_{\Fs\in\{+,-\}}\left\|\left(\CF(w_\Fs),|\lambda|^{-1/2}\rho_\Fs\CG_\lambda(w_\Fs)\right)\right\|_{L_q(\BR_\Fs^N)}  \\
&\leq C_7 M_1 \sum_{\Fs\in\{+,-\}}\|(\lambda^{1/2}w_\Fs,\nabla w_\Fs)\|_{L_q(\BR_\Fs^N)}.
\end{align*}
Here $C_7$ and $\alpha_{M_2,r,\vps}$ are the same positive constants as in Lemma $\ref{lemm:bent-reminder}$. 
\end{lemm}

Let us define $K_\lambda$ and $L_\lambda$ as follows:
\begin{align*}
K_\lambda
&=
\left\{\begin{aligned}
&\sum_{\Fs\in\{+,-\}}
\|(\dv\BF_\Fs, g_\Fs, \lambda^{1/2}F_{\Fs N},\nabla F_{\Fs N},\lambda^{1/2}h_\Fs,\nabla h_\Fs)\|_{L_q(\BR_\Fs^N)}, \\
&\sum_{\Fs\in\{+,-\}}
\|(\dv\BF_\Fs, g_\Fs,\lambda^{1/2}h_\Fs,\nabla h_\Fs)\|_{L_q(\BR_\Fs^N)} 
\text{ when $\wtd\Bn\cdot(\wtd\Bf_+-\wtd\Bf_-)=0$ on $\wtd\Sigma$,} 
\end{aligned}\right.
\\
L_\lambda
&=
\sum_{\Fs\in\{+,-\}}
\left\|\left(\BF_\Fs,|\lambda|^{-1/2}g_\Fs,h_\Fs, |\lambda|^{-1/2}\nabla h_\Fs\right)\right\|_{L_q(\BR_\Fs^N)}.
\end{align*}
Fix a positive number $S$ such that $K_\lambda\leq S/2$.
Let $C_1$ be the same positive constant as in Theorem \ref{theo:2-bdd},
and let $(w_+,w_-)\in X_{q,\lambda}(C_1 S)$.
We consider 
\begin{equation}\label{s-eq:bent-3}
\left\{\begin{aligned}
\rho_\pm\lambda z_\pm-\Delta z_\pm
&=-\dv(\BF_\pm+\CF(w_\pm))+\rho_\pm\CG_\lambda(w_\pm)+g_\pm
&& \text{in $\BR_\pm^N$,} \\
\rho_+ z_+&=\rho_- z_- && \text{on $\BR_0^N$,} \\
\Bn_0\cdot\nabla  (z_+- z_-)
&=\Bn_0\cdot\{(\BF_++\CF(w_+))-(\BF_-+\CF(w_-))\} \\  
&+h_+-h_- && \text{on $\BR_0^N$.}
\end{aligned}\right.
\end{equation}
Noting $\Bn_0\cdot(\BF_+-\BF_-)=0$ on $\BR_0^N$ when $\wtd\Bn\cdot(\wtd\Bf_+-\wtd\Bf_-)=0$ on $\wtd\Sigma$
by virtue of \eqref{0920:4_2019},
one has by Theorem \ref{theo:2-bdd}
unique solutions $z_\pm\in H_q^2(\BR_\pm^N)$ of \eqref{s-eq:bent-3} satisfying
\begin{align*}
&\|(z_+,z_-)\|_{X_{q,\lambda}}  \\
&\leq C_1 \left(K_\lambda+\sum_{\Fs\in\{+,-\}}\|(\dv\CF(w_\Fs),\rho_\Fs\CG_\lambda(w_\Fs),
\lambda^{1/2}\CF_{N}(w_\Fs), \nabla\CF_{ N}(w_\Fs))\|_{L_q(\BR_\Fs^N)}\right).
\end{align*}
Combining this inequality with Lemma \ref{lemm:bent-reminder-2} furnishes 
an a priori estimate for the solutions of \eqref{s-eq:bent-3} as follows: 
\begin{align}\label{0903:5_2019}
\|(z_+,z_-)\|_{X_{q,\lambda}} 
&\leq C_1\left\{K_\lambda+4\left(C_7(M_1+\vps)+\alpha_{M_2,r,\vps}\lambda_1^{-1/2}\right)\|(w_+,w_-)\|_{X_{q,\lambda}}\right\} \\
&\leq C_1\Big\{K_\lambda+4C_1 S\left(C_7(M_1+\vps)+\alpha_{M_2,r,\vps}\lambda_1^{-1/2}\right)\Big\}. \notag
\end{align}
Choosing $M_1$, $\vps$ small enough and $\lambda_1$ large enough so that
\begin{equation}\label{0903:6_2019}
4C_1 C_7(M_1+\vps) \leq \frac{1}{4}, \quad 4C_1 \alpha_{M_2,r,\vps}\lambda_1^{-1/2} \leq \frac{1}{4},
\end{equation}
we see from \eqref{0903:5_2019} that
$$
\|(z_+,z_-)\|_{X_{q,\lambda}}\leq C_1\left(\frac{S}{2}+\frac{S}{2}\right)=C_1 S.
$$
Thus we can define a map $\Psi: X_{q,\lambda}(C_1S)\to X_{q,\lambda}(C_1S)$ by $\Psi(w_+,w_-)=(z_+,z_-)$.
Analogously, recalling Remark \ref{rema:linear} \eqref{rema:linear-1}
and using the a priori estimate \eqref{0903:5_2019}, 
we can show that $\Psi$ is a contraction mapping on $X_{q,\lambda}(C_1 S)$,
which furnishes that there exists a fixed point $(w_+^*,w_-^*)\in X_{q,\lambda}(C_1 S)$, i.e.
$(w_{+}^*,w_-^*)=\Psi(w_+^*,w_-^*)$.
Then $(v_+,v_-)=(w_{+}^*,w_-^*)$ is a solution to \eqref{s-eq:bent-prtb} and satisfies by 
the first inequality of \eqref{0903:5_2019}, together with \eqref{0903:6_2019},
\begin{equation}\label{0903:19_2019}
\|(v_+,v_-)\|_{X_q,\lambda}\leq 2 C_1 K_\lambda. 
\end{equation}
In addition, by \eqref{est:2-whole},
\begin{align*}
&\sum_{\Fs\in\{+,-\}}\|(\lambda^{1/2}v_\Fs,\nabla v_\Fs)\|_{L_q(\BR_\Fs^N)}  \\
&\leq C_1\left(L_\lambda+\sum_{\Fs\in\{+,-\}}
\left\|\left(\CF(v_\Fs),|\lambda|^{-1/2}\rho_\Fs \CG_\lambda(v_\Fs)\right)\right\|_{L_q(\BR_\Fs^N)}\right),
\end{align*}
which, combined with Lemma \ref{lemm:bent-reminder-2} and \eqref{0903:6_2019},  furnishes
\begin{align*}
\sum_{\Fs\in\{+,-\}}\|(\lambda^{1/2}v_\Fs,\nabla v_\Fs)\|_{L_q(\BR_\Fs^N)} 
&\leq C_1\left(L_\lambda + C_7M_1\sum_{\Fs\in\{+,-\}}\|(\lambda^{1/2}v_\Fs,\nabla v_\Fs)\|_{L_q(\BR_\Fs^N)}\right) \\
&\leq  C_1 L_\lambda + \frac{1}{2}\sum_{\Fs\in\{+,-\}}\|(\lambda^{1/2}v_\Fs,\nabla v_\Fs)\|_{L_q(\BR_\Fs^N)}.
\end{align*}
It therefore holds that 
\begin{equation}\label{0903:9_2019}
\sum_{\Fs\in\{+,-\}}\|(\lambda^{1/2}v_\Fs,\nabla v_\Fs)\|_{L_q(\BR_\Fs^N)}
\leq 2 C_1 L_\lambda.
\end{equation}
The uniqueness of solutions of \eqref{s-eq:bent-prtb}
also follows from the a priori estimate \eqref{0903:5_2019} together with \eqref{0903:6_2019}.

Finally, setting $\wtd v_\pm= v_\pm\circ \Phi^{-1}$,
we observe that $\wtd v_\pm\in H_q^2(\wtd\Omega_\pm)$ are solutions to \eqref{s-eq:bent-1}
and that by \eqref{bent:condi1}, \eqref{rema:jacob}, and Lemma \ref{lemm:bent1} with $\vps=1$
\begin{equation*}
\sum_{\Fs\in\{+,-\}}\|(\lambda \wtd v_\Fs,\lambda^{1/2}\nabla \wtd v_\Fs,\nabla^2 \wtd v_\Fs)\|_{L_q(\wtd\Omega_\Fs)} \leq 
C\|(v_+,v_-)\|_{X_{q,\lambda}}
\end{equation*}
for some positive constant $C=C(M_2,N,q,r)$ independent of $M_1$, $\lambda_1$, and $\lambda$.
Combining this inequality with \eqref{0903:19_2019},
together with \eqref{0824:12_2019}-\eqref{0903:15_2019},
furnishes \eqref{0903:11_2019} and \eqref{0903:11_2019-2}.
On the other hand, since
\begin{equation*}
\sum_{\Fs\in\{+,-\}}\|(\lambda^{1/2} \wtd v_\Fs,\nabla \wtd v_\Fs)\|_{L_q(\wtd\Omega_\Fs)} \leq 
C\sum_{\Fs\in\{+,-\}}\|(\lambda^{1/2}v_\Fs,\nabla v_\Fs)\|_{L_q(\BR_{\Fs}^N)}
\end{equation*}
for some positive constant $C$ independent of $M_1$, $M_2$, $\lambda_1$, $\lambda$, and $r$,
the required estimate \eqref{0903:12_2019} follows from \eqref{0903:9_2019} together with \eqref{0903:15_2019}.
The uniqueness of solutions of \eqref{s-eq:bent-1}
follows from the uniqueness of solutions of the equivalent system \eqref{s-eq:bent-prtb}.
This completes the proof of the Theorem \ref{theo:bent}.

%%%%%%%%%%%%%%%%%%%%%%%%%%%%%%%%%%%%%%%%%%%%%%%%%%
\subsection{Proof of Theorem \ref{theo:bent2}}
The uniqueness of solutions of \eqref{s-eq:bent-2} follows from the existence of solutions of \eqref{s-eq:bent-2}
similarly to the proof of Theorem \ref{theo:2-bdd_2} \eqref{theo:2-bdd_2-2},
so that we prove the existence of solutions of \eqref{s-eq:bent-2}
by following the idea of  one-phase flows introduced in 
\cite{Shibata-book} in what follows.

Setting $v_\pm=\wtd v_\pm\circ \Phi$ and $\Bf=\wtd\Bf_\pm\circ \Phi$ in \eqref{s-eq:bent-2},
we have, similarly to the previous subsection, 
\begin{equation}\label{eq:redu-zero}
\left\{\begin{aligned}
\Delta v_\pm &=\dv(\BF_\pm+\CF(v_\pm)) && \text{in $\BR_\pm^N$,} \\
\rho_+v_+&=\rho_-v_- && \text{on $\BR_0^N$,} \\
\Bn_0\cdot\nabla(v_+-v_-)&=\Bn_0\cdot\{(\BF_++\CF(v_+))-(\BF_-+\CF(v_-))\} && \text{on $\BR_0^N$,}
\end{aligned}\right.
\end{equation}
where $\BF_\pm$ and $\CF(v_\pm)$ are given in \eqref{0903:3_2019}.
By Theorem \ref{theo:2-bdd_2},
one sees for $w_\pm\in\wht H_q^1(\BR_\pm^N)\cap \wht H_q^2(\BR_\pm^N)$ that
there is a unique solution $z_\pm \in \wht H_q^1(\BR_\pm^N)\cap \wht H_q^2(\BR_\pm^N)$ to
\begin{equation}\label{eq:redu-zero-2}
\left\{\begin{aligned}
\Delta z_\pm &=\dv(\BF_\pm+\CF(w_\pm)) && \text{in $\BR_\pm^N$,} \\
\rho_+z_+&=\rho_-z_- && \text{on $\BR_0^N$,} \\
\Bn_0\cdot\nabla(z_+-z_-)&=\Bn_0\cdot\{(\BF_++\CF(w_+))-(\BF_-+\CF(w_-))\} && \text{on $\BR_0^N$.}
\end{aligned}\right.
\end{equation}
In addition, by \eqref{0819:1_2019} and \eqref{0819:2_2019},
\begin{align*}
\sum_{\Fs\in\{+,-\}}\|\nabla^2 z_\Fs\|_{L_q(\BR_\Fs^N)} 
&\leq C_2\sum_{\Fs\in\{+,-\}} \|(\dv \BF_\Fs,\nabla F_{\Fs N},\dv\CF(w_\Fs),\nabla\CF_N(w_\Fs))\|_{L_q(\BR_\Fs^N)}, \\
\sum_{\Fs\in\{+,-\}}\|\nabla z_\Fs\|_{L_q(\BR_\Fs^N)}
&\leq C_2\sum_{\Fs\in\{+,-\}}\|(\BF_\Fs,\CF(w_\Fs))\|_{L_q(\BR_\Fs^N)}, \notag
\end{align*}
which, combined with Lemma \ref{lemm:bent-reminder}, 
furnishes a priori estimates for the solutions of \eqref{eq:redu-zero-2} as follows:
\begin{align}
&\sum_{\Fs\in\{+,-\}}\|\nabla^2 z_\Fs\|_{L_q(\BR_\Fs^N)}
\leq C_2 \sum_{\Fs\in\{+,-\}}\Big(\|(\dv\BF_\Fs,\nabla F_{\Fs N})\|_{L_q(\BR_\Fs^N)} \label{bent:a-priori-1} \\
&+C_7(M_1+\vps)\|\nabla^2 w_\Fs\|_{L_q(\BR_\Fs^N)}
+\alpha_{M_2,r,\vps}\|\nabla w_\Fs\|_{L_q(\BR_\Fs^N)}\Big), \notag \\
&\sum_{\Fs\in\{+,-\}}\|\nabla z_\Fs\|_{L_q(\BR_\Fs^N)}
\leq C_2\sum_{\Fs\in\{+,-\}}\Big(\|\BF_\Fs\|_{L_q(\BR_\Fs^N)}+C_7 M_1\|\nabla w_\Fs\|_{L_q(\BR_\Fs^N)}\Big). \label{bent:apriori-2}
\end{align}

\begin{rema}
If $\wtd\Bn\cdot(\wtd\Bf_+-\wtd\Bf_-)=0$ on $\wtd\Sigma$,
then $\Bn_0\cdot(\BF_+-\BF_-)=0$ on $\BR_0^N$ by \eqref{0920:4_2019}.
When $\wtd\Bn\cdot(\wtd\Bf_+-\wtd\Bf_-)=0$ on $\wtd\Sigma$,
one thus has by \eqref{0824:1_2019}
$$
\sum_{\Fs\in\{+,-\}}\|\nabla^2 z_\Fs\|_{L_q(\BR_\Fs^N)} 
\leq C_2\sum_{\Fs\in\{+,-\}}\|(\dv\BF_\Fs,\dv\CF(w_\Fs))\|_{L_q(\BR_\Fs^N)},
$$
and then the a priori estimate \eqref{bent:a-priori-1} is replaced by 
\begin{align}\label{bent:apriori-3}
&\sum_{\Fs\in\{+,-\}}\|\nabla^2 z_\Fs\|_{L_q(\BR_\Fs^N)} 
\leq C_2\sum_{\Fs\in\{+,-\}} \Big(\|\dv \BF_\Fs\|_{L_q(\BR_\Fs^N)} \\
&+C_7(M_1+\vps)\|\nabla^2 w_\Fs\|_{L_q(\BR_\Fs^N)}
+\alpha_{M_2,r,\vps}\|\nabla w_\Fs\|_{L_q(\BR_\Fs^N)}\Big). \notag
\end{align}
\end{rema}

From now on, we prove the existence of solutions $v_\pm$ to \eqref{eq:redu-zero}.
Let $v_\pm^{(0)}=0$ and $v_\pm^{(j)}\in\wht H_q^1(\BR_\pm^N)\cap\wht H_q^2(\BR_\pm^N)$, $j\geq 1$, be
the unique solutions to
\begin{equation}\label{eq:iteration}
\left\{\begin{aligned}
\Delta v_\pm^{(j)} &=\dv(\BF_\pm+\CF(v_\pm^{(j-1)})) && \text{in $\BR_\pm^N$,} \\
\rho_+v_+^{(j)}&=\rho_-v_-^{(j)} && \text{on $\BR_0^N$,} \\
\Bn_0\cdot\nabla(v_+^{(j)}-v_-^{(j)})&=\Bn_0\cdot\{(\BF_++\CF(v_+^{(j-1)})) \\
&-(\BF_-+\CF(v_-^{(j-1)}))\} && \text{on $\BR_0^N$.}
\end{aligned}\right.
\end{equation}
By \eqref{bent:a-priori-1} and \eqref{bent:apriori-2}, $v_\pm^{(j)}$ satisfy
\begin{align}
&\sum_{\Fs\in\{+,-\}}\|\nabla^2 v_\Fs^{(j)}\|_{L_q(\BR_\Fs^N)}
\leq C_2 \sum_{\Fs\in\{+,-\}}\Big(\|(\dv\BF_\Fs,\nabla F_{\Fs N})\|_{L_q(\BR_\Fs^N)} \label{0904:5_2019} \\
&+C_7(M_1+\vps)\|\nabla^2 v_\Fs^{(j-1)}\|_{L_q(\BR_\Fs^N)} 
+\alpha_{M_2,r,\vps}\|\nabla v_\Fs^{(j-1)}\|_{L_q(\BR_\Fs^N)}\Big), \notag \\
&\sum_{\Fs\in\{+,-\}}\|\nabla v_\Fs^{(j)}\|_{L_q(\BR_\Fs^N)} \label{0904:6_2019} \\
&\leq C_2\sum_{\Fs\in\{+,-\}}\Big(\|\BF_\Fs\|_{L_q(\BR_\Fs^N)}+C_7 M_1\|\nabla v_\Fs^{(j-1)}\|_{L_q(\BR_\Fs^N)}\Big).  \notag
\end{align}
Inductively, it follows from \eqref{0904:6_2019} and $v_\pm^{(0)}=0$ that
\begin{equation}\label{0920:15_2019}
\sum_{\Fs\in\{+,-\}}\|\nabla v_\Fs^{(j)}\|_{L_q(\BR_\Fs^N)}  
\leq X_j\sum_{\Fs\in\{+,-\}}\|\BF_\Fs\|_{L_q(\BR_\Fs^N)}, 
\quad X_j=C_2\sum_{k=0}^{j-1}(C_2C_7M_1)^{k},
\end{equation}
which, inserted into \eqref{0904:5_2019}, furnishes 
\begin{align*}
&\sum_{\Fs\in\{+,-\}}\|\nabla^2 v_\Fs^{(j)}\|_{L_q(\BR_\Fs^N)}
\leq C_2 \sum_{\Fs\in\{+,-\}}\Big(\|(\dv\BF_\Fs,\nabla F_{\Fs N})\|_{L_q(\BR_\Fs^N)} \\
&+C_7(M_1+\vps)\|\nabla^2 v_\Fs^{(j-1)}\|_{L_q(\BR_\Fs^N)} 
+\alpha_{M_2,r,\vps} X_j \|\BF_\Fs\|_{L_q(\BR_\Fs^N)} \Big).
\end{align*}
Inductively, it follows from this inequality and $v_\pm^{(0)}=0$ that
\begin{align}\label{0920:17_2019}
\sum_{\Fs\in\{+,-\}}\|\nabla^2 v_\Fs^{(j)}\|_{L_q(\BR_\Fs^N)}
&\leq 
Y_j \sum_{\Fs\in\{+,-\}}\|(\dv\BF_\Fs,\nabla F_{\Fs N})\|_{L_q(\BR_\Fs^N)} \\
&+Z_j\sum_{\Fs\in\{+,-\}}\|\BF_\Fs\|_{L_q(\BR_\Fs^N)}, \notag
\end{align}
where
\begin{equation*}
Y_j=C_2\sum_{k=0}^{j-1}(C_2C_7(M_1+\vps))^k, \quad 
Z_j=C_2\alpha_{M_2,r,\vps}\sum_{k=0}^{j-1}(C_2C_7(M_1+\vps))^k X_{j-k}.
\end{equation*}

Let $u^{(j)}=\rho_+v_+^{(j)}\mathds{1}_{\BR_+^N}+\rho_-v_-^{(j)}\mathds{1}_{\BR_-^N}$.
One then sees by $\rho_+ v_+^{(j)}=\rho_-v_-^{(j)}$ on $\BR_0^N$ that $u^{(j)}\in\wht H_q^1(\BR^N)$ with
$
\nabla u^{(j)} = \rho_+ (\nabla v_+^{(j)})\mathds{1}_{\BR_+^N}+\rho_- (\nabla v_-^{(j)})\mathds{1}_{\BR_-^N}.
$
On the other hand,
recalling Remark \ref{rema:linear} \eqref{rema:linear-1}, we have by \eqref{bent:apriori-2}
\begin{equation*}
\sum_{\Fs\in\{+,-\}}\|\nabla(v_\Fs^{(j)}-v_\Fs^{(j-1)})\|_{L_q(\BR_\Fs^N)} \leq C_2C_7 M_1
\sum_{\Fs\in\{+,-\}}\|\nabla (v_\Fs^{(j-1)}-v_\Fs^{(j-2)})\|_{L_q(\BR_\Fs^N)},
\end{equation*}
which, combined with 
\begin{align*}
&\frac{1}{C_8}\sum_{\Fs\in\{+,-\}}\|\nabla (v_{\Fs}^{(j)}-v_{\Fs}^{(j-1)})\|_{L_q(\BR_\Fs^N)} \\
&\leq \|\nabla (u^{(j)}-u^{(j-1)})\|_{L_q(\BR^N)}
\leq C_8 \sum_{\Fs\in\{+,-\}}\|\nabla (v_{\Fs}^{(j)}-v_{\Fs}^{(j-1)})\|_{L_q(\BR_\Fs^N)}
\end{align*}
for some positive constants $C_8\geq 1$ depending only on $q$, $\rho_+$, and $\rho_-$, furnishes 
\begin{equation*}
\|\nabla(u^{(j)}-u^{(j-1)})\|_{L_q(\BR^N)}
\leq C_2C_7(C_8)^2 M_1\|\nabla (u^{(j-1)}-u^{(j-2)})\|_{L_q(\BR^N)}.
\end{equation*}
Choose $M_1$ and $\vps$ small enough so that 
\begin{equation*}
C_2C_7(C_8)^2 M_1 \leq \frac{1}{2}, \quad C_2C_7(M_1+\vps)\leq \frac{1}{2}.
\end{equation*}
Then the above inequality for $u^{(j)}-u^{(j-1)}$ implies that 
$\{u^{(j)}\}_{j=1}^\infty$ is a Cauchy sequence in $\wht H_q^1(\BR^N)$,
and thus there exists $u\in \wht H_q^1(\BR^N)$ such that
\begin{equation}\label{0920:11_2019}
\lim_{j\to\infty}\|\nabla(u^{(j)}-u)\|_{L_q(\BR^N)}=0.
\end{equation}
Furthermore, we see that
\begin{align*}
X_j \leq C_2\sum_{k=0}^{\infty}2^{-k} =2C_2, \quad Y_j\leq  2C_2, \quad
Z_j \leq C_2 \alpha_{M_2,r,\vps} \cdot 2 \cdot 2C_2,
\end{align*}
which, combined with \eqref{0920:15_2019} and \eqref{0920:17_2019}, furnishes
\begin{align}
\sum_{\Fs\in\{+,-\}}\|\nabla v_\Fs^{(j)}\|_{L_q(\BR_\Fs^N)}  
&\leq 2C_2\sum_{\Fs\in\{+,-\}}\|\BF_\Fs\|_{L_q(\BR_\Fs^N)}, \label{0921:3_2019} \\
\sum_{\Fs\in\{+,-\}}\|\nabla^2 v_\Fs^{(j)}\|_{L_q(\BR_\Fs^N)} 
&\leq 2C_2 K, \label{0921:4_2019}
\end{align}
where we have set
\begin{equation*}
K=\sum_{\Fs\in\{+,-\}}\|(\dv\BF_\Fs,\nabla F_{\Fs N})\|_{L_q(\BR_\Fs^N)} 
+2C_2 \alpha_{M_2,r,\vps}\sum_{\Fs\in\{+,-\}}\|\BF_\Fs\|_{L_q(\BR_\Fs^N)}.
\end{equation*}

Let $u_+$ and $u_-$ be respectively the restriction of $u$ on $\BR_+^N$
and the restriction of $u$ on $\BR_-^N$. It then holds that $u_\pm \in \wht H_q^1(\BR_\pm^N)$ and
\begin{equation}\label{properties:u}
\nabla u=(\nabla u_+)\mathds{1}_{\BR_+^N}+(\nabla u_-)\mathds{1}_{\BR_-^N},
\quad u_+=u_- \text{ on $\BR_0^N$.}
\end{equation}
Let us define $v_\pm = \rho_\pm^{-1}u_\pm$. 
By the second property of \eqref{properties:u},
\begin{equation}\label{0920:10_2019}
\rho_+ v_+ = \rho_- v_- \quad \text{on $\BR_0^N$.}
\end{equation}
In addition, we see by the definition of $u^{(j)}$ and the first property of \eqref{properties:u} that
\begin{equation*}
\|\nabla(v_\pm^{(j)}-v_\pm)\|_{L_q(\BR_\pm^N)}
=\rho_\pm^{-1}\|\nabla(u^{(j)}-u)\|_{L_q(\BR_\pm^N)}
\end{equation*}
which, combined with \eqref{0920:11_2019}, furnishes
\begin{equation}\label{0920:12_2019}
\lim_{j\to \infty}\|\nabla(v_\pm^{(j)}-v_\pm)\|_{L_q(\BR_\pm^N)}=0.
\end{equation}
Taking the the limit: $j\to \infty$ in \eqref{0921:3_2019} thus implies that
$v_\pm$ satisfy
\begin{equation}\label{0904:7_2019}
\sum_{\Fs\in\{+,-\}}\|\nabla v_\Fs\|_{L_q(\BR_\Fs^N)}  
\leq 2C_2\sum_{\Fs\in\{+,-\}}\|\BF_\Fs\|_{L_q(\BR_\Fs^N)},
\end{equation}

Next, we consider the higher regularity of $v_\pm$.
Let $k,l=1,\dots,N$. By \eqref{0921:4_2019},
there exist $v_{\pm,kl}\in L_q(\BR_\pm^N)$ such that $\pd_k\pd_l v_\pm^{(j)} \to v_{\pm,kl}$
weakly in $L_q(\BR_\pm^N)$ as $j\to\infty$.
Then we can prove $\pd_k\pd_l v_\pm=v_{\pm,kl}\in L_q(\BR_\pm^N)$
by using the convergence in distribution.
Consequently, $\pd_k\pd_l v_\pm^{(j)}\to \pd_k\pd_l v_\pm$ weakly in $L_q(\BR_\pm^N)$ as $j\to \infty$.
It thus follows from \eqref{0921:4_2019} that
\begin{equation}\label{0921:7_2019}
\|\pd_k\pd_l v_\pm\|_{L_q(\BR_\pm^N)} \leq \liminf_{j\to \infty}\|\pd_k\pd_lv_{\pm}^{(j)}\|_{L_q(\BR_\pm^N)} \leq 2C_2K.
\end{equation}

Finally, we prove that $v_\pm$ satisfies \eqref{eq:redu-zero} by the limit: $j\to\infty$ in \eqref{eq:iteration}.
The first line of \eqref{eq:redu-zero} is satisfied immediately,
and the third line of \eqref{eq:redu-zero} can be proved by
Remark \ref{rema:linear} \eqref{rema:linear-1}, Lemma \ref{lemm:bent-reminder}, and
the inequality (cf. \cite[Proposition 16.2]{Di2016}):
\begin{equation*}
\|f_\pm \|_{L_q(\BR_0^N)} 
\leq q^{1/q} \|f_\pm\|_{L_q(\BR_\pm^N)}^{1-1/q}\|\nabla f_\pm\|_{L_q(\BR_\pm^N)}^{1/q}
\quad \text{for any $f_\pm\in H_q^1(\BR_\pm^N)$},
\end{equation*}
together with \eqref{0921:4_2019}, \eqref{0920:12_2019}, and \eqref{0921:7_2019}.
In addition, the second line of \eqref{eq:redu-zero} is already obtained in \eqref{0920:10_2019}.
Thus, setting $\wtd v_\pm=v_\pm\circ \Phi$, we see that $\wtd v_\pm$ are solutions to \eqref{s-eq:bent-2}
and that the required estimates \eqref{0920:1_2019} and \eqref{0920:2_2019}
follow from \eqref{0904:7_2019} and \eqref{0921:7_2019}
in the same manner as in the last part of Subsection \ref{subset:1_bent}.
When $\wtd\Bn\cdot(\wtd\Bf_+-\wtd\Bf_-)=0$ on $\wtd\Sigma$,
we can obtain \eqref{0920:3_2019} by using \eqref{bent:apriori-3} instead of \eqref{bent:a-priori-1}
in the above argument.
This completes the proof of Theorem \ref{theo:bent2}.

%%%%%%%%%%%%%%%%%%%%%%%%%%%%%%%%%%%%%%%%%%%%%%%%%%
\subsection{Weak elliptic problem with bent interface}\label{subsec:whole2-weak}
Throughout this subsection, we assume that $r\in(N,\infty)$ and 
$\wtd\Omega_\pm$, $\wtd\Sigma$ are given in Theorem \ref{theo:bent2}.
For $\Bf\in L_q(\BR^N\setminus\wtd\Sigma)^N$ with $q\in(1,\infty)$,
this subsection considers the unique solvability of the following weak elliptic problem:
Find $u\in\wht H_q^1(\BR^N)$ such that
\begin{equation}\label{weakeq:whole-bent}
(\rho^{-1}\nabla u,\nabla\vph)_{\BR^N\setminus\wtd\Sigma}=(\Bf,\nabla\vph)_{\BR^N\setminus\wtd\Sigma}
\quad \text{for any $\vph\in\wht H_{q'}^1(\BR^N)$,}
\end{equation}
where $q'=q/(q-1)$ and $\rho=\rho_+\mathds{1}_{\wtd \Omega_+}+\rho_-\mathds{1}_{\wtd\Omega_-}$ for positive constants $\rho_\pm$.
In the same manner that we have obtained Theorem \ref{w-theo:whole} from Theorem \ref{theo:2-bdd_2},
we can obtain the following theorem from Theorem \ref{theo:bent2}.

\begin{theo}\label{w-theo:whole2}
Let $r$, $\wtd\Omega_\pm$, and $\wtd\Sigma$ be as above, and let $q\in(1,\infty)$.
Assume that $\max(q,q')\leq r$ for $q'=q/(q-1)$
and $\rho=\rho_+\mathds{1}_{\wtd \Omega_+}+\rho_-\mathds{1}_{\wtd\Omega_-}$ for constants $\rho_\pm>0$.
\begin{enumerate}[$(1)$]
\item {\bf Existence}.
Let $\Bf\in L_q(\BR^N\setminus\wtd \Sigma)^N$.
Then the weak elliptic problem \eqref{weakeq:whole-bent} admits a solution $u\in\wht H_q^1(\BR^N)$ satisfying
\begin{equation*}\label{est1:whole}
\|\nabla u\|_{L_q(\BR^N)} \leq 2C'(\rho_++\rho_-)\|\Bf\|_{L_q(\BR^N\setminus\wtd\Sigma)},
\end{equation*}
where $C'$ is the same constant as in Theorem $\ref{theo:bent2}$.
\item {\bf Uniqueness}.
If $u\in \wht H_q^1(\BR^N)$ satisfies
\begin{equation*}
(\rho^{-1}\nabla u,\nabla\vph)_{\BR^N\setminus\wtd\Sigma}=0 \quad \text{for any $\vph\in\wht H_{q'}^1(\BR^N)$,}
\end{equation*}
then $u=c$ for some constant $c$.
\end{enumerate}
\end{theo}

%%%%%%%%%%%%%%%%%%%%%%%%%%%%%%%%%%%%%%%%%%%%%%%%%%%%%%%%%%%%%%%%%%%%%%
\section{Strong problems in bounded domains}\label{sec:bounded}
Throughout this section, we assume

\begin{assu}\label{assu:bdd}
\begin{enumerate}[{\rm (a)}]
\item
$r$ is a real number satisfying $r>N$.
\item
$G$ is a bounded domain with boundary $\Gamma$ of class $W_r^{2-1/r}$.
\item
$G_+$ is a subdomain of $G$ with boundary $\Sigma$ of class $W_r^{2-1/r}$ satisfying $\Sigma\cap \Gamma=\emptyset$.
\item
$G_-=G\setminus(G_+\cup\Sigma)$.
\end{enumerate}
\end{assu}

Let $\Bn$ be a unit normal vector on $\Sigma$ pointing from $G_+$ into $G_-$, and set
\begin{equation*}
F_q(G_\pm)=\{\Bf_\pm\in E_q(G_\pm): \Bn\cdot\Bf_\pm \in H_q^1(G_\pm)\}.
\end{equation*}
In this section, we consider the following strong elliptic problem: 
\begin{equation}\label{s-eq:1-bdd}
\left\{\begin{aligned}
\Delta v_\pm &=\dv\Bf_\pm && \text{in $G_\pm$,} \\
\rho_+v_+&=\rho_-v_- && \text{on $\Sigma$,} \\
\Bn\cdot  \nabla(v_+- v_-)
&=\Bn\cdot(\Bf_+-\Bf_-) && \text{on $\Sigma$,} \\
v_-&=0 && \text{on $\Gamma$.}
\end{aligned}\right.
\end{equation}
More precisely, we prove

\begin{theo}\label{theo:bdd}
Suppose that Assumption $\ref{assu:bdd}$ holds and $\rho_\pm$ are positive constants.
Let $q\in(1,\infty)$ with $\max(q,q')\leq r$ for $q'=q/(q-1)$.
\begin{enumerate}[$(1)$]
\item
{\bf Existence}.
Let $\Bf _\pm \in F_q(G_\pm)$.
Then the strong elliptic problem \eqref{s-eq:1-bdd} admits solutions $v_\pm\in H_q^2(G_\pm)$
satisfying 
\begin{align}
\|\nabla^2 v_\pm\|_{L_q(G_\pm)}
&\leq C \sum_{\Fs\in\{+,-\}}\left(\|\Bf_\Fs\|_{E_q(G_\Fs)}+\|\Bn\cdot\Bf_\Fs\|_{H_q^1(G_\Fs)}\right), \label{0922:1_2019} \\
\|v_\pm \|_{H_q^1(G_\pm)}
&\leq C \sum_{\Fs\in\{+,-\}}\|\Bf_\Fs\|_{L_q(G_\Fs)}, \label{0922:2_2019}
\end{align}
with a positive constant $C=C(N,q,r,\rho_+,\rho_-)$.
Additionally, if $\Bn\cdot(\Bf_+-\Bf_-)=0$ on $\Sigma$, then $v_\pm$ satisfy
\begin{equation}\label{0922:3_2019}
\|\nabla^2 v_\pm\|_{L_q(G_\pm)}
\leq C \sum_{\Fs\in\{+,-\}}\|\Bf_\Fs\|_{E_q(G_\Fs)}. 
\end{equation}
\item
{\bf Uniqueness}. If $v_\pm\in H_q^2(G_\pm)$ satisfy
\begin{equation*}
\Delta v_\pm=0 \text{ in $G_\pm$};
\ \rho_+ v_+= \rho_- v_-, \ \Bn\cdot \nabla( v_+-v_-)=0 \text{ on $\Sigma$};
\ v_-=0 \text{ on $\Gamma$,}
\end{equation*}
then $v_\pm=0$. 
\end{enumerate}
\end{theo}

%%%%%%%%%%%%%%%%%%%%%%%%%%%%%%%%%%%%%%%%%%%%%%%%%%
\subsection{Strong elliptic problem with $\lambda$ in $G_\pm$}
To prove Theorem \ref{theo:bdd}, we consider 
\begin{equation}\label{eq:1-bdd}
\left\{\begin{aligned}
\rho_\pm\lambda v_\pm-\Delta v_\pm &= -\dv\Bf_\pm+ g_\pm && \text{in $G_\pm$,} \\
\rho_+ v_+&=\rho_- v_- && \text{on $\Sigma$,} \\
\Bn\cdot \nabla ( v_+- v_-)&=\Bn\cdot(\Bf_+-\Bf_-)+h_+-h_-  && \text{on $\Sigma$,} \\
v_-&=0 && \text{on $\Gamma$.}  
\end{aligned}\right.
\end{equation}
Note that \cite[Appendix B]{Shibata2018} already studies
the strong elliptic problem in $\BR_+^N$ with the Dirichlet boundary condition
in the case where $\lambda$ is taken into account,
so that resolvent estimates of solutions in the bent half-space case
are also available similarly to the proof of Theorem \ref{theo:bent}.
Combining that result with Theorem \ref{theo:bent} and the standard localization technique yields

\begin{theo}\label{theo:1-bdd}
Suppose that the same assumption as in Theorem $\ref{theo:bdd}$ holds, and let $\sigma\in(0,\pi/2)$.
Then there is a constant $\lambda_2\geq 1$ 
such that, for any $\lambda\in\Sigma_{\sigma,\lambda_2}$, 
$$
g_\pm\in L_q(G_\pm), \quad
\Bf_\pm\in F_q(G_\pm), \quad 
h_\pm\in H_q^1(G_\pm),
$$ 
the strong elliptic problem \eqref{eq:1-bdd} admits unique
solutions $v_\pm \in H_q^2(G_\pm)$.
In addition, the solutions $v_\pm$ satisfy
\begin{align*}
&\|(\lambda v_\pm ,\lambda^{1/2}\nabla v_\pm,\nabla^2 v_\pm)\|_{L_q(G_\pm)} \\
&\leq C\sum_{\Fs\in\{+,-\}}
\|(\dv\Bf_\Fs,g_\Fs,\lambda^{1/2}(\Bn\cdot\Bf_\Fs),\nabla(\Bn\cdot\Bf_\Fs),
\lambda^{1/2}h_\Fs,\nabla h_\Fs)\|_{L_q(G_\Fs)},
\end{align*}
and also
\begin{align*}
\|(\lambda^{1/2}v_\pm,\nabla v_\pm)\|_{L_q(G_\pm)} 
&\leq C\sum_{\Fs\in\{+,-\}}
\Big(\|\Bf_\Fs\|_{L_q(\Omega_\Fs)}+|\lambda|^{-1/2}\|g_+\|_{L_q(G_\Fs)} \\
&+\|h_\Fs\|_{L_q(G_\Fs)}+|\lambda|^{-1/2}\|\nabla h_\Fs\|_{L_q(G_\Fs)}\Big),
\end{align*}
where $C=C(N,q,r,\rho_+,\rho_-,\sigma)$ is a positive constant independent of $\lambda$.
Additionally, if $\Bn\cdot(\Bf_+-\Bf_-)=0$ on $\Sigma$, then $v_\pm$ satisfy 
\begin{align*}
&\|(\lambda v_\pm ,\lambda^{1/2}\nabla v_\pm,\nabla^2 v_\pm)\|_{L_q(G_\pm)} \\
&\leq C\sum_{\Fs\in\{+,-\}}
\|(\dv\Bf_\Fs,g_\Fs,
\lambda^{1/2}h_\Fs,\nabla h_\Fs)\|_{L_q(G_\Fs)}.
\end{align*}
\end{theo}

%%%%%%%%%%%%%%%%%%%%%%%%%%%%%%%%%%%%%%%%%%%%%%%%%%
\subsection{Proof of Theorem \ref{theo:bdd}.}
We start with the following strong elliptic problem: 
\begin{equation}\label{eq:5-bdd}
\left\{\begin{aligned}
-\Delta v_\pm &= d_\pm  && \text{in $G_\pm$,} \\
\rho_+ v_+&= \rho_- v_- && \text{on $\Sigma$,} \\
\Bn\cdot \nabla ( v_+- v_-)&= 0 && \text{on $\Sigma$,} \\
v_-&=0 && \text{on $\Gamma$.}
\end{aligned}\right.
\end{equation}
Concerning this problem, we have

\begin{theo}\label{theo:3-bdd}
Suppose that the same assumption as in Theorem $\ref{theo:bdd}$ holds.
Then,  for any $d_\pm \in L_q(G_\pm)$,
the strong elliptic problem \eqref{eq:5-bdd} admits unique solutions $v_\pm \in H_q^2(G_\pm)$,
which satisfy
\begin{equation*}
\sum_{\Fs\in\{+,-\}}\|v_\Fs\|_{H_q^2(G_\Fs)}
\leq  C \sum_{\Fs\in\{+,-\}}\|d_\Fs\|_{L_q(G_\Fs)} 
\end{equation*}
for some positive constant $C=C(N,q,r,\rho_+,\rho_-)$.
\end{theo}

\begin{proof}
The proof is based on the Riesz-Schauder theory together with Theorem \ref{theo:1-bdd},
so that the detailed proof may be omitted. 
\end{proof}

Now we prove Theorem \ref{theo:bdd}. Let $\mu=2 \lambda_2$ for $\lambda_2$ introduced in Theorem \ref{theo:1-bdd}.
The Theorem \ref{theo:1-bdd} yields $w_\pm\in H_q^2(G_\pm)$ such that
\begin{equation*}
\left\{\begin{aligned}
\rho_\pm \mu w_\pm -\Delta w_\pm &= -\dv\Bf_\pm && \text{in $G_\pm$,} \\
\rho_+ w_+&=\rho_- w_- && \text{on $\Sigma$,} \\
\Bn\cdot\nabla (w_+- w_-)&=\Bn\cdot(\Bf_+-\Bf_-) && \text{on $\Sigma$,} \\
w_- &= 0 && \text{on $\Gamma$,}
\end{aligned}\right.
\end{equation*}
while Theorem \ref{theo:3-bdd} yields $z_\pm\in H_q^2(G_\pm)$ such that
\begin{equation*}
\left\{\begin{aligned}
-\Delta  z_\pm &= \rho_\pm \mu w_\pm && \text{in $G_\pm$,} \\
\rho_+ z_+ &= \rho_- z_- && \text{on $\Sigma$,} \\
\Bn\cdot\nabla (z_+-z_-) &= 0 && \text{on $\Sigma$,} \\
z_- &= 0 && \text{on $\Gamma$.}
\end{aligned}\right.
\end{equation*}
Thus $v_\pm=w_\pm+z_\pm$ become solutions to \eqref{s-eq:1-bdd} and satisfy
\eqref{0922:1_2019}-\eqref{0922:3_2019} by Theorems \ref{theo:1-bdd} and \ref{theo:3-bdd}.
The uniqueness of solutions of \eqref{s-eq:1-bdd}
is already proved in Theorem \ref{theo:3-bdd}.
This completes the proof of Theorem \ref{theo:bdd}.

%%%%%%%%%%%%%%%%%%%%%%%%%%%%%%%%%%%%%%%%%%%%%%%%%%%%%%%%%%%%%%%%%%%%%%
\section{Proof of Theorems \ref{theo:main} and \ref{theo:main2}}\label{sec:exterior}
Throughout this section, we assume that $\Omega_\pm$ satisfy Assumption \ref{assu:1}.

First, let us introduce the notation used in this section.
For $L>0$, we set
\begin{equation*}
B_L=\{x\in\BR^N : |x|<L\}, \quad S_L=\{x\in\BR^N : |x|=L\}.
\end{equation*}
Fix $R>0$ such that $\Omega_+\cup\Sigma \subset B_R$.
We then define $G$, $\Gamma$, and $G_\pm$ as follows:
\begin{equation*}
G=B_{4R}, \quad \Gamma=S_{4R}, \quad  G_+=\Omega_+, \quad G_-=B_{4R}\setminus(\Omega_+\cup\Sigma).
\end{equation*}
Let $\vph$, $\psi_0$, and $\psi_\infty$ be functions in $C^\infty(\BR^N)$ satisfying
$0\leq \vph,\psi_0, \psi_\infty\leq 1$ and
\begin{align*}
\vph(x)&=
\left\{\begin{aligned}
&1 && \text{for $x\in B_{2R}$,} \\
&0 && \text{for $x\in \BR^N\setminus B_{3R}$,}
\end{aligned}\right. \\
\psi_0(x)&=
\left\{\begin{aligned}
& 1 && \text{for $x\in B_{(3+1/3)R}$,} \\
&0 && \text{for $x\in \BR^N\setminus B_{(3+2/3)R}$,}
\end{aligned}\right. \\
\psi_\infty(x)&=
\left\{\begin{aligned}
&0 && \text{for $x\in B_{(2-2/3)R}$,} \\
&1 && \text{for $x\in\BR^N\setminus B_{(2-1/3)R}$.}
\end{aligned}\right.
\end{align*}
In addition, we set
$\vph_0(x)=\vph(x)$ and $\vph_\infty(x)=1-\vph(x)$.
For $q\in(1,\infty)$ and $D\in\{\BR^N,G\}$, a Banach space $E_{q}(D\setminus\Sigma)$ is defined by
\begin{align*}
E_q(D\setminus\Sigma)
&=\{\Bf\in L_q(D\setminus\Sigma)^N : \Bf|_{D_+}\in E_q(D_+) \text{ and } \Bf|_{D_-} \in E_q(D_-)\}, \\
\|\Bf\|_{E_q(D\setminus\Sigma)}&=\|\Bf|_{D_+}\|_{E_q(D_+)}+\|\Bf|_{D_-}\|_{E_q(D_-)},
\end{align*}
where $\Bf|_{D_\pm}$ stands for the restriction of $\Bf$ on $D_\pm$, respectively.
On the other hand,
\begin{align*}
F_q(D\setminus\Sigma)
&=\{\Bf\in E_q(D\setminus\Sigma):
\Bn\cdot\Bf|_{D_+}\in H_q^1(D_+) \text{ and } \Bn\cdot\Bf|_{D_-}\in H_q^1(D_-)\},  \\
\|\Bf\|_{F_q(D\setminus\Sigma)}
&=\|\Bf\|_{E_q(D\setminus\Sigma)}+\|\Bn\cdot\Bf|_{D_+}\|_{H_q^1(D_+)}+\|\Bn\cdot\Bf|_{D_-}\|_{H_q^1(D_-)}.
\end{align*}

\begin{rema}\label{rema:Fq}
One sees that $\Bf\in F_q(D\setminus\Sigma)$ is equivalent to $\Bf=\Bf_+\mathds{1}_{D_+}+\Bf_-\mathds{1}_{D_-}$
for some $\Bf_\pm\in E_q(D_\pm)$ with $\Bn\cdot\Bf_\pm\in H_q^1(\Omega_\pm)$.
\end{rema}

The aim of this section is to prove Theorems \ref{theo:main} and \ref{theo:main2}.
To this end, it suffices to prove the existence of solutions of \eqref{exteq:1}
satisfying \eqref{theo:main-est1}-\eqref{theo:main-est2}.\footnote{
The uniqueness of solutions of \eqref{exteq:1}  
follows from the existence of solutions of \eqref{exteq:1} similarly to the proof of Theorem \ref{theo:2-bdd_2} \eqref{theo:2-bdd_2-2},
while Theorem \ref{theo:main2} follows from Theorem \ref{theo:main}
as Theorem \ref{w-theo:whole} has followed from Theorem \ref{theo:2-bdd_2}.}
Instead of \eqref{exteq:1}, we consider the following equivalent system of \eqref{exteq:1} for simplicity of notation:
\begin{equation}\label{final-existence}
\left\{\begin{aligned}
\Delta v&= \dv\Bf && \text{in $\BR^N\setminus\Sigma$,} \\
\jmp{\rho v}&=0 && \text{on $\Sigma$,} \\
\jmp{\Bn\cdot\nabla v}&=\jmp{\Bn\cdot\Bf} && \text{on $\Sigma$,}
\end{aligned}\right.
\end{equation}
where $\rho=\rho_+\mathds{1}_{\Omega_+}+\rho_-\mathds{1}_{\Omega_-}$ for positive constants $\rho_\pm$ and
$$
\jmp{f}(x_0)=\lim_{x\to x_0, x\in \Omega_+}f(x)-\lim_{x\to x_0,  x\in \Omega_-}f(x) \quad (x_0\in\Sigma).
$$
We here recall the assumption for $q\in(1,\infty)$.

\begin{assu}\label{assu:q}
Let $q\in(1,\infty)$ and $\max(q,q')\leq r$ for $q'=q/(q-1)$.
\end{assu}

As mentioned above, it suffeces to prove in this section

\begin{theo}\label{theo:final}
Suppose that Assumption $\ref{assu:q}$ holds.
Then, for any $\Bf\in F_q(\BR^N\setminus\Sigma)$,
the system \eqref{final-existence} admits a solution 
$v\in \wht H_q^1(\BR^N\setminus\Sigma)\cap \wht H_q^2(\BR^N\setminus\Sigma)$ satisfying
\begin{equation}\label{theo:final-est1}
\|\nabla^2 v\|_{L_q(\BR^N\setminus\Sigma)}\leq C\|\Bf\|_{F_q(\BR^N\setminus\Sigma)}, \quad
\|\nabla v\|_{L_q(\BR^N\setminus\Sigma)} \leq C\|\Bf\|_{L_q(\BR^N\setminus\Sigma)},
\end{equation}
with some positive constant $C=C(N,q,r,\rho_+,\rho_-)$.
Additionally, if $\Bn\cdot(\Bf_+-\Bf_-)=0$ on $\Sigma$, then $v$ satisfies
\begin{equation}\label{theo:final-est2}
\|\nabla^2 v\|_{L_q(\BR^N\setminus\Sigma)} \leq C\|\Bf\|_{E_q(\BR^N\setminus\Sigma)}
\end{equation}
for some positive constant $C=C(R,N,q,r,\rho_+,\rho_-)$.
\end{theo}

%%%%%%%%%%%%%%%%%%%%%%%%%%%%%%%%%%%%%%%%%%%%%%%%%%
\subsection{Solution operators}
First, let us consider the following problem in the whole space: 
\begin{equation}\label{eq:2-ext}
\Delta V= \dv\BF \quad \text{in $\BR^N$.}
\end{equation}
Similarly to \cite[Page 1700]{Shibata2018}, we obtain

\begin{lemm}\label{lemm:1-ext}
Let $q\in(1,\infty)$. Then there exists a linear operator
$$
\CS_\infty:E_q(\BR^N)\to \wht H_q^1(\BR^N)\cap \wht H_q^2(\BR^N)
$$
such that $V=\CS_\infty\BF$ is a solution to \eqref{eq:2-ext}. In addition, 
\begin{gather*}
\|\CS_\infty\BF\|_{L_q(G)}+\|\nabla\CS_\infty\BF\|_{L_q(\BR^N)}
\leq C\|\BF\|_{L_q(\BR^N)}, \\
\|\nabla^2\CS_\infty\BF\|_{L_q(\BR^N)}
\leq C\|\dv\BF\|_{L_q(\BR^N)},
\end{gather*}
with some positive constant $C=C(R,N,q)$.
\end{lemm}

We next consider the following problem in a bounded domain:
\begin{equation}\label{eq:3-ext}
\left\{\begin{aligned}
\Delta V &= \dv\BF && \text{in $G\setminus\Sigma$,} \\
\jmp{\rho V}&=0 && \text{on $\Sigma$,} \\
\jmp{\Bn\cdot \nabla V}&= \jmp{\Bn\cdot \BF} && \text{on $\Sigma$,} \\
V&=0 && \text{on $\Gamma$.}
\end{aligned}\right.
\end{equation}
Then, by Theorem \ref{theo:bdd}, we have

\begin{lemm}\label{lemm:2-ext}
Suppose that Assumption $\ref{assu:q}$ holds.
Then there is a linear operator
\begin{equation*}
\CS_0: F_q(G\setminus\Sigma)  \to  H_q^2(G\setminus\Sigma)
\end{equation*}
such that, for any $\BF\in F_q(G\setminus\Sigma)$,
$V=\CS_0\BF$ is a solution to \eqref{eq:3-ext}. In addition, 
\begin{equation*}
\|\nabla^2 \CS_0\BF\|_{L_q(G\setminus\Sigma)} 
\leq C \|\BF\|_{F_q(G\setminus\Sigma)} , \quad
\|\CS_0\BF\|_{H_q^1(G\setminus\Sigma)}
\leq C\|\BF\|_{L_q(G\setminus\Sigma)},
\end{equation*}
with some positive constant $C=C(N,q,r,\rho_+,\rho_-)$.
Additionally, if $\jmp{\Bn\cdot\BF}=0$ on $\Sigma$, then it holds that 
\begin{equation*}
\|\nabla^2\CS_0\BF\|_{L_q(G\setminus\Sigma)}
\leq C\|\BF\|_{E_q(G\setminus\Sigma)}
\end{equation*}
for some positive constant $C=C(N,q,r,\rho_+,\rho_-)$.
\end{lemm}

Let us define an operator $\CS: F_q(\BR^N\setminus\Sigma)\to \wht H_q^1(\BR^N\setminus\Sigma)\cap \wht H_q^2(\BR^N\setminus\Sigma)$
by
\begin{equation}\label{defi:operatorS}
\CS\Bf=\psi_\infty\CS_\infty(\vph_\infty\Bf)+\psi_0\CS_0(\vph_0\Bf) \quad (\Bf\in F_q(\BR^N\setminus\Sigma)).
\end{equation}
Note that $\vph_\infty\Bf\in E_q(\BR^N)$ and $\vph_0\Bf\in F_q(G\setminus\Sigma)$
when $\Bf\in F_q(\BR^N\setminus\Sigma)$.
By Lemmas \ref{lemm:1-ext} and \ref{lemm:2-ext}, the above $\CS \Bf$ satisfies the following estimates:
\begin{equation}\label{estimate:S}
\|\nabla^2\CS\Bf\|_{L_q(\BR^N\setminus\Sigma)}\leq C\|\Bf\|_{F_q(\BR^N\setminus\Sigma)}, \quad
\|\nabla\CS\Bf\|_{L_q(\BR^N\setminus\Sigma)}\leq C\|\Bf\|_{L_q(\BR^N\setminus\Sigma)}.
\end{equation}
Additionally, if $\Bn\cdot(\Bf_+-\Bf_-)=0$ on $\Sigma$, then it holds that
\begin{equation}\label{estimate:S-2}
\|\nabla^2\CS\Bf\|_{L_q(\BR^N\setminus\Sigma)}\leq C\|\Bf\|_{E_q(\BR^N\setminus\Sigma)}.
\end{equation}

At this point, we introduce  a function space $\CL_q(\BR^N)$ for $q\in(1,\infty)$.
Let 
$R_1=(2-2/3)R$ and $R_2=(3+2/3)R$, and then
$$
\CL_q(\BR^N)=\{f\in L_q(\BR^N) : \spp f\subset D_{R_1,R_2}, (f,1)_{\BR^N}=0\}
$$
and $\|f\|_{\CL_q(\BR^N)}=\|f\|_{L_q(\BR^N)}$, 
where
$
D_{R_1,R_2}=\{x\in\BR^N : R_1\leq |x|\leq R_2\}.
$

Now we consider for $f\in\CL_q(\BR^N)$ the following problem in the whole space:
\begin{equation}\label{eq:4-ext}
\Delta v= f \quad \text{in $\BR^N$.}
\end{equation}
For this problem, we have
\begin{lemm}\label{lemm:3-ext}
Let $q\in(1,\infty)$. Then there exists a linear operator
$$
\wtd \CT_\infty: \CL_q(\BR^N)\to \wht H_q^1(\BR^N)\cap\wht H_q^2(\BR^N)
$$
such that, for any $f\in\CL_q(\BR^N)$, $v=\wtd\CT_\infty f$ is a solution to \eqref{eq:4-ext}.
In addition,
\begin{align*}
\|\wtd\CT_\infty f\|_{L_q(G)} + \sup_{|x|\geq 4R}|x|^{N-1}|[\wtd \CT_\infty f](x)|  &\leq C\|f\|_{\CL_q(\BR^N)}, \\
\|\nabla \wtd\CT_\infty f\|_{L_q(\BR^N)} +\sup_{|x|\geq 4R}|x|^N|\nabla [\wtd\CT_\infty f](x)|&\leq C\|f\|_{\CL_q(\BR^N)}, \\
\|\nabla^2 \wtd\CT_\infty f\|_{L_q(\BR^N)} & \leq C\|f\|_{\CL_q(\BR^N)},
\end{align*}
with some positive constant $C=C(R,N,q)$.
\end{lemm}

\begin{proof}
See e.g. \cite[pp.1703--1704]{Shibata2018} and \cite[Lemma 1]{Szufla1991}.
\end{proof}

Next, we consider the following problem in a bounded domain:
\begin{equation}\label{eq:5-ext}
\left\{\begin{aligned}
\Delta v &= f && \text{in $G\setminus\Sigma$,} \\
\jmp{\rho v}&=0 && \text{on $\Sigma$}, \\
\jmp{\Bn\cdot\nabla v}&=0 && \text{on $\Sigma$,} \\
v&=0 && \text{on $\Gamma$.}
\end{aligned}\right.
\end{equation}
By Theorem \ref{theo:3-bdd}, we have

\begin{lemm}\label{lemm:4-ext}
Suppose that Assumption $\ref{assu:q}$ holds.
Then there is a linear operator
$$
\CT_0 : L_q(G\setminus\Sigma)\to H_q^2(G\setminus\Sigma)
$$
such that, for any $f\in L_q(G\setminus \Sigma)$,
$v=\CT_0 f$ is a solution to \eqref{eq:5-ext}.
In addition, 
\begin{align*}
\|\CT_0 f\|_{H_q^2(G\setminus\Sigma)}\leq C\|f\|_{L_q(G\setminus\Sigma)}
\end{align*}
for some positive constant $C=C(N,q,r,\rho_+,\rho_-)$.
\end{lemm}

Finally, we introduce $\CT:\CL_q(\BR^N)\to \wht H_q^1(\BR^N\setminus\Sigma)\cap\wht H_q^2(\BR^N\setminus\Sigma) $ as follows:
For $f\in\CL_q(\BR^N)$, we choose a constant $c_f$ so that
\begin{equation}\label{condi:average0}
\int_{D_{R_3,R_4}}(\wtd\CT_\infty f+c_f)\intd x=0,
\end{equation}
where $D_{R_3,R_4}=\{x\in\BR^N : R_3\leq |x| \leq R_4\}$ for $R_3= (3+1/3)R$ and $R_4=(4-1/3)R$.
Let us define
\begin{equation}\label{defi:operatorT}
\CT_\infty f=\wtd \CT_\infty f + c_f, \quad \CT f=\vph_\infty\CT_\infty f+\vph_0\CT_0 f.
\end{equation}
We then have by \eqref{condi:average0}
\begin{equation}\label{condi:average}
\int_{D_{R_3,R_4}}\CT f \intd x=\int_{D_{R_3,R_4}}\CT_\infty f\intd x=0.
\end{equation}
In addition, since $c_f=-|D_{R_3,R_4}|^{-1}\int_{D_{R_3,R_4}}\wtd \CT_\infty f\intd x$,
it holds by Lemma \ref{lemm:3-ext} that
\begin{equation*}
|c_f|\leq \frac{1}{|D_{R_3,R_4}|}\int_{D_{R_3,R_4}}|\wtd\CT_\infty f|\intd x
\leq C\|\wtd\CT_\infty f\|_{L_q(G)} \leq C\|f\|_{\CL_q(\BR^N)}.
\end{equation*}
Thus $\CT f$ satisfies, together with Lemmas \ref{lemm:3-ext} and \ref{lemm:4-ext},
the following estimate:
\begin{equation}\label{est:T}
\|\nabla\CT f\|_{L_q(\BR^N\setminus\Sigma)}+\|\nabla^2\CT f\|_{L_q(\BR^N\setminus\Sigma)}
\leq C\|f\|_{\CL_q(\BR^N)}.
\end{equation}

%%%%%%%%%%%%%%%%%%%%%%%%%%%%%%%%%%%%%%%%%%%%%%%%%%
\subsection{An auxiliary problem}\label{subsec:5-2}
In this subsection, we consider for  $f\in\CL_q(\BR^N)$
the following auxiliary problem:
\begin{equation}\label{eq:1-auxiext}
\left\{\begin{aligned}
\Delta v &= f && \text{in $\BR^N\setminus\Sigma$,} \\
\jmp{\rho v} &=0 && \text{on $\Sigma$,} \\
\jmp{\Bn\cdot\nabla v}&=0 && \text{on $\Sigma$.}
\end{aligned}\right.
\end{equation}
Concerning this system, we prove

\begin{lemm}\label{lemm:1-auxiext}
Suppose that Assumption $\ref{assu:q}$ holds.
Let $q\in[2,\infty)$ additionally and $f\in\CL_q(\BR^N)$.
Then \eqref{eq:1-auxiext} admits a solution 
$v\in\wht H_q^1(\BR^N\setminus\Sigma)\cap \wht H_q^2(\BR^N\setminus\Sigma)$ satisfying 
\begin{equation}\label{eq1:1112_2019}
\|\nabla v\|_{L_q(\BR^N\setminus\Sigma)} + \|\nabla^2 v\|_{L_q(\BR^N\setminus\Sigma)} 
\leq C \|f\|_{\CL_q(\BR^N)}
\end{equation}
for some positive constant $C=C(R,N,q,r,\rho_+,\rho_-)$.
\end{lemm}

\begin{proof}
Let $\CT$ be the operator defined as \eqref{defi:operatorT}.
Then, in $\BR^N\setminus\Sigma$,
\begin{align*}
\Delta\CT f
&=\vph_\infty\Delta\CT_\infty f+2\nabla\vph_\infty\cdot\nabla\CT_\infty f+(\Delta\vph_\infty)\CT_\infty f \\
&+\vph_0\Delta\CT_0 f+2\nabla\vph_0\cdot\nabla\CT_0 f+(\Delta\vph_0)\CT_0 f.
\end{align*}
Since
\begin{equation*}
\vph_\infty\Delta\CT_\infty f = \vph_\infty f, \quad \vph_0\Delta\CT_0 f =\vph_0 f, \quad 
\vph_0+\vph_\infty=1,
\end{equation*}
it holds that
\begin{equation*}
\Delta\CT f= f + \CG f \quad \text{in $\BR^N\setminus\Sigma$,}
\end{equation*}
where
\begin{equation}\label{defi:1-extgarb}
\CG f=2\nabla\vph_\infty\cdot\nabla\CT_\infty f+(\Delta\vph_\infty)\CT_\infty f
+2\nabla\vph_0\cdot\nabla\CT_0 f+(\Delta\vph_0)\CT_0 f.
\end{equation}
Thus we have achieved
\begin{equation}\label{eq:8:20191114}
\left\{\begin{aligned}
\Delta \CT f &= f+ \CG f && \text{in $\BR^N\setminus\Sigma$,} \\
\jmp{\rho\CT f}&=0 && \text{on $\Sigma$,} \\
\jmp{\Bn\cdot\nabla\CT f}&=0 && \text{on $\Sigma$.}
\end{aligned}\right.
\end{equation}
The following lemma is proved in the appendix A below.

\begin{lemm}\label{lemm:2-auxiext}
Suppose that Assumption $\ref{assu:q}$ holds. 
\begin{enumerate}[$(1)$]
\item\label{lemm:2-auxiext-1}
Then $\CG$ is a compact operator on $\CL_q(\BR^N)$.
\item\label{lemm:2-auxiext-2}
Let $q\in[2,\infty)$ additionally. 
Then $(I+\CG)^{-1}$ exists in $\CL(\CL_q(\BR^N))$. 
\end{enumerate}
\end{lemm}

Setting $v=\CT(I+\CG)^{-1}f$ for $f\in\CL_q(\BR^N)$,
we see that $v$ is a solution to \eqref{eq:1-auxiext} by \eqref{eq:8:20191114} 
and satisfies \eqref{eq1:1112_2019} by \eqref{est:T}.
This completes the proof of Lemma \ref{lemm:1-auxiext}
\end{proof}

%%%%%%%%%%%%%%%%%%%%%%%%%%%%%%%%%%%%%%%%%%%%%%%%%%
\subsection{Proof of Theorem \ref{theo:final}}
Let $\CS$ be the operator defined as \eqref{defi:operatorS}.
Then we observe that in $\BR^N\setminus\Sigma$
\begin{align*}
\Delta \CS\Bf&= \psi_\infty\Delta\CS_\infty(\vph_\infty\Bf)+2\nabla\psi_\infty\cdot\nabla\CS_\infty(\vph_\infty\Bf)
+(\Delta\psi_\infty)\CS_\infty(\vph_\infty\Bf) \\
&+\psi_0\Delta\CS_0(\vph_0\Bf)+2\nabla\psi_0\cdot\nabla\CS_0(\vph_0\Bf)+(\Delta\psi_0)\CS_0(\vph_0\Bf).
\end{align*}
Note that in $\BR^N\setminus\Sigma$ 
\begin{equation*}
\psi_\infty\Delta\CS_\infty(\vph_\infty\Bf)+\psi_0\Delta\CS_0(\vph_0\Bf) 
=\psi_\infty\dv(\vph_\infty\Bf)+\psi_0\dv(\vph_0\Bf)=\dv\Bf.
\end{equation*}
This relation implies
$$
\Delta\CS\Bf=\dv\Bf+\CR\Bf \quad \text{in $\BR^N\setminus\Sigma$},
$$
where 
\begin{align}\label{defi:Rf}
\CR\Bf&=2\nabla\psi_\infty\cdot\nabla\CS_\infty(\vph_\infty\Bf) +(\Delta\psi_\infty)\CS_\infty(\vph_\infty\Bf) \\
&+2\nabla\psi_0\cdot\nabla\CS_0(\vph_0\Bf)+(\Delta\psi_0)\CS_0(\vph_0\Bf). \notag
\end{align}
On the other hand,
\begin{equation*}
\jmp{\rho\CS\Bf}=0, \quad \jmp{\Bn\cdot\nabla\CS\Bf}=\jmp{\Bn\cdot\Bf} \quad \text{on $\Sigma$,}
\end{equation*}
and also one has\footnote{The first property of \eqref{est1:R} is proved in the appendix B below,
while the second property of \eqref{est1:R} follows from Lemmas \ref{lemm:1-ext} and \ref{lemm:2-ext} immediately.}
for $\Bf\in F_q(\BR^N\setminus\Sigma)$
\begin{equation}\label{est1:R}
\CR\Bf\in\CL_q(\BR^N), \quad \|\CR\Bf\|_{\CL_q(\BR^N)}\leq C\|\Bf\|_{L_q(\BR^N\setminus\Sigma)}.
\end{equation}

Next, we consider
\begin{equation}\label{eq:6-ext}
\left\{\begin{aligned}
\Delta w &= -\CR\Bf && \text{in $\BR^N\setminus\Sigma$,} \\
\jmp{\rho w}&=0 && \text{on $\Sigma$}  \\
\jmp{\Bn\cdot\nabla w}&=0 && \text{on $\Sigma$.}
\end{aligned}\right.
\end{equation}
In the following, the discussion of $w$ is divided into two cases.

{\bf Case 1}: $q\in[2,\infty)$.
By Lemma \ref{lemm:1-auxiext} and \eqref{est1:R},
the system \eqref{eq:6-ext} admits a solution $w\in \wht H_q^1(\BR^N\setminus\Sigma)\cap\wht H_q^2(\BR^N\setminus\Sigma)$
satisfying 
\begin{equation}\label{est:correction1}
\|\nabla w\|_{L_q(\BR^N\setminus\Sigma)}+\|\nabla^2 w\|_{L_q(\BR^N\setminus\Sigma)}\leq C\|\Bf\|_{L_q(\BR^N\setminus\Sigma)}.
\end{equation}
Thus $v=\CS\Bf+w$ solves \eqref{final-existence} and satisfies \eqref{theo:final-est1}-\eqref{theo:final-est2}
by \eqref{estimate:S}, \eqref{estimate:S-2}, and \eqref{est:correction1}.
This completes the proof of Theorem \ref{theo:final} for $q\in[2,\infty)$.

{\bf Case 2}: $q\in(1,2)$.
Since Theorem \ref{theo:final} is already proved for $q\in[2,\infty)$ in Case 1 as above,
one can prove by the result of Case 1 the following lemma.

\begin{lemm}\label{lemm:6-ext}
Suppose that Assumption $\ref{assu:q}$ holds.
Let $q\in (1,2)$ additionally and $f\in\CL_q(\BR^N)$. 
Then \eqref{eq:1-auxiext} admits a solution $v\in \wht H_q^1(\BR^N\setminus\Sigma)
\cap \wht H_q^2(\BR^N\setminus\Sigma)$ satisfying
\begin{equation*}
\|\nabla v\|_{L_q(\BR^N\setminus\Sigma)} + \|\nabla^2 v\|_{L_q(\BR^N\setminus\Sigma)} 
\leq C\|f\|_{L_q(\BR^N)}
\end{equation*}
for some positive constant $C=C(R,N,q,r,\rho_+,\rho_-)$.
\end{lemm}

\begin{proof}
See the appendix C below.
\end{proof}

By Lemma \ref{lemm:6-ext} and \eqref{est1:R},
the system \eqref{eq:6-ext} admits a solution $w\in \wht H_q^1(\BR^N\setminus\Sigma)
\cap \wht H_q^2(\BR^N\setminus\Sigma)$ satisfying \eqref{est:correction1}.
Thus $v=\CS\Bf+ w$ solves \eqref{final-existence} and satisfies 
\eqref{theo:final-est1}-\eqref{theo:final-est2}
by \eqref{estimate:S}, \eqref{estimate:S-2}, and \eqref{est:correction1}.
This completes the proof of Theorem \ref{theo:final} for $q\in(1,2)$,
which furnishes the proof of Theorem \ref{theo:final}.

%%%%%%%%%%%%%%%%%%%%%%%%%%%%%%%%%%%%%%%%%%%%%%%%%%%%%%%%%%%%%%%%%%%%%%
\def\thesection{A}
\renewcommand{\theequation}{A.\arabic{equation}}
\section{}
This appendix proves Lemma \ref{lemm:2-auxiext} for $\CG$ defined as \eqref{defi:1-extgarb}.
We start with

\begin{lemm}\label{lemm:A1}
For $q$ satisfying Assumption $\ref{assu:q}$,
$\CG$ is a compact operator on $\CL_q(\BR^N)$.
\end{lemm}

\begin{proof}
{\bf Step 1}: $\CG f\in\CL_q(\BR^N)$ for any $f\in\CL_q(\BR^N)$.
It is clear that
\begin{equation}\label{eq:5_20191115}
\spp\CG f \subset D_{R_1,R_2}.
\end{equation}
In addition, by $\nabla\vph_\infty =-\nabla\vph_0$ and $\Delta\vph_\infty=-\Delta\vph_0$,
\begin{align}\label{eq:7_20191115}
\CG f
&=-2\nabla\vph_0\cdot\nabla\CT_\infty f-(\Delta\vph_0)\CT_\infty f
+2\nabla\vph_0\cdot\nabla\CT_0 f+(\Delta\vph_0)\CT_0 f
\\
&=\vph_0\Delta\CT_\infty f
-\dv((\nabla\vph_0)\CT_\infty f+\vph_0\nabla\CT_\infty f) \notag\\
&-\vph_0\Delta\CT_0 f+\dv ((\nabla\vph_0)\CT_0 f+\vph_0\nabla\CT_0 f) \quad \text{in $\BR^N\setminus\Sigma$},
\notag
\end{align}
which, combined with $\vph_0\Delta\CT_\infty f=\vph_0 f$ and $\vph_0\Delta\CT_0 f=\vph_0 f$ in $\BR^N\setminus\Sigma$,
furnishes
\begin{align*}
\CG f&=-\dv((\nabla\vph_0)\CT_\infty f+\vph_0\nabla\CT_\infty f) \\
&+\dv((\nabla\vph_0)\CT_0 f+\vph_0\nabla\CT_0 f) \quad \text{in $\BR^N\setminus\Sigma$}.
\end{align*}
It then holds that 
\begin{align*}
(\CG f,1)_{B_{4R}\setminus\Sigma}
&=(\dv((\nabla\vph_0)\CT_\infty f+\vph_0\nabla\CT_\infty f),1)_{B_{4R}\setminus\Sigma} \\
&+(\dv((\nabla\vph_0)\CT_0 f),1)_{B_{4R}\setminus\Sigma}
+(\dv(\vph_0\nabla\CT_0 f),1)_{B_{4R}\setminus\Sigma}  \\
&=:I_1+I_2+I_3.
\end{align*}
By Gauss's divergence theorem, we see that
\begin{align*}
I_1=I_2=0, \quad
I_3=\int_{\Sigma} \jmp{\Bn \cdot\nabla \CT_0 f} \intd \sigma  =0,
\end{align*}
where $d\sigma$ is the surface element of $\Sigma$.
Hence $(\CG f,1)_{\BR^N}=(\CG f,1)_{B_{4R}\setminus\Sigma}=0$,
which, combined with \eqref{eq:5_20191115}, implies $\CG f\in \CL_q(\BR^N)$.

{\bf Step 2:} $\CG$ is a compact operator on $\CL_q(\BR^N)$.

Let $\{f^{(j)}\}_{j=1}^\infty$ be a bounded sequence in $\CL_q(\BR^N)$, i.e.
there exists a positive constant $M$ such that $\|f^{(j)}\|_{\CL_q(\BR^N)}\leq M$ for any $j\in\BN$.
Note that by \eqref{defi:1-extgarb}, Lemmas \ref{lemm:3-ext} and \ref{lemm:4-ext},
and $\spp\nabla\vph_0,\spp\nabla\vph_\infty\subset\{x\in\BR^N:2R\leq |x|\leq 3R\}$
\begin{equation}\label{eq:2-A}
\CG\in\CL(\CL_q(\BR^N),H_q^1(\BR^N)).
\end{equation}
Since $H_q^1(\BR^N)$ is compactly embedded into $L_q(D_{R_1,R_2})$
by the Rellich-Kondrachov theorem (cf. \cite[Theorem 6.3]{AF03}),
$\CG$ can be regarded as a compact operator from $\CL_q(\BR^N)$ into $L_q(D_{R_1,R_2})$ by \eqref{eq:2-A}.
Thus there exists 
$G_f\in L_q(D_{R_1,R_2})$ such that
\begin{equation}\label{eq:4-A}
\lim_{j\to\infty}\|\CG f^{(j)}-G_f\|_{L_q(D_{R_1,R_2})}=0,
\end{equation}
up to some extraction.
Let us define
$$
\wtd G_f=\left\{\begin{aligned}
&G_f && \text{in $D_{R_1,R_2}$,} \\
&0 && \text{in $\BR^N\setminus D_{R_1,R_2}$,} 
\end{aligned}\right.
$$
and then $\spp\wtd G_f\subset D_{R_1,R_2}$. In addition, by \eqref{eq:4-A},
\begin{equation*}
(\wtd G_f, 1)_{\BR^N}=(G_f,1)_{D_{R_1,R_2}}=\lim_{j\to\infty}(\CG f^{(j)} ,1)_{D_{R_1,R_2}} 
\end{equation*}
which, combined with $\CG f^{(j)}\in\CL_q(\BR^N)$ as was proved in Step 1, implies $(\wtd G_f, 1)_{\BR^N}=0$.
Hence $\wtd G_f\in \CL_q(\BR^N)$.
On the other hand, by \eqref{eq:4-A} and $\CG f^{(j)}\in\CL_q(\BR^N)$,
\begin{equation*}
\lim_{j\to\infty}\|\CG f^{(j)}-\wtd G_f\|_{L_q(\BR^N)}=\lim_{j\to\infty}\|\CG f^{(j)} - G_f\|_{L_q(D_{R_1,R_2})}=0,
\end{equation*}
and therefore $\CG$ is a compact operator on $\CL_q(\BR^N)$.
This completes the proof. 
\end{proof}

Now we prove Lemma \ref{lemm:2-auxiext} \eqref{lemm:2-auxiext-2}.
In view of Lemma \ref{lemm:A1} and the Riesz-Schauder theory,
it suffices to prove that the kernel of $I+\CG$ is trivial in what follows. 
Let us begin with some property of $\CT$ defined in \eqref{defi:operatorT}.

\begin{lemm}\label{lemm:A2-2}
Suppose that Assumption $\ref{assu:q}$ holds, 
and let $f\in\CL_q(\BR^N)$ with $\CT f=0$.
Then $f=0$.  
\end{lemm}

\begin{proof}
By the assumption $\CT f=0$, we have 
\begin{equation}\label{eq:4_20191114}
\CT_\infty f=0 \quad \text{when $|x|\geq 3R$,} \quad 
\CT_0 f=0 \quad \text{when $|x|\leq 2R$.}
\end{equation}
Here we set
\begin{equation*}
w=
\left\{\begin{aligned}
&\CT_0 f && (x\in B_{4R}\setminus B_{(3/2)R}), \\
&0 && (x\in B_{(3/2)R}).
\end{aligned}\right.
\end{equation*}
Then $w$ satisfies by the second property of \eqref{eq:4_20191114} and $\CT_0f=0$ on $\Gamma$
\begin{equation}\label{eq:5-A}
\Delta w= f  \quad \text{in $G$,} \quad w=0 \quad \text{on $\Gamma$.}
\end{equation}
On the other hand, $\CT_\infty f$ also satisfies \eqref{eq:5-A} by the first property of \eqref{eq:4_20191114}.
The uniqueness of solutions of \eqref{eq:5-A} thus implies $\CT_0f=\CT_\infty f$ in $G$,
which, combined with the assumption $\CT f=0$ and $\vph_\infty=1-\vph_0$, furnishes
\begin{equation*}
0=\CT f=(1-\vph_0)\CT_\infty f+\vph_0\CT_0 f = \CT_\infty f \quad \text{in $G$.}
\end{equation*}
Combining this property with \eqref{eq:4_20191114} yields $\CT_\infty f=0$ in $\BR^N$. Since
$$
\Delta \CT_\infty f=f \quad \text{in $\BR^N$,}
$$
one concludes $f=0$. This completes the proof of the lemma.
\end{proof}

Next we complete the proof of Lemma \ref{lemm:2-auxiext} \eqref{lemm:2-auxiext-2}
by the following result.

\begin{lemm}\label{lemm:A2}
Suppose that Assumption $\ref{assu:q}$ holds.
Let $q\in[2,\infty)$ additionally and $f\in \CL_q(\BR^N)$ with $(I+\CG)f=0$.
Then $f=0$.
\end{lemm}

\begin{proof}
Let $\omega(x)$ be an element of $C_0^\infty(\BR^N)$ satisfying
\begin{equation*}
\omega(x)=
\left\{\begin{aligned}
&1 && (|x|\leq 1), \\
&0 && (|x|\geq 2),
\end{aligned}\right.
\end{equation*}
and set $\omega_L(x)=\omega(x/L)$ for $L>0$ large enough.
Note that $\CT f\in H_{q',\lc}^2(\BR^N\setminus\Sigma)$ for $q'=q/(q-1)$ by $q\in[2,\infty)$.
Then, by the assumption $(I+\CG)f=0$, \eqref{eq:8:20191114}, and integration by parts,
\begin{align}
0
&=((I+\CG)f,\omega_L\rho\CT f)_{\BR^N\setminus\Sigma} 
=(\Delta\CT f,\omega_L\rho\CT f)_{\BR^N\setminus\Sigma} \nonumber \\ 
&=-(\nabla \CT f,\omega_L\rho\nabla \CT f)_{\BR^N\setminus\Sigma}
-(\nabla \CT f,(\nabla\omega_L) \rho\CT f)_{\BR^N}. \label{eq:1_20191130}
\end{align}
Since $L>0$ is large enough and $\spp (\nabla \omega)(\cdot\,/L)\subset D_{L,2L}=\{x\in\BR^N: L\leq |x|\leq 2L\}$,
one sees that on $\spp (\nabla\omega)(\cdot\,/L)$
\begin{equation*}
\CT f =\wtd \CT_\infty f+c_f, \quad 
\nabla\CT f =\nabla\wtd \CT_\infty f, \quad \rho=\rho_-.
\end{equation*}
We thus observe that
\begin{align*}
&|(\nabla \CT f,(\nabla\omega_L) \rho\CT f)_{\BR^N}|
=\frac{\rho_-}{L}\left|\int_{\BR^N}\nabla \wtd\CT_\infty f\cdot(\nabla\omega)\left(\frac{x}{L}\right)(\wtd\CT_\infty f+c_f) \intd x\right| \\
&\leq \frac{\rho_-}{L}
\int_{D_{L,2L}}|\nabla \wtd\CT_\infty f|\left|(\nabla\omega)\left(\frac{x}{L}\right)\right||\wtd\CT_\infty f|\intd x \\
&+\frac{\rho_-|c_f|}{L}\int_{D_{L,2L}}|\nabla \wtd\CT_\infty f|\left|(\nabla\omega)\left(\frac{x}{L}\right)\right| \intd x \\
&=:I_1+I_2.
\end{align*}
By Lemma \ref{lemm:3-ext},
\begin{align*}
I_1&\leq \frac{C}{L}\left(\sup_{x\in\BR^N}|\nabla\omega(x)|\right)\|f\|_{\CL_q(\BR^N)}^2\int_{D_{L,2L}}|x|^{-(2N-1)} \intd x, \\
I_2&\leq  \frac{C}{L}|c_f|\left(\sup_{x\in\BR^N}|\nabla\omega(x)|\right)
\|f\|_{\CL_q(\BR^N)} \int_{D_{L,2L}}|x|^{-N} \intd  x.
\end{align*}
Hence $\lim_{L\to\infty}(\nabla \CT f,(\nabla\omega_L)\rho \CT f)_{\BR^N}=0$.

Now we take the limit: $L\to\infty$ in \eqref{eq:1_20191130}, i.e. 
\begin{equation*}
\lim_{L\to\infty}\int_{\BR^N\setminus\Sigma}\rho(x)\left|[\nabla\CT f](x)\right|^2 \omega\left(\frac{x}{L}\right) \intd x=0.
\end{equation*}
Then the monotone convergence theorem yields
$$
\int_{\BR^N\setminus\Sigma}
\rho(x)\left|[\nabla\CT f](x)\right|^2=0.
$$
Thus $\nabla\CT f=0$ in $\BR^N\setminus\Sigma$, which implies that
there are constants $c_\pm$ so that $\rho\CT f=c_\pm$ in $\Omega_\pm$. 
Since $\jmp{\rho\CT f}=0$ on $\Sigma$, one has $c_+=c_-$. 
On the other hand, it holds by \eqref{condi:average} that
\begin{equation*}
c_+|D_{R_3,R_4}|=\int_{D_{R_3,R_4}}\rho\CT f\intd x=\rho_-\int_{D_{R_3,R_4}}\CT f\intd x=0,
\end{equation*}
which implies $c_+=c_-=0$. Hence $\CT f=0$,
which, combined with \ref{lemm:A2-2}, furnishes $f=0$. This completes the proof of the lemma.
\end{proof}

%%%%%%%%%%%%%%%%%%%%%%%%%%%%%%%%%%%%%%%%%%%%%%%%%%%%%%%%%%%%%%%%%%%%%%
\def\thesection{B}
\renewcommand{\theequation}{B.\arabic{equation}}
\section{}

In this appendix, we prove for $\CR\Bf$ defined as \eqref{defi:Rf}

\begin{lemm}\label{lemm:apB-1}
Suppose that Assumption $\ref{assu:q}$ holds. Then $\CR\Bf\in\CL_q(\BR^N)$ for any $\Bf\in F_q(\BR^N\setminus\Sigma)$.
\end{lemm}

\begin{proof}
Since $\spp \CR\Bf\subset D_{R_1,R_2}$, 
it suffices to verify that $(\CR\Bf,1)_{\BR^N}=0$ for $\Bf\in F_q(\BR^N\setminus\Sigma)$ in what follows.
Let $\psi=1-\psi_\infty$. 
Similarly to \eqref{eq:7_20191115}, we write
\begin{align*}
\CR\Bf
&=\psi\Delta\CS_\infty(\vph_\infty\Bf)-\dv((\nabla\psi)\CS_\infty(\vph_\infty\Bf)
+\psi\nabla \CS_\infty(\vph_\infty\Bf)) \\
&-\psi_0\Delta\CS_0(\vph_0\Bf)+\dv((\nabla\psi_0)\CS_0(\vph_0\Bf)+\psi_0\nabla\CS_0(\vph_0\Bf))
\quad \text{in $\BR^N\setminus\Sigma$},
\end{align*}
which, combined with the facts:
\begin{equation*}
\psi\Delta\CS_\infty(\vph_\infty\Bf)=\psi\dv(\vph_\infty\Bf)=0, \quad
\psi_0\Delta\CS_0(\vph_0\Bf)
=\psi_0\dv(\vph_0\Bf) =\dv(\vph_0\Bf), 
\end{equation*}
furnishes that
\begin{align*}
\CR\Bf&=-\dv(\vph_0\Bf)-\dv((\nabla\psi)\CS_\infty(\vph_\infty\Bf)+\psi\nabla \CS_\infty(\vph_\infty\Bf)) \\
&+\dv((\nabla\psi_0)\CS_0(\vph_0\Bf)+\psi_0\nabla\CS_0(\vph_0\Bf)) \quad \text{in $\BR^N\setminus\Sigma$.}
\end{align*}
It then holds that
\begin{align*}
(\CR\Bf,1)_{B_{4R}\setminus\Sigma}&=
-(\dv(\vph_0\Bf),1)_{B_{4R}\setminus\Sigma} \\
&-(\dv((\nabla\psi)\CS_\infty(\vph_\infty\Bf)+\psi\nabla \CS_\infty(\vph_\infty\Bf)),1)_{B_{4R}\setminus\Sigma} \\
&+(\dv((\nabla\psi_0)\CS_0(\vph_0\Bf)),1)_{B_{4R}\setminus\Sigma}+(\dv(\psi_0\nabla\CS_0(\vph_0\Bf)),1)_{B_{4R}\setminus\Sigma} \\
&=:I_1+I_2+I_3+I_4.
\end{align*}
By Gauss's divergence theorem, we see that 
\begin{gather*}
I_1=-\int_\Sigma \jmp{\Bn\cdot\Bf}\intd\sigma, \quad 
I_2=I_3=0, \\
I_4=\int_\Sigma\psi_0\jmp{\Bn\cdot\nabla\CS_0(\vph_0\Bf)}\intd\sigma
=\int_\Sigma \vph_0 \jmp{\Bn\cdot \Bf}\intd\sigma
=\int_\Sigma \jmp{\Bn\cdot\Bf}\intd\sigma,
\end{gather*}
where $d\sigma$ is the surface element of $\Sigma$.
Hence $(\CR\Bf,1)_{\BR^N}=(\CR\Bf,1)_{B_{4R}\setminus\Sigma}=0$, 
which completes the proof. 
\end{proof}

%%%%%%%%%%%%%%%%%%%%%%%%%%%%%%%%%%%%%%%%%%%%%%%%%%%%%%%%%%%%%%%%%%%%%%
\def\thesection{C}
\renewcommand{\theequation}{C.\arabic{equation}}
\section{}
In this appendix, we  prove Lemma \ref{lemm:6-ext}.
Similarly to Subsection \ref{subsec:5-2} and the appendix A,
it suffices to prove the following lemma.

\begin{lemm}\label{lemm:B1}
Suppose that Assumption $\ref{assu:q}$ holds.
Let $q\in(1,2)$ additionally and $f\in\CL_q(\BR^N)$ with $(I+\CG)f=0$. Then $f=0$.
\end{lemm}

\begin{proof}
Since $(I+\CG)f=0$, we have by \eqref{eq:8:20191114}
\begin{equation*}
\left\{\begin{aligned}
\Delta \CT f&=0 && \text{in $\BR^N\setminus \Sigma$,} \\
\jmp{\rho\CT f}&=0 && \text{on $\Sigma$,} \\
\jmp{\Bn\cdot\nabla\CT f}&=0 && \text{on $\Sigma$.}
\end{aligned}\right.
\end{equation*}
On the other hand,
one sees by Theorem \ref{theo:final}
that for any $\Phi\in C_0^\infty(\BR^N\setminus\Sigma)\subset F_{q'}(\BR^N\setminus\Sigma)$, $q'=q/(q-1)\in(2,\infty)$,
there exists $w\in \wht H_{q'}^1(\BR^N\setminus\Sigma)\cap \wht H_{q'}^2(\BR^N\setminus\Sigma)$ such that
\begin{equation*}
\left\{\begin{aligned}
\Delta w &= \dv\Phi && \text{in $\BR^N\setminus\Sigma$,} \\
\jmp{\rho w}&=0 && \text{on $\Sigma$,} \\
\jmp{\Bn\cdot\nabla w}&=0 && \text{on $\Sigma$.}  
\end{aligned}\right.
\end{equation*}
Similarly to the proof of uniqueness in Subsection \ref{subsec:3-3}, 
it holds that $\rho\CT f= c$ for some constant $c$.
One then concludes $c=0$ by \eqref{condi:average}, i.e. $\CT f=0$,
which, combined with Lemma \ref{lemm:A2-2} furnishes $f=0$.
This completes the proof of Lemma \ref{lemm:B1}.
\end{proof}

%%%%%%%%%%%%%%%%%%%%%%%%%%%%%%%%%%%%%%%%%%%%%%%%%%%%%%%%%%%%%%%%%%%%%%
%\bibliographystyle{plain}
%\bibliography{bib_weak}

\end{document}